\documentclass[11pt, a4paper]{article}
\usepackage{t1enc}
\usepackage[latin1]{inputenc}

\usepackage[english]{babel}
\usepackage{amsmath,amsthm}
\usepackage{mathrsfs}
\usepackage{amsfonts}
\usepackage{latexsym}
\usepackage{graphicx}
\usepackage{tikz}
\usetikzlibrary{shapes.geometric}  
\usetikzlibrary{positioning}

\usepackage{float}
\usepackage[natural]{xcolor}
\usepackage{algorithm}
\usepackage{algorithmic}
\usepackage{enumerate,enumitem}
\usepackage{multirow}
\usepackage{xcolor}
\usepackage[colorlinks,linkcolor=blue]{hyperref}
\usepackage{lineno}
\usepackage{setspace}
\usepackage{fullpage}
\usepackage[title]{appendix}
\usepackage{todonotes}
\usepackage{amssymb}
\usepackage{array}

\usepackage{listings}
\usepackage{bbm}
\usepackage{caption}
\usepackage{subfigure}
\usepackage{url}
\usepackage{fullpage}
\usepackage{hyperref}
\usepackage[margin=2.3cm]{geometry}
\usepackage{enumitem,verbatim}
\usepackage{algorithm}
\usepackage{algorithmic}

\newtheorem{theorem}{Theorem}[section]
\newtheorem{claim}[theorem]{Claim}

\newtheorem{lemma}[theorem]{Lemma}
\newtheorem{conj}[theorem]{Conjecture}
\newtheorem{corollary}[theorem]{Corollary}
\newtheorem{proposition}[theorem]{Proposition}

\newtheorem{fact}[theorem]{Fact}

\theoremstyle{definition}

\newtheorem{definition}[theorem]{Definition}
\newcommand{\eps}{\varepsilon}

\usepackage{todonotes}
\title{A step toward Chen-Lih-Wu conjecture}

\author{
Yangyang Cheng\thanks{{Faculty of Computer Science and Mathematics, University of Passau, Passau, Germany. Email: {\tt yangyang.cheng@uni-passau.de}}. Partially supported by the Deutsche Forschungsgemeinschaft (DFG, German Research Foundation)-542321564.}
\and 
Zhenyu Li \thanks{School of Mathematics, Shandong University, Jinan, 250100, China. Email: {\tt zy.li@mail.sdu.edu.cn}.}
\and
Wanting Sun \thanks{Data Science Institute, Shandong University, Jinan, 250100, China. Email: {\tt wtsun@sdu.edu.cn}. Supported by the Natural Science Foundation of China (12501488) and the Natural Science Foundation of Shandong Province (ZR2024QA023).}
\and
Guanghui Wang \thanks{School of Mathematics, Shandong University, Jinan, 250100, China. Email: {\tt ghwang@sdu.edu.cn}. Supported by the National Key Research and Development Program (2023YFA1009603) and the Natural Science Foundation of China (12231018).
}
}
\date{}

\begin{document}
\maketitle
\begin{abstract}
 An equitable $k$-coloring of a graph is a proper $k$-coloring where the sizes of any two different color classes differ by at most one. In 1973, Meyer conjectured that every connected graph $G$  has an equitable {$k$-coloring for some $k\leq \Delta(G)$}, unless $G$ is a complete graph or an odd cycle. Chen, Lih, and Wu strengthened this in 1994 by conjecturing that 
 for $k\geq 3$, the only connected graphs of maximum degree at most $k$ with no equitable $k$-coloring 
 are the complete bipartite graph $K_{k,k}$ for odd $k$ and the complete graph $K_{k+1}$. A more refined conjecture was proposed by Kierstead and Kostochka, relaxing the maximum degree condition to an Ore-type condition. 
 Their conjecture states the following: for $k\geq 3$, if $G$ is an $n$-vertex graph such that $d(x) + d(y)\leq 2k$ for every edge $xy\in E(G)$, and $G$ admits no equitable $k$-coloring, then $G$ contains either $K_{k+1}$ or $K_{m,2k-m}$ for some odd $m$. We prove that for any constant $c>0$ and all sufficiently large $n$, the latter two  conjectures hold for every $k\geq cn$. Our proof yields an algorithm with polynomial time 
that decides whether $G$ has an equitable $k$-coloring, thereby answering a conjecture of Kierstead, Kostochka, Mydlarz, and Szemer\'{e}di when $k \ge cn$. 
 
\end{abstract}
\maketitle

\section{Introduction}
A \textit{proper  $k$-coloring} of a graph is defined as a vertex coloring from a set of $k$
 colors such that no two adjacent vertices share a common color.  The \textit{chromatic number} of a graph  $G$, denoted by $\chi(G)$, is the minimum number of colors required in a proper coloring.  An easy observation deduces that $\chi(G)\leq \Delta(G)+1$, where $\Delta(G)$ is the maximum degree of $G$. The classical Brooks' theorem states that the chromatic number of a connected graph is at most its maximum degree, unless the graph is a complete graph or an odd cycle. 


An \textit{equitable $k$-coloring} of a graph $G$ is a proper $k$-coloring where the sizes of any two different color classes differ by at most one.  A classical problem in graph coloring theory is to ask whether a graph $G$ has an equitable $k$-coloring if $G$ has some degree constraints with respect to $k$. In 1968, Gr\"{u}nbaum \cite{Gr1968} conjectured that every graph with maximum degree less than $k$ has an equitable $k$-coloring. Remarkably, via graph complements, this is equivalent to a conjecture of Erd\H{o}s from 1964 \cite{E1964}: every graph on $n = rk$ vertices with minimum degree at least $(1 - \frac{1}{r})n$ admits a $K_r$-factor. 
In 1970, Hajnal and Szemer\'{e}di confirmed the conjecture \cite{HS1970}, establishing what is now known as the Hajnal-Szemer\'{e}di Theorem. Kierstead and Kostochka \cite{KK-algor,KKMS2008} found a polynomial time algorithm for such colorings.


Observe that the degree condition in Gr\"{u}nbaum's conjecture is tight, as $G$ admits no equitable $k$-coloring when it contains $K_{k+1}$ as a subgraph. Perhaps inspired by this and the Brooks' theorem,  Meyer \cite{M1973} further conjectured that every connected graph $G$  has an equitable {$k$-coloring for some integer $k\leq \Delta(G)$}, unless $G$ is a complete graph or an odd cycle. Lih and Wu \cite{LW1991} proved that Meyer's conjecture is true for connected bipartite graphs. In fact, they showed a stronger result that a connected bipartite graph $G$ has an equitable $\Delta(G)$-coloring if $G$ is not the complete bipartite graph $K_{2m+1,2m+1}$ for some $m\geq 0$. This line of research led Chen, Lih, and Wu \cite{CLW1994} to propose the following conjecture, which is considered one of the most significant open problems in equitable coloring according to~\cite{kierstead-survey}.

\begin{conj}[Chen-Lih-Wu Conjecture \cite{CLW1994}] \label{CLW}
    Let $G$ be a connected
graph with $\Delta(G)\leq k$. Then $G$ has no equitable $k$-coloring if and only if  $G=K_{k+1}$, or  $k=2$ and $G$ is an odd cycle,
or $k$ is odd and $G=K_{k,k}$.
\end{conj}
Combined with the Hajnal-Szemer\'{e}di Theorem, it suffices to consider the case $k=\Delta(G)$ in Chen-Lih-Wu Conjecture. 
This conjecture has been verified for several special cases: $k\leq 3$ or $k\geq \frac{|V(G)|}{2}$ \cite{CLW1994}; $\frac{|V(G)|+1}{3}\leq k<\frac{|V(G)|}{2}$ \cite{chenhuang}; $k\leq 4$ or $k\geq \frac{|V(G)|}{4}$ \cite{KK2012,KK2015}, and 
some sparse graphs \cite{cly,KLX2024,LZ2024,Na2004,Na2012,YZ1997,Zh2016,YZ1998}. Moreover, Kierstead and Kostochka \cite{KK-algo} proved that there exists a polynomial time algorithm for deciding whether a graph $G$
 with $\Delta(G)\leq k$ admits an equitable  $k$-coloring.


One may view the Hajnal-Szemer\'{e}di Theorem as a Dirac-type result, as it guarantees the existence of a $K_r$-factor in a graph under  the minimum degree condition. 
About two decades ago, Kostochka and Yu \cite{KY2007} conjectured that every graph $G$ in which $d(x)+d(y)\le 2k$ for every edge $xy\in E(G)$ has an equitable $(k+1)$-coloring. In 2008, Kierstead and Kostochka \cite{KK2008} proved this conjecture with a weaker degree condition. 

\begin{theorem}[\cite{KK2008}]\label{KK2008}
    Every graph satisfying $d(x)+d(y)\le 2k+1$ for every edge $xy$, has an equitable $(k+1)$-coloring.
\end{theorem}
Notice that the degree sum bound in Theorem \ref{KK2008} is tight, as the complete graph $K_{k+2}$ admits no proper $(k+1)$-coloring. 
 Moreover,  for each odd $m\le k+1$, the graph $K_{m,2k-m+2}$ satisfies $d(x)+d(y)\le 2k+2$ for every edge $xy$, yet has no equitable $(k+1)$-coloring.  
 In 2008, Kierstead and Kostochka \cite{KK2008}  
 conjectured that the following Ore-type analogue of the Chen-Lih-Wu Conjecture holds. 
\begin{conj}
[\cite{KK2008}]\label{cKK2008}
Let $k\ge 3$. If $G$ is a graph  satisfying $d(x)+d(y)\le 2k$ for every edge $xy$, and $G$ has no equitable $k$-coloring, then $G$ contains either $K_{k+1}$ or $K_{m,2k-m}$ for some odd $m$.
\end{conj}
A conjecture concerning the algorithmic version of the Ore-type condition was proposed by Kierstead, Kostochka, Mydlarz, and Szemer\'{e}di \cite{KKMS2008}. 
\begin{conj}[\cite{KKMS2008}]\label{cKKMS2008}
    Let $G$ be a graph with $d(x)+d(y)< 2k$ for all $xy\in E(G)$. Then there exists a polynomial-time algorithm to determine an  equitable $k$-coloring for $G$.
\end{conj}


Our work answers Conjecture \ref{cKK2008} and Conjecture \ref{cKKMS2008} when $|G|$ is sufficiently large and $k$ is linear with respect to $|G|$.

\begin{theorem}\label{mainore}
    For a positive constant $c$, there exists an integer $n_0:=n_0(c)$ such that the following holds for every $n\geq n_0$. Let $k$ be an integer such that $k\geq cn$, and let $G$ be an $n$-vertex graph with  $d(x)+d(y)\le 2k$ for every edge $xy\in E(G)$. 
    Then either $G$ has an equitable $k$-coloring, or $G$ contains $K_{k+1}$ or $K_{m,2k-m}$ for some odd integer $m\in [k]$. Moreover, there
is an algorithm with polynomial time 
that decides whether $G$ has an equitable $k$-coloring. 
\end{theorem}

By replacing the Ore-type condition in Theorem~\ref{mainore} with the maximum degree constraint, we immediately obtain the following corollary. This resolves Conjecture~\ref{CLW} for all sufficiently large 
$n$ when the maximum degree is linear in $n$.

\begin{corollary}\label{Dirac}
   For a positive constant $c$, there exists an integer $n_0:=n_0(c)$ such that the following holds for every $n\geq n_0$. Let $k$ be an integer such that $k \geq cn$ and $G$ be an $n$-vertex graph with $\Delta(G)\le k$ that admits no equitable $k$-coloring. Then $G$ contains either $K_{k+1}$ or $K_{k,k}$ (with $k$ odd). If, in addition, $G$ is connected, then $G$ is either $K_{k+1}$ or $K_{k,k}$ (with $k$ odd).
\end{corollary}


\subsection{Reformulation of Theorem \ref{mainore}}

We introduce a new approach to prove Theorem \ref{mainore}, which differs from standard coloring techniques by treating it as a problem in extremal graph theory. Without loss of generality, we may assume that $k$ divides $n$. Indeed, if $k\nmid n$, we consider the graph $G'$ obtained by adding a clique on $q$ vertices to $G$, where $q$ satisfies $k\mid (n+q)$ and $1\leq q\leq k-1$. The conclusion for $G'$ immediately implies the result for $G$ since the addition of a complete graph preserves both the degree condition and the nonexistence of equitable colorings. 

Let $\overline{G}$ denote the complement graph of $G$. Then the following are equivalent:
\begin{itemize}
\item $G$ admits no equitable $k$-coloring;
\item $\overline{G}$ contains no $K_{\frac{n}{k}}$-factor.
\end{itemize}
We now define two types of extremal graphs. Let $n,\,r$ be two  nonnegative integers with $r\mid n$,  and let $s$ be an odd integer satisfying $1\leq s\le \frac{n}{r}$. 
We say an $n$-vertex graph $G$ is \textit{extremal} if one of the following holds. 
 \begin{enumerate}
[label =\rm  (EX\arabic{enumi})]
    \item\label{EX1} The graph $G$ has an independent set of size $\frac{n}{r}+1$. 
    
    \item\label{EX2} 
    There exists a partition  $A_1\cup \cdots\cup A_{r-2}\cup B_0\cup B_1$ of $V(G)$ with $|A_i|=\frac{n}{r}$ for each $i\in [r-2]$ and $|B_0|=s$ such that the edge set of $G$ is 
$$
E(G)=\bigcup_{i\in [r-2]}\{uv:u\in A_i, v\notin A_i\}\cup \bigcup_{j\in \mathbb{Z}_2}\{uv:u\in B_j, v\notin B_{j+1}\}\ \text{(see Figure \ref{fig:extremal})}.
$$   
    
\end{enumerate}
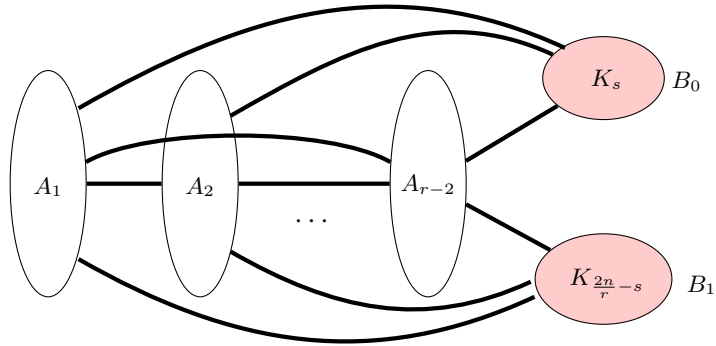
\begin{figure}[ht!]
    \centering
    \begin{tikzpicture}[
    node distance=1.5cm,
    set/.style={ellipse, draw, minimum width=1cm, minimum height=3cm},
    smallset/.style={ellipse, draw, minimum width=1.8cm, minimum height=1.2cm},
    ssmallset/.style={ellipse, draw, minimum width=1.6cm, minimum height=1.1cm},
    label/.style={above, font=\footnotesize},
]

\foreach \i in {1,...,2} {
    \node[set, fill=blue!0] (A\i) at (\i*2,0) {};
    \node[label] at (\i*2,-0.3) {$A_\i$};
}
\node at (5.5,-0.5) {$\cdots$};
\node[set, fill=blue!0] (A3) at (7,0) {};
\node[label] at (7,-0.3) {$A_{r-2}$};
\draw[-, ultra thick, black] (A1) 
    to[out=30, in=150, looseness=0.6] 
    node[above, pos=0.5] {} (A3);
\node[ssmallset, fill=red!20, right=of A3] at (7,1.4) (B0) {};
\node[label] at (9.35,1.12) {$K_s$};
\node[label] at (10.4,1.1) {$B_0$};

\node[smallset, fill=red!20, below=of B0] (B1) {};
\node[label] at (9.35,-1.65) {$K_{\frac{2n}{r}-s}$};
\node[label] at (10.6,-1.6) {$B_1$};

\draw[ultra thick,black] (A1)--(A2);
\draw[ultra thick,black] (A2)--(A3);
\draw[ultra thick,black] (A3)--(B0);
\draw[ultra thick,black] (B1)--(A3);
\draw[-, ultra thick, black] (4.4,0.9) 
    to[out=30, in=155, looseness=1] 
    node[above, pos=0.5] {} (B0);
\draw[-, ultra thick, black] (4.4,-0.9) 
    to[out=-30, in=-155, looseness=1] 
    node[below, pos=0.5] {} (8.35,-1.3);
\draw[-, ultra thick, black] (2.4,1) 
    to[out=30, in=155, looseness=1] 
    node[above, pos=0.5] {} (8.8,1.8);
\draw[-, ultra thick, black] (2.4,-1) 
    to[out=-30, in=-155, looseness=1] 
    node[below, pos=0.5] {} (8.4,-1.5);





\end{tikzpicture}
    \caption{(EX2). Each thick line indicates that all possible edges exist between the two corresponding parts.}
    \label{fig:extremal}
\end{figure}

Define $\sigma(G):=\min \left\{d(x)+d(y): x,y\in V(G) \ \text{and} \ xy\notin E(G)\right\}$. Hence the condition  $d(x)+d(y)\leq 2k$ for every edge $xy\in E(G)$ is equivalent to $\sigma(G)\geq 2(n-k)-2$. Having established the above, we now only need to prove the following result, as Theorem~\ref{mainore}  follows directly by considering the complement graph. 
\begin{theorem}\label{FVersion}
    Let $r$ be a positive integer. There exists an integer $n_0:=n_0(r)$ such that the following holds for every integer $n\geq n_0$ with $r\mid n$. Let $G$ be an $n$-vertex graph with $\sigma(G)\geq 2(1-\frac{1}{r})n-2$. Then either $G$ admits a $K_r$-factor or $G$ is extremal. Moreover, there
is an algorithm with polynomial time 
that decides whether $G$ has a $K_r$-factor. 
\end{theorem}

\subsection{Sketch of the proof of Theorem \ref{FVersion}}

Suppose that $r\ge 3$. Given $\gamma>0$ and  a vertex subset $S\subseteq V(G)$, we say $S$ is a $\gamma$-\textit{independent set} of an $n$-vertex graph $G$ if $|E(G[S])|\leq \gamma n^2$. Let $G$ be an $n$-vertex graph with $\sigma(G)\geq 2(1-\frac{1}{r})n-2$. We divide the proof of Theorem \ref{FVersion} into two parts according to $G$ contains a $\gamma$-independent set of size $\frac{n}{r}$ (extremal case) or not (non-extremal case). 

\medskip
\textbf{Non-extremal case: $G$ contains no $\gamma$-independent sets of size $\frac{n}{r}$.}
\medskip

For the non-extremal case, our proof primarily adopts the framework of the absorption method proposed by R\"odl, Ruci\'nski and Szemer\'edi~\cite{RRS2009}, which provides an efficient framework for constructing spanning subgraphs.
The absorption method typically decomposes the problem of finding perfect tilings into two steps. First, one constructs a `small' absorbing set $M$ in the host graph $G$, which can `absorb' any small set of `remaining' vertices. The second step is to find an almost perfect $K_r$-tiling covering almost all vertices in $G-M$.

{\bf Absorbing.} 
We are to find a `small'
$K_r$-absorbing set $M$ in $G$ that absorbs any `very small' set of vertices in $G$. To prove that such a $K_r$-absorbing set $M$ exists, we refine  the widely used method by R\"odl, Ruci\'nski, and Szemer\'edi~\cite{RRS2009} and H\`an, Person, and Schacht  \cite{Han2009}. That is, we show that
$$
\textit{almost all}\ r\text{-subset}\ T\subseteq V(G)\ \text{has}\  \Omega(n^{r^2})\ K_r\text{-absorbers for}\ T.
$$

 
{\bf Almost cover.} 
 We apply the Regularity Lemma to $G$ and thus obtain a reduced graph $R$ that inherits the Ore's condition of $G$. Our goal is to find an almost perfect $K_r$-tiling of $R$. We begin by greedily selecting a maximal family of vertex-disjoint $K_r$ copies in $G$. 
 Then, for each 
$s=r-1,r-2,\ldots,1$ in descending order, we iteratively maximize the number of vertex-disjoint $K_s$-copies in the remaining graph. This process constructs  a maximum $(K_r,K_{r-1},\ldots,K_1)$-factor in $R$ (see Definition \ref{maximum-factor}). 

Suppose that $R$ contains no almost perfect $K_r$-tiling. Assume the largest
$K_r$-tiling in $R$ covers exactly  $q\leq (1-o(1))|R|$  vertices. We then prove that the blow-up graph  $R[2^{r-1}]$ admits a maximum $(K_r, K_{r-1}, \ldots, K_1)$-factor covering significantly more than $q \cdot 2^{r-1}$ vertices. 
Thus,
crucially, the largest $K_r$-tiling in $R[2^{r-1}]$ covers a larger proportion of vertices than that in $R$. By repeating this process, we obtain a blow-up graph $R'$ of $R$ that
contains an almost perfect $K_r$-tiling.  Our approach adapts a proof technique of Treglown \cite{Treg2016}, developed for finding $K_r$-factors under a degree sequence condition.

\medskip
\textbf{ Extremal case: $G$ contains a $\gamma$-independent set of size $\frac{n}{r}$.}
\medskip


For the extremal case, we will make use of a result 
concerning $K_r$-factors in balanced $r$-partite graphs (Lemma~ \ref{GHH2024}) as well as the Graph Removal Lemma  (Lemma~\ref{GRL}). 
Define $S:=\{v:\in V(G):d(x)<(1-\frac{1}{r})n-1\}$. Clearly, $G[S]$ is a clique by Ore's condition. Since $G$ has a $\gamma$-independent set of size $\frac{n}{r}$, one may construct a \textit{good partition} $\mathcal{Q}=\{A_1,A_2,\ldots,A_s,B\}$ of $V(G)$ (see Definition~\ref{Partition}) such that each $|A_i|=\frac{n}{r}$ and $S\subseteq B$, where each $A_i$ is an almost independent set while $G[B]$ has no $\gamma$-independent set of size $\frac{n}{r}$. 

Under the partition $\mathcal{Q}$, the process of using Lemma~\ref{GHH2024} to construct a $K_r$-factor in $G$ consists of the following two steps:

Step 1. Cover all non-excellent vertices (i.e., those failing the minimum degree condition) using vertex-disjoint $K_r$ copies. To achieve this, we first identify a small structural `base' that covers all such vertices. This base is selected with extension properties that guarantee it can be extended into a $K_r$-tiling while maintaining the original vertex proportion across all parts. Denote the resulting $K_r$-tiling obtained in this step by $\mathcal{T}$. 

Step 2. Find a $K_{r-s}$-tiling in $G[B\setminus V(\mathcal{T})]$. For $r-s\geq 3$, such a $K_{r-s}$-tiling can be obtained by using the non-extremal case. Now, suppose that $r-s=2$ and $G[B\setminus V(\mathcal{T})]$ contains no perfect matching. Assume that $G$ is not extremal.  Proposition \ref{PMatching} implies that $G[B\setminus V(\mathcal{T})]$ consists of two odd components. 
By using some structure analysis, we can perform  minor adjustments to the bases obtained in Step 1 to construct a new $K_r$-tiling $\mathcal{T}'$ that covers all non-excellent vertices and $G[B\setminus V(\mathcal{T}')]$ contains a perfect matching.

After the above steps, we contract each $K_{r-s}$ copy in the $K_{r-s}$-tiling of  $G[B\setminus V(\mathcal{T})]$ into a single vertex. This allows us to apply Lemma~\ref{GHH2024} to obtain a $K_{s+1}$-factor in the remaining graph, which in turn yields a $K_r$-factor in the graph $G- V(\mathcal{T}')$. Combining this with $\mathcal{T}$ or $\mathcal{T}'$ yields the desired $K_r$-factor in $G$.

\subsection{Organization}
In Section \ref{section2}, we present some preliminaries and results on matchings in graphs. Section~3 is devoted to the proof of Theorem~\ref{FVersion}, which is divided into two parts: the non-extremal case (Theorem~\ref{non-extremal}) and the extremal case (Theorem~\ref{Extremal case}). The non-extremal case is proved in Section~\ref{section4} by using the almost cover lemma (Lemma~\ref{almost-cover}, proved in Section \ref{section4.2}) and the absorbing lemma (Lemma~\ref{absorbing}, proved in Section \ref{section4.3}). Section~\ref{section5} deals with the extremal case and provides a proof of Theorem~\ref{Extremal case}. The proof consists of two main steps: constructing a good partition of $G$ (Lemma~\ref{Lpartition}, proved in Section~\ref{section6}) and covering all non-excellent vertices (Lemma~\ref{K_r-tiling}, proved in Section~\ref{section7}). In the last section, we show that there
is an algorithm with polynomial time 
that decides whether $G$ has a $K_r$-factor, and some concluding remarks are also given in this  section.

\section{Notation and Preliminaries}\label{section2}

\subsection{Notation}
Throughout the paper we follow the standard graph-theoretic notation~\cite{diestel2017graph}. Let $G=(V(G),E(G))$ be a  graph, where $V(G)$ and $E(G)$ are the vertex set and edge set of $G$, respectively.
The \textit{order} of a graph $G$, denoted by $|G|$, is the  number of vertices, and its \textit{size}, denoted by $e(G)$, is the  number of edges.   
Let $U\subseteq V(G)$ be a vertex subset. Denote by $G[U]$ the induced 
subgraph of $G$ on $U$. The notation $G-U$ is used to denote the induced subgraph after removing $U$,
that is $G-U:=G[V(G)\setminus U]$. Given two vertex subsets $A,B\subseteq V(G)$, we use $G[A,B]$ to denote the subgraph of $G$ with vertex set $A\cup B$ and edge set consisting of all edges having one endpoint in $A$ and the other in $B$. For a family $\mathcal{G}$ of graphs, denote by $|\mathcal{G}|$ the number of graphs in $\mathcal{G}$, and let $V(\mathcal{G}) = \bigcup_{G \in \mathcal{G}} V(G)$.

The \textit{neighborhood} of a vertex $v$ in a graph $G$, denoted $N_G(v)$, is the set of all vertices adjacent to $v$. The \textit{degree} of $v$ in $G$, denoted $d_G(v)$, is the number of vertices in $N_G(v)$; when there is no ambiguity, we abbreviate these as $N(v)$ or $d(v)$.  For two vertex subsets $S, A \subseteq V(G)$, denote  $N(S):=\bigcap_{v\in S}N(v)$, $N(v, A):=N(v)\cap A$ and $N(S, A):=N(S)\cap A$, 
 and denote by $d(v,A), d(S)$ and $d(S,A)$ the sizes of $N(v,A),N(S)$, and $N(S,A)$, respectively. When $A$ is a subgraph of $G$, we always abbreviate $V(A)$ as $A$ in the notations above.  Let $\delta(G):=\min\{d(v):v\in V(G)\}$ and $\Delta(G):=\max \{d(v):v\in V(G)\}$ be the \textit{minimum degree} and \textit{maximum degree} of $G$ respectively.   

A \textit{partition} of $V(G)$ is a family $\mathcal{P}=\{A_1,\ldots,A_s\}$ of disjoint subsets satisfying $V(G)=\bigcup_{i=1}^s A_i$. Given two graphs $H$ and $G$, an  \textit{$H$-tiling} of $G$ is a collection of vertex-disjoint copies of $H$ in the graph $G$. A \textit{perfect $H$-tiling} of $G$, or an \textit{$H$-factor}, is an $H$-tiling that covers all vertices of $G$. 
A \textit{matching} in a graph is a set of edges such that no two edges share a common vertex. A \textit{maximum matching} is a matching that contains the largest possible number of edges. We call a matching $M$ of $G$ a \textit{$d$-matching} if $e(M)=d$.


Consider a graph $F$ and an integer $t\in \mathbb{N}$. Let $F[t]$ denote the graph obtained from $F$ by replacing every vertex of $F$ 
 with an independent set of size $t$, and connecting two vertices from different sets if and only if their original vertices are adjacent in $F$. 
We will refer to $F[t]$ as a \textit{blow-up} copy of $F$. 
The \textit{join} of two vertex-disjoint graphs $G_1$ and $G_2$, denoted $G_1\vee G_2$, is formed by taking their disjoint union and then adding every possible edge between vertices of $G_1$ and vertices of $G_2$. 
 

Given a set $U$ 
and an integer $k$, we write
$\binom{U}{k}$ the collection of all 
$k$-element subsets of $U$. For any integers $a\leq b$, define $[a,b]:=\{i\in \mathbb{N}:a\leq i\leq b\}$ and $[b]:=[1,b]$. When we write  $\alpha\ll \beta\ll \gamma$, we always mean that $\alpha, \beta, \gamma$  are constants in $(0,1)$; and $\alpha\ll \beta$ means that there exists $\alpha_0=\alpha_0(\beta)$ such that the subsequent arguments hold for all $0<\alpha\leq \alpha_0$. Hierarchies of other lengths are defined analogously.

\subsection{Preliminaries}

In this subsection, we will give some preliminary results that will be used in our proof. The following fact from Ore's condition will be used throughout this paper. 

\begin{fact}\label{clique}
   Let $t$ be a positive integer. Assume that $G$ is an  $n$-vertex graph with $\sigma(G)\ge t$. Define $S:=\{x\in V(G): d(x)<\frac{t}{2}\}$. Then $G[S]$ is a complete graph.
\end{fact}



We next recall the definition of $\gamma$-independent set.  Given an $n$-vertex graph $G$ and a constant $\gamma > 0$, we define a vertex subset $S \subseteq V(G)$ as a $\gamma$-\textit{independent set} if $e(G[S])\leq \gamma n^2$. 
Now, we will use the following graph removal lemma to construct a large $\gamma$-{independent set} in graphs. 
\begin{lemma}[Graph Removal Lemma \cite{ADLR1994}]\label{GRL}

For any graph $H$ with $h$ vertices and any $\eps>0$, there exists $\delta>0$ such that any graph on $n$ vertices which contains at most $\delta n^h$ copies of $H$ can be made $H$-free by removing at most $\eps n^2$ edges.
\end{lemma}

\begin{proposition}\label{SuperT}
    Let $0< {1}/{n} \ll \eps \ll \alpha \ll {1}/{r}$ where $r\in \mathbb{N}$ and $r\ge 2$. Suppose that $G$ is an $n$-vertex graph that contains at most $\eps n^r$ copies of $K_r$ and 
$$\sigma(G)\ge 2\Big(1-\frac{1}{r-1}-\alpha\Big)n.$$
Then $G$ admits a $\sqrt{\alpha}$-independent set of size at least $\frac{n}{r-1}$.
\end{proposition}

\begin{proof}
    Choose constants satisfying 
    $$\frac{1}{n} \ll \eps\ll \gamma \ll \alpha \ll \frac{1}{r},
    $$
    let $G$ be a graph as in the assumption. The result is trivial when $r=2$ as $\eps\ll \alpha$. Thus, one may assume that $r\ge 3$. 
    Since $G$ contains at most $\eps n^r$ copies of $K_r$, Lemma~\ref{GRL} implies that one can remove at most $\gamma n^2$ edges from $G$ to obtain a spanning subgraph $G'$ that is $K_r$-free. 
    Define
    $$
    L:=\left\{v\in V(G):d_G(v)-d_{G'}(v)>2\sqrt{\gamma}n\right\}\ \text{and}\ G'':=G'-L.
    $$
    Clearly, $|L|\leq \sqrt{\gamma}n$. 
    Since $\gamma\ll \alpha$, 
    one has 
    $$
    n'':=|G''|\ge (1-\alpha)n,\ \text{and}\ \sigma(G'')\ge 2\Big(1-\frac{1}{r-1}-2\alpha\Big)n''.
    $$ 
    \begin{claim}\label{Independentset}
        The graph $G''$ contains an independent set $S$ of size at least 
$\big(\frac{1}{r-1}-3r \alpha\big)n.$
    \end{claim}
    \begin{proof}[Proof of Claim \ref{Independentset}]
        Let $S:=\left\{v\in V(G''):~d(v)< (1-\frac{1}{r-1}-2\alpha)n''\right\}$ and $\Tilde{G}:= G''-S$. Fact \ref{clique} indicates that $G''[S]$ is a clique. Notice that $G''$ is $K_r$-free, we have $|S|\le r-1$.  Then 
\begin{align}\label{MinKfree}
\delta(\Tilde{G})\ge \Big(1-\frac{1}{r-1}-2\alpha\Big)n''-|S|\ge \Big(1-\frac{1}{r-1}-2\alpha\Big)|\Tilde{G}|-r.
\end{align} 
Thus, by Theorem \ref{KK2008} (or by  Hajnal-Szemer\'{e}di Theorem), for each vertex $v\in V(\Tilde{G})$, there exists a copy of $K_{r-2}$  containing $v$ in $\Tilde{G}$, say $K^v$. 
It follows from \eqref{MinKfree} that 
\begin{align*}
\Big|\bigcap_{w\in V(K^v)}N_{\Tilde{G}}(w)\Big|&\ge \left(1-\frac{r-2}{r-1}-2(r-2)\alpha\right)|\Tilde{G}|-r(r-2)\\
&\ge \left(\frac{1}{r-1}-2(r-1) \alpha\right)n\ge \left(\frac{1}{r-1}-3r \alpha\right)n.
\end{align*}
Since $G''$ is $K_r$-free, we obtain that $\bigcap_{w\in V(K^v)}N_{G''}(w)$ is an independent set of size at least  $\big(\frac{1}{r-1}-3r \alpha\big)n$ in $G''$, as desired.
    \end{proof}

    
    Let $S$ be an independent set of size at least 
$\big(\frac{1}{r-1}-3r \alpha\big)n$ in $G''$. 
The construction of $G''$ ensures that $S$ is a $\gamma$-independent set in $G$. By augmenting $S$ with at most $3r\alpha n$ additional vertices, we obtain a $\sqrt{\alpha}$-independent set in $G$ of size at least $\frac{n}{r-1}$, as required.
\end{proof}


The following theorem provides a degree condition that guarantees the existence of $K_r$-factors in a balanced $r$-partite graph, which follows from \cite[Lemma 6.3]{GHH2024} by using Edmonds algorithm. 
\begin{lemma}[\cite{GHH2024}]\label{GHH2024} 
For any $r,n\in \mathbb{N}$, let $G$ be an $r$-partite graph on vertex classes $V_1,\ldots,V_r$ where $|V_i|=n$ for each $1\le i\le r$. If  
$$
\min_{i\in [r]}\min_{v\in V_i}\min_{j\in [r]\setminus \{i\}} |N(v)\cap V_j|\geq \Big(1-\frac{1}{2r}\Big)n,
$$
then in time $O(rn^4)$ we can find a $K_r$-factor of $G$.
\end{lemma}

We will use the following version of Chernoff's bound.
\begin{lemma}\label{chernoff}
    Let $X$ be a random variable with binomial or hypergeometric distribution, and let $0<\eps<\frac{3}{2}$. Then 
    $$
    \mathbb{P}[|X-\mathbb{E}(X)|\geq \eps \mathbb{E}(X)]\leq 2e^{-\frac{\eps^2}{3}\mathbb{E}(X)}.
      $$
\end{lemma}

\subsection{Matchings in graphs}
In this subsection, we  introduce some preliminaries that are related to matchings in graphs. 
\begin{lemma}[\cite{Treg2015}]\label{matching}
    Let $d,n \in \mathbb{N}$. Suppose that $G$ is a graph on $n\ge 2d$ vertices with  $\delta(G)\ge d$. Let $X\subseteq V(G)$ be a subset of size $d$.  Then $G$ contains a $d$-matching that covers all vertices in $X$. 
\end{lemma}
Let $M$ be a maximum matching of a graph $G$ and $v$ be a vertex in $V(G)\setminus V(M)$, define 
$$
SN(v):=\{z\in V(M):wz\in E(M) \text{ for some } w\in N(v)\}.
$$
The following fact follows immediately from the definition of a maximum matching; the proof is omitted.  
\begin{fact}\label{MMProperty}
   Let $M$ be a maximum matching of a graph $G$ and $x,y$ be two distinct vertices in $V(G)\setminus V(M)$. 
   Then $E(G[SN(x),SN(y)])=\emptyset$. 
\end{fact}

    

\begin{lemma}\label{MStab}
    Let $d\ge 1$ and $n\ge 3d+2$. Assume that $G$ is an $n$-vertex graph with $d\leq \delta(G)\le \Delta(G)\leq n-2d-1$. Then $G$ admits a matching of size $d+1$. 
\end{lemma}

\begin{proof}
   Choose a maximum matching, say $M,$ inside $G$. Suppose that $e(M)\leq d+1$. Together with Lemma \ref{matching}, we obtain $e(M)=d$.  Let $x_1,x_2,x_3$ be three distinct vertices in $V(G)\setminus V(M)$. Clearly, $N(x_i)\subseteq V(M)$ for each $i\in [3]$.  We are to prove that $N(x_1)=N(x_2)=N(x_3)$. 

   For each $i\in [3]$, assume that there are $a_i$ edges in $M$ such that both endpoints of them are adjacent to $x_i$. If follows from Fact \ref{MMProperty} that $E(SN(x_i),SN(x_j))=\emptyset$ for $1\leq i<j\leq 3$. 
   Therefore, 
   $d(x_i)\leq 2a_i+d-(a_1+a_2+a_3)$ for each $i\in [3]$. Together with $\delta(G)\geq d$, one has 
   $$
   3d\leq d(x_1)+d(x_2)+d(x_3)\leq 3d-(a_1+a_2+a_3).
   $$
   Hence $a_1=a_2=a_3=0$ and $d(x_1)=d(x_2)=d(x_3)=d$. 
   Consequently, every vertex in $V(G)\setminus V(M)$ is adjacent to exactly one vertex in each edge of $M$. Applying Fact \ref{MMProperty} again yields that every two distinct vertices in $V(G)\setminus V(M)$ share the same neighbors in $V(M)$. Hence each vertex $v\in N(x_1)$ satisfies 
   $
d(v)\geq n-2d
    $, which contradicts the maximum degree  condition of $G$. 
\end{proof}
The main goal of this subsection is to prove the following proposition. 
\begin{proposition}\label{PMatching}
Let $n$ be a positive even integer and $\gamma$ be a constant such that $0<{1}/{n}\ll\gamma\ll 1$. Suppose that $G$ is a graph on $n$ vertices so that 
\begin{align}\label{B2Ore}
\sigma(G)\ge n-\gamma n.
\end{align}
Then one of the following holds:
\begin{itemize}
\setlength{\parskip}{0pt}
    \item[{\rm (i)}] $G$ admits a perfect matching;
    \item[{\rm (ii)}] $G$ contains a $2\gamma$-independent set of size at least $\frac{n}{2}$;
    \item[{\rm (iii)}] $G$ consists of two odd components $C_1$ and $C_2$. Moreover, if $|C_i|\leq \frac{1-\gamma}{2}n$ for some $i\in [2]$, then 
    $C_i$ is a clique.
\end{itemize}
\end{proposition}

\begin{proof}
Suppose that $G$ has no perfect matching. Let $M$ be a maximum matching in $G$. Then, $|V(M)|\leq n-2$ and there are two  distinct vertices $x,y\in V(G)\setminus V(M)$. 
The maximality of $M$ implies that $N(x)\cup  N(y)\subseteq  V(M)$, and so $SN(x)\cup SN(y)\subseteq V(M)$. Since $xy\notin E(G)$,  
by \eqref{B2Ore} we have 
\begin{align}\label{SNOre}
   d(x)+d(y)= |SN(x)|+|SN(y)|\ge n-\gamma n.
\end{align}
Define $SN(x,y):=SN(x)\cap SN(y)$. Fact \ref{MMProperty} implies that $SN(x,y)$ is an independent set of $G$ and  
\begin{align}\label{degreesum}
    d(x)+d(y)\leq 2e(M)\leq n-2.
\end{align}
We divide our proof into two cases.

\medskip
{\bf Case 1. } 
$SN(x,y)\ne \emptyset$. 
\medskip

In this case, we will show that (ii) holds. 
Without loss of generality, assume that $d(x)\le d(y)$. Then \eqref{degreesum} implies $d(x)<\frac{n}{2}$. Choose $z\in SN(x,y)$ and assume $zz'\in E(M)$. Hence  $z'\in N(x)\cap N(y)$ and  $N(z)\subseteq V(M)$.  Thus, \eqref{B2Ore} implies that 
$$
d(z)\ge n-\gamma n-d(x)>\frac{n}{2}-\gamma n.
$$
Notice that $N(z)\cap (SN(x)\cup SN(y))=\emptyset$. Otherwise, there exists a matching of size $|E(M)|+1$, a contradiction. Thus, 
$|SN(x)\cup SN(y)|\le n-d(z)<\frac{n}{2}+\gamma n$. Together  with \eqref{SNOre}, one has  
$$|SN(x)\cap SN(y)|= |SN(x)|+|SN(y)|-|SN(x)\cup SN(y)|\ge \frac{n}{2}-2\gamma n.$$
Recall that $SN(x,y)$ is an independent set in $G$. By adding at most $2\gamma n$ arbitrary vertices to $SN(x,y)$, we obtain a $2\gamma$-independent set of size at least $\frac{n}{2}$ in $G$, as desired.

\medskip
{\bf Case 2.} 
$SN(x,y)=\emptyset.$
\medskip

In this case, one may assume that 
$SN(x,y)=\emptyset$ for every maximum matching $M$ of $G$ and any two distinct  vertices $x,y\in V(G)\setminus V(M)$. Suppose that  $|V(G)\setminus V(M)|>2$. Then there is a vertex $z\in V(G)\setminus (V(M)\cup \{x,y\})$. It follows from  \eqref{B2Ore} that 
$$d(x)+d(y)+d(z)\ge \frac{3}{2}(n-\gamma n)>|V(M)|.$$ 
Recall that $N(x)\cup N(y)\cup N(z)\subseteq V(M)$. Therefore, there are two vertices $u,v\in \{x,y,z\}$ such that $N(u)\cap N(v)\neq \emptyset$, thus $SN(u,v)\neq \emptyset$, a contradiction. Hence, $|V(G)\setminus V(M)|=2$. We will show that (iii) holds. 

Since $SN(x,y)=\emptyset$ and $E(G[SN(x),SN(y)])=\emptyset$, there are two vertex-disjoint subgraphs $G_1:=G[\{x\}\cup SN(x)]$ and $G_2:=G[\{y\}\cup SN(y)]$ in $G$ such that $E(G[V(G_1),V(G_2)])=\emptyset$. Define $R:=V(G)\setminus (V(G_1)\cup V(G_2))$. If $R=\emptyset$, then the two endpoints of each edge in $M$ are adjacent to the same vertex in $\{x, y\}$. It follows that $G$ consists of two connected odd components $G_1$ and $G_2$, as desired. Suppose that $R\neq \emptyset$. In order to prove (iii) holds, it is necessary to assign vertices in $R$ to $G_1$ or $G_2$. Choose a vertex $v \in R$ with $vv'\in E(M)$. Hence $xv',yv'\notin E(G)$. Together with $xy\notin E(G)$ and \eqref{B2Ore}, we have 
$$
d(x)+d(y)+d(v')=|SN(x)|+|SN(y)|+|N(v')|\geq \frac{3}{2}(n-\gamma n).
$$
Recall that $SN(x,y)=\emptyset$. Hence, 
\begin{align}\label{eitheror}
    \text{either}\ SN(x)\cap N(v')\ne \emptyset\ \text{or}\ SN(y)\cap N(v')\ne \emptyset.
\end{align}

$\bullet$\ Assume  $v'\notin R$. Then $v'\in SN(x)\cup SN(y)$. Without loss of generality, assume that $v'\in SN(x)$
, that is, $xv\in E(G)$ and $yv\notin E(G)$. Suppose that $N(v)\cap SN(y)\ne \emptyset$. Then choose $u\in N(v)\cap SN(y)$ with $uu'\in E(M)$. Hence $M':=(M\setminus \{vv',uu'\})  \cup \{vu,u'y\}$  is also a maximum matching of $G$. Notice that $v\in N(x)\cap N(v')$. Therefore, under the matching $M'$ we have $SN(x,v')\neq \emptyset$, contradicting our initial assumption in Case 2. Thus, $N(v)\cap SN(y)= \emptyset$. Together with $vy\notin E(G)$, we know $N(v)\cap V(G_2)=\emptyset$. Thus, one may add such a vertex $v$ to $V(G_1)$. 

$\bullet$\ Assume  $v'\in R$. Then $xv,yv\notin E(G)$. By a similar argument as \eqref{eitheror}, we obtain that
\begin{align}\label{either2}
    \text{either}\ SN(x)\cap N(v)\ne \emptyset\ \text{or}\ SN(y)\cap N(v)\ne \emptyset. 
\end{align}
We claim that 
none of the following holds:
\begin{enumerate}
    \item[{\rm (1)}] $SN(x)\cap N(v)\ne \emptyset$ and $SN(y)\cap N(v')\neq \emptyset$;
    \item[{\rm (2)}] $SN(x)\cap N(v')\ne\emptyset$ and $SN(y)\cap N(v)\ne \emptyset$. 
\end{enumerate}
Otherwise, without loss of generality, suppose the former case happens. Let $a\in SN(x)\cap N(v)$ and $b\in SN(y)\cap N(v')$, and assume  $aa',bb'\in E(M)$. Then $M':=(M\setminus\{aa',bb',vv'\})\cup \{xa',av,v'b,b'y\}$ is a perfect matching of $G$, a contradiction. Together with \eqref{eitheror} and \eqref{either2}, we obtain that 
\begin{align}\label{belong}
    \text{either}\ SN(x)\cap (N(v)\cup N(v'))= \emptyset\ \text{or}\ SN(y)\cap (N(v)\cup N(v'))= \emptyset.
\end{align}
Together with $xv,xv',yv,yv'\notin E(G)$, we have either $(N(v)\cup N(v'))\cap V(G_1)=\emptyset$ or $(N(v)\cup N(v'))\cap V(G_2)=\emptyset$. Therefore, $v$ and $v'$ can be assigned to $G_1$ or $G_2$ based on where their neighbors are located, which preserves the absence of edges between $G_1$ or $G_2$. 

By sequentially applying the above operation to each vertex in 
$R$, we can assign them to either 
$G_1$ or $G_2$ accordingly. In order to preserve the disconnectivity  between $G_1$ and $G_2$, we only need to show that in the above process, for each edge in $G[R]$, 
its two endpoints are assigned to the same part. Let $uv\in E(G[R])$ with $uu',vv'\in E(M)$. We proceed our proof by considering the following three subcases.



\medskip
{\bf Subcase 2.1.} $u',v'\notin R$.
\medskip

We claim that there exists a vertex $z\in \{x,y\}$ such that 
$
u',v'\in SN(z).
$ 
Otherwise, suppose that $u'\in SN(x)$ and $v'\in SN(y)$. By the argument in the above,  we obtain that $N(u)\cap SN(y)=\emptyset$ and $N(v)\cap SN(x)=\emptyset$. 
Without loss of generality, assume that $d(x)\le d(y)$. Then $d(x)\le \frac{n}{2}$. It follows from $v\in R$ that $xv'\notin E(G)$. Together with  \eqref{B2Ore} one has $d(v')\geq \frac{(1-2\gamma)n}{2}$. By $v'\in SN(y)$ and Fact~\ref{MMProperty}, we obtain $N(v')\cap SN(x)=\emptyset$. Thus,  $N(v')\subseteq SN(y)\cup R$. Notice that $|R|\leq \gamma n$.  Hence, there exists a vertex $a\in SN(y)\cap N(v')$. Therefore, $M':=(M\setminus \{aa',uu',vv'\})\cup \{uv,v'a,a'y\}$ is also a maximum matching in $G$. But $SN(x,u')\neq \emptyset$, a contradiction. This implies that there exists a vertex $z\in \{x,y\}$ such that 
$
u',v'\in SN(z). 
$
Therefore, $u$ and $v$ are assigned to the same part by using the above argument. 


\medskip
{\bf Subcase 2.2.} $u'\in R,\,v'\notin R$.
\medskip

By a similar discussion as Subcase 2.1, we can show that $u$ and $v$ are assigned to the same part. 

\medskip
{\bf Subcase 2.3.} $u',v'\in R$.
\medskip

We claim that there exists a vertex $z\in \{x,y\}$ such that 
$$
SN(z)\cap (N(v)\cup N(v'))=\emptyset\ \text{and}\ SN(z)\cap (N(u)\cup N(u'))=\emptyset.
$$ 
Otherwise, based on \eqref{belong}, one may assume that $SN(x)\cap (N(v)\cup N(v'))=\emptyset$ and $SN(y)\cap (N(u)\cup N(u'))=\emptyset$. Furthermore, $SN(y)\cap (N(v)\cup N(v'))\neq \emptyset$ and $SN(x)\cap (N(u)\cup N(u'))\neq \emptyset$. By $uv\in E(G)$ and the maximality of $M$, one has either $SN(y)\cap N(v')=\emptyset$ or $SN(x)\cap N(u')=\emptyset$. 
Assume, without loss of generality, that $SN(y)\cap N(v')=\emptyset$. 
Together with $SN(x)\cap N(v')=\emptyset$, one has $N(v')\subseteq R$. Notice that  $xv',yv'\notin E(G)$. By  \eqref{B2Ore}, one has $$d(x),d(y)\geq n-2\gamma n,$$
which contradicts \eqref{degreesum}. 

With the above discussion, by assigning vertices in $R$ to $G_1$ or $G_2$ accordingly, we know that $G$ can be decomposed into two connected components, say $C_1$ and $C_2$, each of which has odd order.
Furthermore, by Ore's condition, if $|C_i|\leq \frac{1-\gamma}{2}n$ for some $i\in [2]$, then $C_i$ is a clique, as desired.
\end{proof}

\section{Proof of Theorem \ref{FVersion}}\label{section3}

In this section, we will give a quick proof of Theorem \ref{FVersion}. The proof is divided into two cases, depending on whether the graph contains a $\gamma$-independent set of size ${n}/{r}$ (the extremal case) or not (the non-extremal case). 
\begin{theorem}[Non-extremal case]\label{non-extremal}
   Let $r\geq 3$ be an integer. For any $0< \alpha\ll \gamma\ll \frac{1}{r}$, there exists $n_0\in \mathbb{N}$ such that the following holds for all $n\geq n_0$ with $r\mid n$. 
   Assume that $G$ is a graph on $n$ vertices satisfying  $\sigma(G)\geq 2(1-\frac{1}{r}-\alpha)n$. If $G$ contains no $\gamma$-independent set of size $\frac{n}{r}$, then $G$ has a $K_r$-factor. 
\end{theorem}

\begin{theorem}
    [Extremal case]\label{Extremal case}
   Let $r\geq 3$ be an integer. 
   There exist $\gamma > 0$ and $n_0\in \mathbb{N}$ such that the following holds for each $n\ge n_0$ with $r \mid n$. Suppose that $G$ is a graph on $n$ vertices with $\sigma(G) \ge 2(1-\frac{1}{r})n-2$. 
If $G$ contains a $\gamma$-independent set of size $\frac{n}{r}$, then 
\begin{itemize}
\setlength{\parskip}{0pt}
    \item[$\mathrm{(i)}$] either $G$ has a $K_r$-factor, or 
    \item[$\mathrm{(ii)}$] $G$ is an extremal graph,  i.e., \ref{EX1} or \ref{EX2}. 
\end{itemize}
\end{theorem}

We now prove the first part of Theorem~\ref{FVersion}; the algorithmic part will be presented in the final section.

\begin{proof}[\bf Proof of Theorem \ref{FVersion}]
The case $r=1$ is trivial. If $r\geq 3$, then by combining Theorem~\ref{non-extremal} and Theorem~\ref{Extremal case},  we obtain Theorem \ref{FVersion} immediately. In the following, we consider the case that $r=2$ and $n\geq 4$ is even. 


Suppose that $G$ is an $n$-vertex graph with $\sigma(G)\ge n-2$. 
Let  $P:=u_1u_2\ldots u_t$ be a longest path in $G$. Hence $N(u_1)\cup N(u_t)\subseteq V(P)$. If $t=n$, then $G$ admits a $K_2$-factor, as desired. In the following, we assume that $t\le n-1$. 

We first assume that $G[V(P)]$ contains a Hamilton cycle. 
Then any vertex in $P$ is not adjacent to any vertex outside $P$. Hence for each $i\in [t]$ and each $w\in V(G)\setminus V(P)$, we have 
$$
n-2\leq d(u_i)+d(w)\leq (t-1)+(n-t-1)=n-2.
$$
Thus, $d(u_i)=t-1$ for each $i\in [t]$ and $d(w)=n-t-1$ for each $w\in V(G)\setminus V(P)$. It follows that $G=K_t\cup K_{n-t}$. Therefore, $G$ has a $K_2$-factor if $t$ is even, and $G$ satisfies \ref{EX2} otherwise, as desired. 

Next, we assume that $G[V(P)]$ contains no Hamilton cycles. Then $u_1u_t\notin E(G)$. 
Let
$$
A=\{j\in [t-2]:u_1u_{j+1}\in E(G)\}\ \text{and}\ B=\{j\in [2,t-1]:u_tu_{j}\in E(G)\}.
$$
Obviously, $A\cap B=\emptyset$ as $G[V(P)]$ contains no  Hamilton cycles. 
Together with Ore's condition, one has
$$
n-2\leq d(u_1)+d(u_t)=|A|+|B|=|A\cup B|\leq t-1\leq n-2.
$$
Thus, $t=n-1$ and $A\cup B=[t-1]$.  Without loss of generality, assume that $|A|\leq |B|$. Hence  $|A|\leq \frac{n}{2}-1\leq |B|.$ 
Define 
$$
C=\{j\in [t]:u_tu_{j-1}\in E(G)\}.
$$
Then $C\subseteq [3,t]$ and $|C|=|B|$. Moreover,
$$
t\in C\setminus A,\ \text{and}\ 1\in A\setminus C.
$$
For each $j\in A$, there is a path $u_ju_{j-1}\ldots u_1u_{j+1}\ldots u_t$ with endpoints $u_j$ and $u_t$; for each $j\in C$, there is a path $u_1u_2\ldots u_{j-1}u_{t}u_{t-1}\ldots u_{j}$ with endpoints $u_1$ and $u_j$. Let $x$ be the unique vertex in $V(G)\setminus V(P)$. Recall that the longest path in $G$ has order $n-1$. Hence $xu_j\notin E(G)$ for every $j\in A\cup C$. 
In particular, $xu_1\notin E(G)$. The Ore's condition implies that 
$$
n-2\leq d(u_1)+d(x)\leq |A|+(n-1-|A\cup C|)\leq 
n-1-|C\setminus A|\leq n-2. 
$$
It follows that $d(x)=n-1-|A\cup C|$ and $C\setminus A=\{t\}$. 
Together with the fact that $1\in A\setminus C$, we have $C\setminus \{t\}\subseteq A\setminus \{1\}$. Therefore, 
$$
n-2=|A|+|B|=|A|+|C|=|A|+|C\setminus \{t\}|+1\leq |A|+|A\setminus \{1\}|+1 =2|A|\leq n-2.
$$
It follows that $C\setminus \{t\}=A\setminus \{1\}$ and   $|A|=|B|=|C|=\frac{n}{2}-1$. 

Choose $j\in A\cap C$. Then $j\in [3,n-3]$ and $j-1\in B$.  Together with $A\cap B=\emptyset$, one has ${j-1}\notin A$. Since $|A\cap C|=\frac{n}{2}-2$, we obtain $A\cap  C=\{3,5,\ldots,n-3\}$. Therefore, $N(u_1)=N(u_{n-1})=\{u_2,u_4,\ldots,u_{n-2}\}$. For each even integer $i\in [n-1]$, define  $P^i:=u_1u_2\ldots u_{i}u_{n-1}u_{n-2}\ldots u_{i+1}$  to be a new path with endpoints $u_1$ and $u_{i+1}$. By a similar discussion as above, we obtain that for each even integer $i\in [n-1]$,
$$
N(u_1)=N(u_{i+1})=\{u_2,u_4,\ldots,u_{n-2}\}\ \text{and}\ xu_{i+1}\notin E(G).
$$
It follows that $\{x,u_1,u_3,\ldots,u_{n-1}\}$ forms an independent set of $G$ with size $\frac{n}{2}+1$, as desired.   
\end{proof}

\section{Non-extremal case}\label{section4}


For the non-extremal case, our proof primarily adopts the framework of the absorption method proposed by R\"odl, Ruci\'nski and Szemer\'edi~\cite{RRS2009}. 
The absorption method typically decomposes the problem of finding perfect tilings into two steps: almost cover and the absorption of the remaining vertices.

\begin{lemma}[Almost cover]\label{almost-cover}
    Let $0<\frac{1}{n}\ll\mu\ll \alpha\ll \gamma\ll \frac{1}{r}$ and $r\ge 3$. Assume that $G$ is a graph on $n$ vertices such that $\sigma(G)\geq 2(1-\frac{1}{r}-\alpha)n$. If $G$ does not contain any $\gamma$-independent set of size $\frac{n}{r}$, then $G$ has a $K_r$-tiling that covers all but at most $\mu n$ vertices. 
\end{lemma}
We need the following notation of
 absorbers and absorbing sets.
 \begin{definition}[Absorbing set, abosrber]
      Let $G$ be an $n$-vertex graph and $r\in \mathbb{N}$. 
      \begin{itemize}
          \item A subset $M\subseteq V(G)$ is an \textit{$\eps$-absorbing set} in $G$ for some constant $\eps$ if for any subset $U\subseteq V(G)\setminus M$ of size at most $\eps n$ and $|M\cup U|\in r\mathbb{N}$, $G[M\cup U]$ contains a $K_r$-factor.
          \item For an $r$-set $Q\subseteq V(G)$, a  subset $S\subseteq V(G)$ is called a $K_r$-\textit{absorber} 
for $Q$ if $|S|=r^2$ and both $G[S]$ and $G[S\cup Q]$ contain $K_r$-factors.
      \end{itemize}
 \end{definition}

For example, in Figure \ref{fig:absorber}, $\{v_1,\ldots,v_9\}$ is a $K_3$-absorber for $\{u_1,u_2,u_3\}$. 

\begin{figure}[ht!]
    \centering
    \begin{tikzpicture}[scale=0.8,
    node distance=1.5cm,
    set/.style={ellipse, draw, minimum width=1cm, minimum height=2cm},
    smallset/.style={ellipse, draw, minimum width=0.23cm, minimum height=0.9cm},
    label/.style={above, font=\normalsize},
    point/.style={circle, fill=black, inner sep=1pt},
    box/.style={draw, rectangle, minimum width=5.2cm, minimum height=3cm},
    verysmallset/.style={ellipse, draw, minimum size=0pt,  
    inner sep=0pt,     
    outer sep=0pt, minimum width=0.2cm, minimum height=0.5cm},
    box/.style={draw, rectangle, minimum width=9cm, minimum height=3.3cm},
]
\node[box, fill=blue!0] (A1) at (3.5,-0.5) {};
    \coordinate (a) at (0, 3);
    \fill (a) circle (2pt);  
    
    \coordinate (b) at (-1, 1);
    \fill (b) circle (2pt);
    
    \coordinate (c) at (1, 1);
    \fill (c) circle (2pt);
    
    \coordinate (d) at (3.5, 3);
    \fill (d) circle (2pt);
    
    \coordinate (e) at (2.5, 1);
    \fill (e) circle (2pt);
    
    \coordinate (f) at (4.5, 1);
    \fill (f) circle (2pt);
    
    \coordinate (g) at (7, 3);
    \fill (g) circle (2pt);
    
    \coordinate (h) at (6, 1);
    \fill (h) circle (2pt);
    
    \coordinate (i) at (8, 1);
    \fill (i) circle (2pt);
    
    \coordinate (j) at (0, -1);
    \fill (j) circle (2pt);
    
    \coordinate (k) at (3.5, -1);
    \fill (k) circle (2pt);
    
    \coordinate (l) at (7, -1);
    \fill (l) circle (2pt);

    \draw[blue] (a) -- (b) -- (c) -- cycle;
    \draw[blue] (d) -- (e) -- (f) -- cycle;
    \draw[blue] (g) -- (h) -- (i) -- cycle;

    \draw[green] (b) -- (c) -- (j) -- cycle;
    \draw[green] (e) -- (f) -- (k) -- cycle;
    \draw[green] (h) -- (i) -- (l)--cycle;
    
    \draw[red] (j) -- (k);

    \draw[red] (k) -- (l);

    \draw[red]
    (j) to[out=-30, in=-150] (l);

    \node[above] at (a) {$u_1$};
    \node[left] at (b) {$v_4$};
    \node[right] at (c) {$v_5$};
    \node[above] at (d) {$u_2$};
    \node[left] at (e) {$v_6$};
    \node[right] at (f) {$v_7$};
    \node[above] at (g) {$u_3$};
    
    \node[left] at (h) {$v_8$};
    \node[right] at (i) {$v_9$};
    \node[below] at (j) {$v_1$};
    \node[below] at (k) {$v_2$};
    \node[below] at (l) {$v_3$};
\end{tikzpicture}
    \caption{A $K_3$-absorber for $\{u_1,u_2,u_3\}$.}
    \label{fig:absorber}
\end{figure}
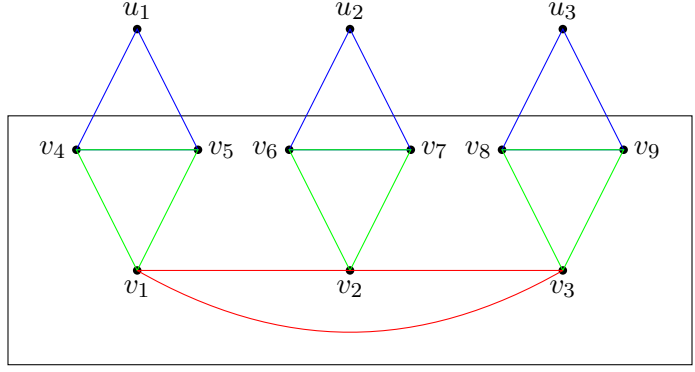
\begin{lemma}[Absorbing lemma]\label{absorbing}
    Let $0<\frac{1}{n}\ll \eps\ll \alpha\ll \xi\ll \gamma\ll \frac{1}{r}$ and $r\ge 3$. Assume that $G$ is a graph on $n$ vertices such that $\sigma(G)\geq 2(1-\frac{1}{r}-\alpha)n$. If $G$ does not contain any $\gamma$-independent set of size $\frac{n}{r}$, then $G$ contains an  $\eps$-absorbing set $M$ of size at most $2\xi n$. 
\end{lemma}
Using Lemmas \ref{almost-cover} and \ref{absorbing}, whose proofs are deferred to later subsections, we now prove Theorem~\ref{non-extremal}.
\begin{proof}[\bf Proof of Theorem \ref{non-extremal}]
    Let $0<\frac{1}{n}\ll \mu\ll\eps\ll\alpha\ll \xi\ll \gamma\ll \frac{1}{r}$. Assume that $G$ is a graph on $n$ vertices and $\sigma(G)\geq 2(1-\frac{1}{r}-\alpha)n$. By Lemma \ref{absorbing}, $G$ contains an $\eps$-absorbing set $M$ such that $|M|\leq 2\xi n$. 
    Let $G':=G-M$ and $n':=|G'|\geq (1-2\xi)n$. Therefore, 
    $$
    \sigma(G')\geq 2\Big(1-\frac{1}{r}-\alpha\Big)n-4\xi n\geq 2\Big(1-\frac{1}{r}-4\xi\Big)n'. 
    $$
    Applying Lemma \ref{almost-cover} on $G'$ yields that $G'$ has a $K_r$-tiling, say $\mathcal{K}$, that covers all but at most $\mu n'$ vertices. Denote by $U$ the set of vertices in $G'$ that are not covered by $\mathcal{K}$.  Together with the absorbing property of $M$, one has $G[M\cup U]$ contains a $K_r$-factor $\mathcal{K}'$. Hence, $\mathcal{K}\cup \mathcal{K}'$ is a $K_r$-factor of $G$, as desired. 
\end{proof}

\subsection{Regularity}
We begin by defining regularity for graphs.
\begin{definition}
  Given a graph $G$ and a pair $(X,Y)$ of vertex-disjoint subsets in $V(G)$, the \textit{density} of $(X,Y)$ is defined as 
  $$
  d(X,Y)=\frac{e(G[X,Y])}{|X||Y|}.
  $$
  Given a constant $\eps>0$, we say that $(X,Y)$ is an $\eps$-\textit{regular pair} in $G$ if for all $X'\subseteq X$, $Y'\subseteq Y$ with $|X'|\geq \eps |X|$ and $|Y'|\geq \eps |Y|$, we have 
        $$
          \left|d(X',Y')-d(X,Y)\right|\leq \eps.
        $$
Moreover, a pair $(X,Y)$ is called $(\eps,d)$-\textit{regular} for some $d>0$ if $(X,Y)$ is $\eps$-regular and $d(X,Y)\geq d$. 
\end{definition}

\begin{lemma}[Slicing lemma  \cite{Komlos2000}]\label{regular-subgraph}
    Let $0<\eps<\alpha$ and $\eps':=\max\{\eps/\alpha,2\eps\}.$ Let $(X,Y)$ be an $\eps$-regular pair of density $d$. Suppose $X'\subseteq X$ and $Y'\subseteq Y$ with $|X'|\geq \alpha |X|$ and $|Y'|\geq \alpha |Y|$. Then $(X',Y')$ is an $\eps'$-regular pair with density $d'$ where $|d'-d|<\eps$. 
\end{lemma} 

We next state the degree form of Szemer\'{e}di's Regularity Lemma, which can be easily derived from the classical version \cite{Szeme1978}.

\begin{lemma}[Degree form of the Regularity Lemma  \cite{ADLR1994,Szeme1978}]\label{thm: degree form of RL}
For every $\eps>0$, there exist integers $M$ and $n_0$ such that the following holds for any $d\in [0,1]$ and any positive integer $n\ge n_0$. Let $G$ be a graph on $n$ vertices. Then we can find in time $O(n^{2.376})$ a partition $\mathcal{P}:=\{V_0,V_1, \dots,V_k\}$ of $V(G)$ and a spanning subgraph $G'\subseteq G$ with the following properties:
\begin{enumerate}[label = (\arabic{enumi})]
\rm \item\label{pro1} $\frac{1}{\eps}\le k \le M$;
\rm \item\label{pro2} $|V_0|\le \eps n$ and $|V_1|=|V_2|=\dots=|V_k|=:m$; 
\rm \item\label{pro3} $d_{G'}(v)>d_{G}(v)-(d+2\eps)n$ for every $v\in V(G)$;
\rm \item\label{pro4} $e(G'[V_i])=0$ for each $i\in [k]$;
\rm \item\label{pro5} for $1\le i<j\le k$, each pair $(V_i,V_j)$ is $\eps$-regular in $G'$ with density {either} $0$ or at least $d$.
\end{enumerate}
\end{lemma}

The sets $V_i$ are called \textit{clusters} and  $V_0$ is the \textit{exceptional set}. 
A widely used auxiliary graph associated with the regular partition is the \textit{reduced graph} defined as follows. 
\begin{definition}[Reduced graph]
Let $G$ be an $n$-vertex graph. Given parameters $\eps>0$ and $d\in [0,1]$, let $G'$ be a spanning subgraph of $G$ and $\mathcal{P}=\{V_0,V_1, \dots,V_k\}$ be a partition of $V(G)$, as given in Lemma \ref{thm: degree form of RL}. We define the \textit{reduced graph} $R:=R(\eps,d)$ for $\mathcal{P}$ as follows: 
$R$ is a graph on vertex set $\{V_1,\ldots,V_k\}$ in which $V_i$ and $V_j$ are adjacent if and only if $(V_i,V_j)$ has density at least $d$ in $G'$.
\end{definition}


The next lemma states that clusters inherit the Ore-type condition in the reduced graph.
\begin{lemma}[\cite{Kuhn}]\label{reduced graph}
    Given a constant $c$, let $G$ be a graph on $n$ vertices such that $\sigma(G)\geq cn$. Suppose we have applied Lemma \ref{thm: degree form of RL} with parameters $\eps$ and $d$
 to $G$. Let $R$ be the corresponding reduced graph. Then $\sigma(R)\geq (c-2d-4\eps)|R|$.
\end{lemma}

The following result will be used to convert an almost perfect $K_r$-tiling in a reduced graph into an almost perfect $K_r$-tiling in the original graph $G$.
\begin{lemma}[\cite{Alon1992}]\label{reduce-original}
    Let $\eps,d>0$ and $m,r\in \mathbb{N}$ such that $0<{1}/{m}\ll \eps\ll d\ll 1/r$. Let $H$ be a graph obtained from $K_r$ by replacing every vertex with $m$ independent  vertices and every edge with an $\eps^2$-regular pair of density at least $d$. Then $H$ admits a $K_r$-tiling that covers all but at most $\eps m r$ vertices.
\end{lemma}
Our next lemma shows that the reduced graph inherits the property of having no large almost independent set. 
\begin{lemma}\label{blow-up}
    Let $r,n,t\in \mathbb{N}$ with  $0<1/n\ll \eps\ll d\ll\gamma\ll 1/r$, and 
    let $G$ be an $n$-vertex graph containing no  $\gamma$-independent set of size ${n}/{r}$. 
    Suppose we have applied Lemma \ref{thm: degree form of RL} with parameters $\eps$ and $d$ 
 to $G$, yielding a reduced graph $R$ for a partition $\mathcal{P}$ of $V(G)$. 
 Then,  $R[t]$ can be viewed as a subgraph of some reduced graph $R':=R(2\eps t,d/2)$ for a refinement partition of $\mathcal{P}$. Moreover, $R[t]$ contains no  ${\gamma}/2$-independent set of size ${tk}/{r}$. 
\end{lemma}
\begin{proof}
Let $\mathcal{P}:=\{V_0,V_1, \dots,V_k\}$ be a partition of  $V(G)$ obtained in Lemma \ref{thm: degree form of RL}. For each $i\in [k]$, partition $V_i$ into $V_{i,0}\cup V_{i,1}\cup \cdots\cup V_{i,t}$, where $m':=|V_{i,j}|=\lfloor\frac{m}{t}\rfloor\geq \frac{m}{2t}$ for each $1\leq j\leq t$. Since $mk\geq (1-\eps)n$, one has
$$
m'\cdot |R[t]|=\lfloor\frac{m}{t}\rfloor\cdot kt\geq mk-kt\geq (1-2\eps)n.
$$
Lemma \ref{regular-subgraph} implies that if $(V_{i_1},V_{i_2})_{G'}$ is $\eps$-regular with density at least $d$, then $(V_{i_1,j_1},V_{i_2,j_2})_{G'}$ is $2\eps t$-regular with density at least $d-\eps\geq d/2$ for all $1\leq j_1,j_2\leq t$. In particular, we can label the vertex set of $R[t]$ so that $V(R[t])=\{V_{i,j}:1\leq i\leq k,1\leq j\leq t\}$, where $V_{i_1,j_1}V_{i_2,j_2}\in E(R[t])$ implies that $(V_{i_1,j_1},V_{i_2,j_2})_{G'}$ is $2\eps t$-regular with density at least $d-\eps\geq d/2$. That is, $R[t]$ can be viewed as a subgraph of some reduced graph $R':=R(2\eps t,d/2)$ for the partition $\big\{(\bigcup_{i\in [k]}V_{i,0}\cup V_0), V_{1,1}, \ldots,  V_{1,t},\ldots, V_{k,1}, \ldots,  V_{k,t}\big\}$.

    Next, we prove that $R[t]$ contains no ${\gamma}/2$-independent set of size $t{k}/{r}$.  Suppose that 
    there exists a set $S:=\{V_{i_1,j_1},\ldots,V_{i_{{tk}/{r}},j_{{tk}/{r}}}\}\subseteq V(R[t])$ such that $|E((R[t])[S])|\leq \frac{\gamma}{2} (tk)^2$. Therefore, 
    $$
    |V_{i_1,j_1}\cup \cdots\cup V_{i_{{tk}/{r}},j_{{tk}/{r}}}|\geq \frac{(1-2\eps)n}{r}\ \text{and}\ |E(G[V_{i_1,j_1}\cup \cdots\cup V_{i_{{tk}/{r}},j_{{tk}/{r}}}])|\leq \frac{{\gamma}}{2} (tk)^2m'^2\leq \frac{{\gamma}}{2} n^2. 
    $$
     By augmenting $V_{i_1,j_1}\cup \cdots\cup V_{i_{{tk}/{r}},j_{{tk}/{r}}}$ with at most $2\eps n/r$ additional vertices, we obtain a $\gamma$-independent set in $G$ of size at least $\frac{n}{r}$, 
    a contradiction. 
\end{proof}
\subsection{Proof of Lemma \ref{almost-cover}}\label{section4.2}

We first prove the almost cover lemma. 

\begin{definition}\label{maximum-factor}
    A \textit{maximum $(K_r,K_{r-1},\ldots,K_1)$-factor} of a graph $G$ is a collection of vertex-disjoint subgraphs ${H_1, \ldots, H_k}$ of $G$ such that
\begin{itemize}
    \item each $H_i$ is isomorphic to some $K_s$ where $1 \leq s \leq r$; 
    \item the vertex sets of these subgraphs form a partition $V(G) = U_r \cup U_{r-1} \cup \cdots \cup U_1$, where $U_s$ consists of the vertices covered by $K_s$ components;
    \item for each $s\in \{r,r-1,\ldots,1\}$, the number of $K_s$ copies in $G\left[V(G) \setminus \left( \bigcup_{i=s+1}^r U_i \right) \right]$ is maximized.
\end{itemize}
\end{definition}

\begin{proof}[\bf Proof of Lemma \ref{almost-cover}]
   Given parameters satisfying 
   $$0<\frac{1}{n}\ll \eps\ll d\ll\mu\ll \alpha\ll \gamma\ll \frac{1}{r},
   $$
   assume that $G$ is an $n$-vertex graph with $\sigma(G)\geq 2(1-{1}/{r}-\alpha)n$ and $G$ contains no  $\gamma$-independent set of size $n/r$. By applying Lemma \ref{thm: degree form of RL} on $G$ with parameters $\eps$ and $d$, we obtain  a partition $\mathcal{P}:=\{V_0,V_1,\ldots,V_k\}$ with $k \geq \frac{1}{\eps}$ and a spanning subgraph $G'\subseteq G$ with properties $(1)$-$(5)$ as stated. Let $R$ be the reduced graph corresponding to $\mathcal{P}$. Based on Lemma \ref{reduced graph}, we have $\sigma(R)\geq 2(1-\frac{1}{r}-2\alpha)k$. Define $S:=\{v\in V(R):d_R(v)<(1-\frac{1}{r}-2\alpha)k\}$.

   Choose a maximum $(K_r,K_{r-1},\ldots,K_1)$-factor of $R$ and assume $V(R) = U_r \cup U_{r-1} \cup \cdots \cup U_1$, where $U_s$ consists of the vertices covered by $K_s$ components for each $s\in [r]$. Hence $|(V(G)\setminus U_r)\cap S|\leq r-1$. Therefore, there are at most $r-1$ components $K_t$ with $t<r$ that intersect with $S$. In the following, if $K_t$ is a clique in the maximum $(K_r,K_{r-1},\ldots,K_1)$-factor of $R$, then we simply say $K_t\in U_t$ for each $t\in [r]$.  
\begin{claim}\label{a-almost}
    $|U_r|\geq (1-\alpha r^2)k.$
\end{claim}
\begin{proof}[Proof of Claim \ref{a-almost}]
    Choose a minimum constant $\zeta$ satisfying $\zeta\geq r\alpha$ and $r|(1+\zeta)k$, let $R^*:=R\vee K_{\zeta k}$ and $k^*:=|V(R^*)|=(1+\zeta)k$. Notice that if $xy\notin E(R^*)$, then $xy\notin E(R)$. Hence, for each edge $xy\notin E(R^*)$, we have  
    $$
    d_{R^*}(x)+d_{R^*}(y)=d_{R}(x)+d_{R}(y)+2\zeta k\geq 2\Big(1-\frac{1}{r}-\alpha+\zeta\Big)k\geq 2\Big(1-\frac{1}{r}\Big)k^*,
    $$
    where the last inequality holds as $\zeta \geq r\alpha$ and $r\ge 3$. 
    Together with Theorem \ref{KK2008}, we obtain that $R^*$ contains a $K_r$-factor. It follows from $\zeta k\leq r\alpha k+r$ that $|U_r|\geq (1-\alpha  r^2)k$, as desired. 
\end{proof}
In the following claim, we consider $K_r$-tilings in the blow-up graph of $R$. 
\begin{claim}\label{reduced-blow}
    There exists an integer $z$ such that the blow-up graph $R[2^{z(r-1)}]$, has a $K_r$-tiling covering at least $(1-\mu/2)|R[2^{z(r-1)}]|$ vertices.
\end{claim}
\begin{proof}[Proof of Claim \ref{reduced-blow}]
If $|V(R)\setminus U_r|\leq \frac{\mu}{2}k$, then  we are done. So, we assume that $|U_r|:=q$ with $(1-\alpha r^2)k\leq q\leq (1-\mu/2)k$. Then there are at least $\frac{\mu}{2}k-(r-1)^2$ vertices that are  in $K_t\in U_t$ with $t<r$ and $V(K_t)\cap S=\emptyset$. 

We construct an auxiliary bipartite graph $H=(A,B)$, where each $K_r\in U_r$ corresponds to a vertex $v_{K_r}$ in $A$ and each $K_t\in U_t$ with $t<r$ and $V(K_t)\cap S=\emptyset$ corresponds to a vertex $v_{K_t}$ in $B$; $v_{K_r}v_{K_t}\in E(H)$ if and only if $|N_R(K_t)\cap V(K_r)|\geq r-t+1$. Choose a maximum matching $M$ in $H$. Hence, for each $v_{K_t}\in B\setminus V(M)$, we have $|N_R(K_t)\cap V(K_r)|\leq r-t$ for each $v_{K_r}\in A\setminus V(M)$. Furthermore, Claim~\ref{a-almost} implies that $|E(M)|\leq \alpha r^2 k$. Next, we will show that each $K_t\in U_t$ with $t<r$  and $V(K_t)\cap S=\emptyset$ can be augmented to a copy of $K_{t+1}$ in the blow-up graph $R[2]$.  

\begin{fact}\label{claim-b1}
    If $v_{K_r}v_{K_t}\in E(M)$, then $(K_r\cup K_t)[2]$ contains $2K_r\cup K_{t+1}\cup K_{t-1}$ as a subgraph. 
\end{fact}
\begin{proof}[Proof of Fact  \ref{claim-b1}]
  Notice that $|N_R(K_t)\cap V(K_r)|\geq r-t+1$.  Let $V(K_r)=\{u_1,\ldots,u_r\}$, $V(K_t)=\{w_1,\ldots,w_t\}$,  and assume $\{u_1,\ldots,u_{r-t+1}\}\subseteq N_R(K_t)\cap V(K_r)$. Consider the blow-up graph $(K_r \cup K_t)[2]$ with vertex set $\{u^1,u^2:u\in V(K_r\cup K_t)\}$. It is straightforward to verify that $(K_r \cup K_t)[2]$ contains four disjoint subgraphs induced by the following vertex sets: $\{u_1^1,\ldots,u_r^1\}$, $\{u_1^2,\ldots,u_{r-t}^2,w_1^1,\ldots,w_t^1\}$, $\{w_1^2,\ldots,w_t^2,u_{r-t+1}^2\}$, and $\{u_{r-t+2}^2,\ldots, u_r^2\}$. That is, $(K_r\cup K_t)[2]$ contains $2K_r\cup K_{t+1}\cup K_{t-1}$ as a subgraph. (see Figure \ref{fig:almost-1} for $r=3$ and $t=2$). 
\end{proof}


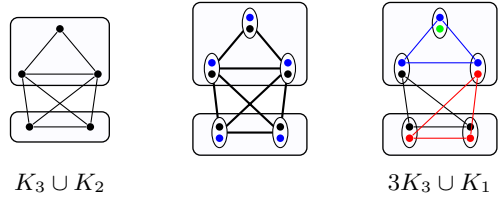
\begin{figure}[ht!]
    \centering
    \begin{tikzpicture}[
bigroundbox/.style={
        draw, 
        rectangle, 
        minimum width=1.3cm, 
        minimum height=0.9cm,
        rounded corners=1mm,  
        line width=0.5pt,       
        fill=blue!1          
    },
   roundbox/.style={
        draw, 
        rectangle, 
        minimum width=1.3cm, 
        minimum height=0.4cm,
        rounded corners=1mm,  
        line width=0.5pt,       
        fill=blue!1          
    },
    bbigroundbox/.style={
        draw, 
        rectangle, 
        minimum width=1.5cm, 
        minimum height=1.1cm,
        rounded corners=1mm,  
        line width=0.5pt,       
        fill=blue!1          
    },
    bbbigroundbox/.style={
        draw, 
        rectangle, 
        minimum width=1.5cm, 
        minimum height=0.55cm,
        rounded corners=1mm,  
        line width=0.5pt,       
        fill=blue!1          
    },
    point/.style={circle, fill=black, inner sep=1pt},
    set/.style={ellipse, draw, minimum size=0pt,  
    inner sep=0pt,     
    outer sep=0pt, minimum width=0.2cm, minimum height=0.4cm},
    label/.style={above, font=\footnotesize},
]
\node[bigroundbox] (box) at (0.5,0.3) {};
\node[roundbox] (box2) at (0.5,-0.7) {};

\node[point] (A1) at (0,0) {};
\node[point] (A2) at (1,0) {};
\node[point] (A3) at (0.5,0.6) {};
\draw (A1)--(A2)--(A3)--(A1);

\node[point] (B1) at (0.1,-0.7) {};
\node[point] (B2) at (0.9,-0.7) {};

\draw (B1)--(B2);

\draw (A1)--(B1)--(A2)--(B2)--(A1);
\node[label] at (0.5,-1.7) {{{$K_3\cup K_2$}}};
\node[bbigroundbox] (box) at (3,0.4) {};
\node[bbbigroundbox] (box2) at (3,-0.8) {};

\node[point] (A1) at (2.5,0) {};
\node[set] (cycle1) at (2.5,0.075) {};
\node[point,blue] (A01) at (2.5,0.15) {};
\node[point] (A2) at (3.5,0) {};
\node[set] (cycle2) at (3.5,0.075) {};
\node[point,blue] (A02) at (3.5,0.15) {};
\node[point] (A3) at (3,0.6) {};
\node[set] (cycle3) at (3,0.675) {};
\node[point,blue] (A03) at (3,0.75) {};
\draw[thick] (cycle1)--(cycle2)--(cycle3)--(cycle1);

\node[point] (B1) at (2.6,-0.7) {};
\node[set] (cycle4) at (2.6,-0.775) {};
\node[point,blue] (B01) at (2.6,-0.85) {};
\node[point] (B2) at (3.4,-0.7) {};
\node[set] (cycle5) at (3.4,-0.775) {};
\node[point,blue] (B02) at (3.4,-0.85) {};

\draw[thick] (cycle4)--(cycle5);

\draw[thick] (cycle1)--(cycle4)--(cycle2)--(cycle5)--(cycle1);

\node[bbigroundbox] (box) at (5.5,0.4) {};
\node[bbbigroundbox] (box2) at (5.5,-0.8) {};

\node[point] (A1) at (5,0) {};
\node[set] (cycle1) at (5,0.075) {};
\node[point,blue] (A01) at (5,0.15) {};
\node[point,red] (A2) at (6,0) {};
\node[set] (cycle2) at (6,0.075) {};
\node[point,blue] (A02) at (6,0.15) {};
\node[point,green] (A3) at (5.5,0.6) {};
\node[set] (cycle3) at (5.5,0.675) {};
\node[point,blue] (A03) at (5.5,0.75) {};

\node[point] (B1) at (5.1,-0.7) {};
\node[set] (cycle4) at (5.1,-0.775) {};
\node[point,red] (B01) at (5.1,-0.85) {};
\node[point] (B2) at (5.9,-0.7) {};
\node[set] (cycle5) at (5.9,-0.775) {};
\node[point,red] (B02) at (5.9,-0.85) {};
\draw[blue] (A01)--(A02)--(A03)--(A01);

\draw (A1)--(B1)--(B2)--(A1);
\draw[red] (A2)--(B01)--(B02)--(A2);

\node[label] at (5.5,-1.7) {{{$3K_3\cup K_1$}}};
\end{tikzpicture}
    \caption{$|N_R(K_t)\cap V(K_r)|\geq r-t+1$ for $r=3$ and $t=2$.}
    \label{fig:almost-1}
\end{figure}


Choose a clique $K_t\in U_t$ with $v_{K_t}\in B\setminus V(M)$, say $K$. 
Denote 
\begin{align*}
   &X:=\{v_{K_r}\in  A\setminus V(M): |N_R(K)\cap V(K_r)|=r-t\}\ \text{and}\ Y:=A\setminus (V(M)\cup X),
\end{align*}
assume $x:=|X|$ and $y:=|Y|$. It follows from Claim \ref{a-almost} that  $x+y\geq (\frac{1}{r}-\alpha r-\alpha r^2)k$. Moreover, 
$$
tx+(t+1)y\leq |V(R)\setminus N_R(K)|=k-|N_R(K)|\leq  t\Big(\frac{1}{r}+2\alpha\Big)k,
$$
the last inequality follows by $V(K)\cap S=\emptyset$. 
Therefore, $y\leq t(2\alpha+\alpha r+\alpha r^2)k$ and $x\geq (\frac{1}{r}-2\alpha r^3)k$.

If $|N_R(w)\cap V(K_r)|=r$ for some $w\in V(K)$ and $v_{K_r}\in X$, then a similar proof to Fact~\ref{claim-b1} shows that $(K_r\cup K)[2]$ contains $2K_r\cup K_{t+1}\cup K_{t-1}$ as a subgraph. Now, we assume that $|N_R(w)\cap V(K_r)|\leq r-1$ for each $w\in V(K)$ and $v_{K_r}\in X$. 
Choose $w\in V(K)$, denote 
\begin{align*}
&X^w:=\{v_{K_r}\in  A\setminus V(M): |N_R(w)\cap V(K_r)|=r-1\}\ \text{and}\\
&Y^w:=\{v_{K_r}\in  A\setminus V(M): |N_R(w)\cap V(K_r)|<r-1\},
\end{align*}
assume $x^w:=|X^w|$ and $y^w:=|Y^w|$.  Similarly, we have 
$$
x^w+y^w\geq \Big(\frac{1}{r}-\alpha r-\alpha r^2\Big)k\ \text{and}\ x^w+2y^w\leq \Big(\frac{1}{r}+2\alpha\Big)k.
$$
Hence $y^w\leq (2\alpha+2\alpha r+2\alpha r^2)k$ and $x^w\geq (\frac{1}{r}-4\alpha r^2)k$. Note that $X\cup (\bigcup_{u\in V(K)}X^u)\subseteq A\setminus V(M)$ and $|A|\leq \frac{|U_r|}{r}<\frac{k}{r}$. Thus,  $$\ell:=\Big|X\cap (\bigcap_{u\in V(K)}X^u)\Big|\geq \Big(\frac{1}{r}-7\alpha r^3\Big)k.$$  

For each $i\in [\ell]$, choose a vertex $w_i$ from each copy of  $K_r$, say $K^i$, with $v_{K^i}\in X\cap (\bigcap_{u\in V(K)}X^u)$ such that $w_iw\notin E(R)$. 
It follows from Lemma \ref{blow-up} that $R$ contains no $\gamma/2$-independent set of size ${k}/{r}$. Thus, there is an edge, say $w_1w_2$, in $R[\{w_1,\ldots,w_{\ell}\}]$. We claim that $w_1,w_2$ are adjacent to all vertices in $V(K)\setminus \{w\}$. Otherwise, suppose that $w_1w'\notin E(R)$ for some $w'\in V(K)\setminus \{w\}$. By the choice of $X\cap (\bigcap_{u\in V(K)}X^u)$, we know that at most $t$ edges are missing between $V(K)$ and $V(K^1)$ in $R$. However, at least two vertices in $V(K)$ are not adjacent to the vertex $w_1$. Hence,  at most $t-1$ vertices in $K^1$ has non-neighbors in $K$. It follows that $|N_R(K)\cap V(K^1)|\geq r-t+1$, a contradiction. Therefore, $(K^1\cup K^2\cup K)[2]$ contains $4K_r\cup K_{t+1}\cup K_{t-1}$, as shown by an argument similar to Fact~\ref{claim-b1}  (see Figure \ref{fig:almost-2} for $r=3$ and $t=2$). 
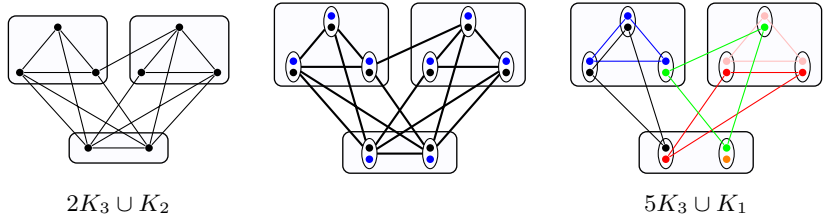
\begin{figure}[ht!]
    \centering
    \begin{tikzpicture}[
bigroundbox/.style={
        draw, 
        rectangle, 
        minimum width=1.3cm, 
        minimum height=0.9cm,
        rounded corners=1mm,  
        line width=0.5pt,       
        fill=blue!1          
    },
   roundbox/.style={
        draw, 
        rectangle, 
        minimum width=1.3cm, 
        minimum height=0.4cm,
        rounded corners=1mm,  
        line width=0.5pt,       
        fill=blue!1          
    },
    bbigroundbox/.style={
        draw, 
        rectangle, 
        minimum width=1.5cm, 
        minimum height=1.1cm,
        rounded corners=1mm,  
        line width=0.5pt,       
        fill=blue!1          
    },
    bbbigroundbox/.style={
        draw, 
        rectangle, 
        minimum width=1.5cm, 
        minimum height=0.55cm,
        rounded corners=1mm,  
        line width=0.5pt,       
        fill=blue!1          
    },
    point/.style={circle, fill=black, inner sep=1pt},
    set/.style={ellipse, draw, minimum size=0pt,  
    inner sep=0pt,     
    outer sep=0pt, minimum width=0.2cm, minimum height=0.4cm},
    label/.style={above, font=\footnotesize},
]
\node[bigroundbox] (box) at (-0.3,0.3) {};
\node[bigroundbox] (box1) at (1.3,0.3) {};
\node[roundbox] (box2) at (0.5,-1) {};

\node[point] (A1) at (-0.8,0) {};
\node[point] (A2) at (0.2,0) {};
\node[point] (A3) at (-0.3,0.6) {};
\draw (A1)--(A2)--(A3)--(A1);

\node[point] (0A1) at (0.8,0) {};
\node[point] (0A2) at (1.8,0) {};
\node[point] (0A3) at (1.3,0.6) {};
\draw (0A1)--(0A2)--(0A3)--(0A1);
\draw (A2)--(0A3);
\node[point] (B1) at (0.1,-1) {};
\node[point] (B2) at (0.9,-1) {};

\draw (B1)--(B2);

\draw (A3)--(B1)--(A1)--(B2)--(A2);
\draw (0A1)--(B1)--(0A2)--(B2)--(0A3);
\node[label] at (0.5,-2) {{{$2K_3\cup K_2$}}};
\node[bbigroundbox] (box) at (3.3,0.37) {};
\node[bbigroundbox] (box2) at (5.1,0.37) {};
\node[bbbigroundbox] (box3) at (4.2,-1.06) {};

\node[point] (A1) at (2.8,0) {};
\node[set] (cycle1) at (2.8,0.075) {};
\node[point,blue] (A01) at (2.8,0.15) {};
\node[point] (A2) at (3.8,0) {};
\node[set] (cycle2) at (3.8,0.075) {};
\node[point,blue] (A02) at (3.8,0.15) {};
\node[point] (A3) at (3.3,0.6) {};
\node[set] (cycle3) at (3.3,0.675) {};
\node[point,blue] (A03) at (3.3,0.75) {};
\draw[thick] (cycle1)--(cycle2)--(cycle3)--(cycle1);

\node[point] (0A1) at (4.6,0) {};
\node[set] (cycle01) at (4.6,0.075) {};
\node[point,blue] (0A01) at (4.6,0.15) {};
\node[point] (0A2) at (5.6,0) {};
\node[set] (cycle02) at (5.6,0.075) {};
\node[point,blue] (0A02) at (5.6,0.15) {};
\node[point] (0A3) at (5.1,0.6) {};
\node[set] (cycle03) at (5.1,0.675) {};
\node[point,blue] (0A03) at (5.1,0.75) {};
\draw[thick] (cycle01)--(cycle02)--(cycle03)--(cycle01);
\draw[thick] (cycle2)--(cycle03);

\node[point] (B1) at (3.8,-1) {};
\node[set] (cycle4) at (3.8,-1.075) {};
\node[point,blue] (B01) at (3.8,-1.15) {};
\node[point] (B2) at (4.6,-1) {};
\node[set] (cycle5) at (4.6,-1.075) {};
\node[point,blue] (B02) at (4.6,-1.15) {};

\draw[thick] (cycle4)--(cycle5);

\draw[thick] (cycle3)--(cycle4)--(cycle1)--(cycle5)--(cycle2);
\draw[thick] (cycle01)--(cycle4)--(cycle02)--(cycle5)--(cycle03);

\node[bbigroundbox] (box) at (7.2,0.37) {};
\node[bbigroundbox] (box2) at (9,0.37) {};
\node[bbbigroundbox] (box3) at (8.1,-1.06) {};

\node[point] (A1) at (6.7,0) {};
\node[set] (cycle1) at (6.7,0.075) {};
\node[point,blue] (A01) at (6.7,0.15) {};
\node[point,green] (A2) at (7.7,0) {};
\node[set] (cycle2) at (7.7,0.075) {};
\node[point,blue] (A02) at (7.7,0.15) {};
\node[point] (A3) at (7.2,0.6) {};
\node[set] (cycle3) at (7.2,0.675) {};
\node[point,blue] (A03) at (7.2,0.75) {};

\node[point,red] (0A1) at (8.5,0) {};
\node[set] (cycle01) at (8.5,0.075) {};
\node[point,pink] (0A01) at (8.5,0.15) {};
\node[point,red] (0A2) at (9.5,0) {};
\node[set] (cycle02) at (9.5,0.075) {};
\node[point,pink] (0A02) at (9.5,0.15) {};
\node[point,green] (0A3) at (9,0.6) {};
\node[set] (cycle03) at (9,0.675) {};
\node[point,pink] (0A03) at (9,0.75) {};

\node[point] (B1) at (7.7,-1) {};
\node[set] (cycle4) at (7.7,-1.075) {};
\node[point,red] (B01) at (7.7,-1.15) {};
\node[point,green] (B2) at (8.5,-1) {};
\node[set] (cycle5) at (8.5,-1.075) {};
\node[point,orange] (B02) at (8.5,-1.15) {};
\draw[blue] (A01)--(A02)--(A03)--(A01);
\draw[pink] (0A01)--(0A02)--(0A03)--(0A01);

\draw (A1)--(B1)--(A3)--(A1);
\draw[red] (0A2)--(B01)--(0A1)--(0A2);
\draw[green] (A2)--(B2)--(0A3)--(A2);

\node[label] at (8.1,-2) {{{$5K_3\cup K_1$}}};
\end{tikzpicture}
    \caption{$|N_R(K_t)\cap V(K_r)|\leq r-t$ for $r=3$ and $t=2$.}
    \label{fig:almost-2}
\end{figure}

By the above discussion, for each $K_t\in U_t$ with $t<r$ and $V(K_t)\cap S=\emptyset$, we conclude that 
\begin{itemize}
    \item either there exists a copy of $K_r\in U_r$ such that $(K_r\cup K_t)[2]$ contains $2K_r\cup K_{t+1}\cup K_{t-1}$,
    \item or there are two copies of $K_r\in U_r$ such that $(2K_r\cup K_t)[2]$ contains $4K_r\cup K_{t+1}\cup K_{t-1}$.
\end{itemize}
Moreover, we can choose the $K_r$ copies for different $K_t$ so that they are pairwise disjoint. Let $q$ be the number of vertices covered by $U_r$. Assume $|U_j|=\max\{|U_i|:i\in[r-1]\}$. Then $|U_j|\geq \frac{k-q}{2(r-1)}$. By applying the augmentation process to the cliques in $U_j$, we obtain a maximum $(K_r,K_{r-1},\ldots,K_1)$-factor of $R[2]$ such that copies of $K_i$ cover at least $2|U_i|$ vertices for all $i\in [r]\setminus \{j,j+1\}$, and copies of $K_{j+1}$ cover at least $2|U_{j+1}| + |U_j|$ vertices. (Here,  $2|U_{j+1}|$ and $|U_j|$  come from the original cliques in $U_{j+1}$  and the augmentation of cliques in $U_j$, respectively). 
By iterating this augmentation--first on the new $K_{j+1}$ copies, then on the resulting $K_{j+2}$ copies, and so forth, up to the $K_{r-1}$ copies, we ultimately construct a $K_r$-tiling within $R[2^{r-j}]$ of order at least 
$$
2^{r-j}|U_r|+2^{r-j-1}|U_{r-1}|+\cdots+|U_j|\geq 2^{r-j}q+|U_j|.
$$
Therefore, there exists a $K_r$-tiling in $R[2^{r-1}]$ covering at least 
$$
2^{r-1}q+\frac{|U_j|}{2^{r-1}}\cdot 2^{r-1}\geq q2^{r-1}+\frac{\mu k}{4(r-1)2^{r-1}}\cdot 2^{r-1}
$$
vertices, as $k-q\ge \mu k/2$. Thus, an additional  $\eta$-proportion of the vertices in $R[2^{r-1}]$ are covered, where $\eta=\frac{\mu}{4(r-1)2^{r-1}}$.

By Claim \ref{a-almost} and Lemma \ref{blow-up}, we repeat this argument until the blow-up graph of $R$ has a  $K_r$-tiling covering all but at most $\mu/2$-proportion of its vertices. Clearly, this process will terminate within $z:=\frac{1}{\eta}(\alpha r^2-\mu/2)$ iterations. 
\end{proof}
Denote $R':=R[2^{z(r-1)}]$ and $z':=2^{z(r-1)}$. For each $1\leq i\leq k$, partition $V_i$ into $V_{i,0}\cup V_{i,1}\cup \cdots\cup V_{i,z'}$, where $m':=|V_{i,j}|=\lfloor\frac{m}{z'}\rfloor\geq \frac{m}{2z'}$ for each $1\leq j\leq z'$. 
By Lemma \ref{blow-up}, we can label the vertex set of $R'$ so that $V(R')=\{V_{i,j}:1\leq i\leq k,1\leq j\leq z'\}$, where $V_{i_1,j_1}V_{i_2,j_2}\in E(R')$ implies that $(V_{i_1,j_1},V_{i_2,j_2})_{G'}$ is $2\eps  z'$-regular with density at least $d-\eps \geq d/2$.

By Claim \ref{reduced-blow}, $R'$ has a $K_r$-tiling $\mathcal{M}$ that contains at least $(1-\mu/2)|R'|$ vertices. Consider any copy of $K_r$, say $K$ in $\mathcal{M}$ and let $V(K)=\{V_{i_1,j_1},V_{i_2,j_2},\ldots, V_{i_r,j_r} \}$. Set $V':=V_{i_1,j_1}\cup V_{i_2,j_2}\cup \cdots\cup V_{i_r,j_r}$. Note that $0<{1}/{m'}\ll 2\eps  z'\ll {d}/{2}\ll \gamma \ll 1/r$. Applying Lemma \ref{reduce-original} yields that  $G'[V']$ contains a $K_r$-tiling covering all but at most $\sqrt{2\eps  z'}m'r$ vertices. By considering each copy of $K_r$ in $\mathcal{M}$ we conclude that $G'\subseteq G$ contains a $K_r$-tiling covering at least 
$$
(1-\sqrt{2\eps  z'})m'r\times (1-\mu/2)|R'|/r\geq (1-\mu)n
$$
vertices, as desired. 
\end{proof}

\subsection{Proof of Lemma \ref{absorbing}}\label{section4.3}

In this subsection, we will give the proof of Lemma \ref{absorbing}. 
\begin{proof}[\bf Proof of Lemma \ref{absorbing}]
    Given parameters satisfying 
   $$0<\frac{1}{n}\ll\eps\ll\alpha\ll\xi\ll \gamma\ll \frac{1}{r},
   $$
   assume that $G$ is an $n$-vertex graph with $\sigma(G)\geq 2(1-\frac{1}{r}-\alpha)n$ and $G$ contains no  $\gamma$-independent set of size $\frac{n}{r}$. 
   Denote $S:=\{v\in V(G):d(v)<(1-\frac{1}{r}-\alpha)n\}$. Clearly, $G[S]$ forms a clique. We aim to establish that almost all  $r$-sets in $V(G)$ have $\Omega(n^{r^2})$ $K_r$-absorbers. To this end, we proceed by considering the following two cases.

\medskip
    {\bf Case 1.} $|S|\leq \xi n$. 
\medskip

Let $\{v_1,\ldots,v_r\}$ be a set in $V(G)\setminus S$. For any set $W\subseteq V(G)\setminus S$ with size at most $r-1$, we have $|N(W)\setminus S|\geq \big(1-\frac{|W|}{r}-|W|\alpha-\xi\big)n$. 
Hence, there are at least 
$$
(n-{\xi}n-r)\Big(\big(1-\frac{1}{r}-\alpha-\xi\big)n-r-1\Big)\cdots\Big(\big(1-\frac{r-1}{r}-(r-1)\alpha-\xi\big)n-2r+1\Big)\geq \Big(\frac{n}{2r}\Big)^r
$$
$r$-sets, say $\{u_1,u_2,\ldots,u_r\}$,  in $V(G)\setminus (S\cup \{v_1,v_2,\ldots, v_r\})$ such that each spans a copy of $K_r$ in $G$. 
For every pair $(u_i,v_i)$ with $i\in [r]$, the number of $(r-1)$-sets $X\subseteq  V(G)\setminus S$ for which both $G[X\cup \{u_i\}]$ and $G[X\cup \{v_i\}]$ are copies of $K_r$ is at least 
$$
\Big(\big(1-\frac{2}{r}-2\alpha-\xi\big)n-2r\Big)\cdots\Big(\big(1-\frac{r-2}{r}-(r-2)\alpha-\xi\big)n-3r\Big)\cdot \frac{\gamma}{2}n^2\geq \Big(\frac{n}{2r}\Big)^{r-3}\frac{\gamma}{2}n^2.
$$ 
Here, the last two vertices have at least $\frac{\gamma}{2}n^2$ choices since $G$ contains no $\gamma$-independent set of size $\frac{n}{r}$. 
It follows from $|S|\leq \xi n$ that all but at most $\xi n^r$ $r$-sets 
have at least $\Omega(n^{r^2})$ $K_r$-absorbers.


\medskip
{\bf Case 2.} $|S|> \xi n$. 
\medskip

Let $\{v_1,\ldots,v_r\}$ be a set in $V(G)$. Clearly, there are at least $\binom{|S|-r}{r}$ $r$-sets, say $\{u_1,\ldots,u_r\}$,  in $S$ that span copies of $K_r$ in $G[S]-\{v_1,v_2,\ldots,v_r\}$. 
Fix an $i\in [r]$, we consider the pair $(u_i,v_i)$. 
If $|N(v_i)\cap S|\ge \frac{\xi}{4}n$, then the number of $(r-1)$-sets $X$ such that both $G[X\cup \{u_i\}]$ and $G[X\cup \{v_i\}]$ are copies of $K_r$ is at least $c\xi^{r-1}n^{r-1}$ for some constant $c$. 
If $|N(v_i)\cap S|\le \frac{\xi}{4}n$, then there are at least $\frac{3\xi}{4}n$ choices for $u_i\in S\setminus N(v_i)$. By Ore's condition, we have $d(v_i)+d(u_i)\ge 2(1-\frac{1}{r}-\alpha)n$. Hence 
$$
|(N(v_i)\cap N(u_i))\setminus S|\ge |N(v_i)\cap N(u_i)|-\frac{\xi}{4}n\geq \frac{r-2}{r}n-\frac{\xi}{3} n.
$$
Let   
\begin{align*}
&W^j:=\big((N(v_i)\cap N(u_i))\setminus S\big)\cap \big(\bigcap_{\ell\in [j-1]}N(w^{\ell})\big)\ \text{for each}\ j\in [r-2],\\
&\text{and}\ w^j\in W^j\ \text{for each}\ j\in [r-3].
\end{align*}
Similarly, for each $j\in [2,r-3]$, one has $|W^j|\geq \frac{r-2-j+1}{r}n-\frac{\xi}{3-j/r} n$. Thus, $w^j$ is well-defined for each $j\in [r-2]$. 
Together with $G$ contains no $\gamma$-independent set of size $\frac{n}{r}$, there are at least $\frac{\gamma}{2}n^2$ edges in $G[W^{r-2}]$. Thus, the number of $(r-1)$-sets $X\subseteq (N(v_i)\cap N(u_i))\setminus S$ that forms a clique  is at least $\frac{\gamma}{2\cdot r^{r-3}}n^{r-1}$. 
Therefore, for each $r$-set  in $V(G)$, there are at least $\Omega(n^{r^2})$ $K_r$-absorbers for it.

Let $\mathcal{Q}$ denote the set of $r$-set in $V(G)$ each of which has at least $\Omega(n^{r^2})$ $K_r$-absorbers. Then $\mathcal{Q}=\binom{V(G)}{r}$ if $|S|>\xi n$ and  $|\mathcal{Q}|\geq (1-o(1))n^r$ otherwise. Given an $r$-set $Q\subseteq \mathcal{Q}$, let $L_Q\subseteq \binom{V(G)}{r^2}$ denote the family of all $K_r$-absorbers for $Q$. Thus, $|L_Q|\geq \sqrt{\eps} n^{r^2}$ for each $Q\in \mathcal{Q}$. Let $F$ be the family obtained by selecting each of the $\binom{n}{r^2}$ elements of $\binom{V(G)}{r^2}$ independently with probability $p:=\frac{\sqrt{\eps}}{n^{r^2-1}}$. Then  
$$
\mathbb{E}(|F|)=p\binom{n}{r^2}<\frac{\sqrt{\eps}}{r^2!}n,\ \text{and}\ \mathbb{E}(|L_Q\cap F|)\geq p\sqrt{\eps} n^{r^2}=\eps n\ \text{for each}\ Q\in \mathcal{Q}.
$$

By Lemma \ref{chernoff}, with high probability we have
\begin{align}\label{eq:F}
    |F|\leq 2\mathbb{E}(|F|)<\frac{2\sqrt{\eps}}{r^2!}n,
\end{align}
and 
\begin{align}\label{eq:F2}
    |L_Q\cap F|\geq \frac{1}{2}\mathbb{E}(|L_Q\cap F|)\geq \frac{\eps}{2}n \ \text{for each}\ Q\in \mathcal{Q}.
\end{align}

Let $Y$ be the number of intersecting pairs of members in $F$. Then 
$$
\mathbb{E}(Y)\leq p^2\binom{n}{r^2}r^2\binom{n}{r^2-1}\leq \frac{\eps n}{(r^2-1)!(r^2-1)!}.
$$
By Markov's bound, the probability that $Y\leq \frac{\eps n}{(r^2-1)!(r^2-1)!}$ is at least $\frac{1}{2}$. Therefore, we can find a family $F$ of $r^2$-sets satisfying \eqref{eq:F} and $\eqref{eq:F2}$, and having at most $\frac{\eps n}{(r^2-1)!(r^2-1)!}$ intersecting pairs. Removing 
one set from each of the intersecting pairs in $F$, we get a family $F'$ 
such that $|F'|\leq |F|< \frac{2\sqrt{\eps}}{r^2!}n$ and for all $Q\in \mathcal{Q}$ we have
\begin{align}\label{eq:F3}
    |L_Q\cap F'|\geq \frac{\eps}{2}n-\frac{\eps n}{(r^2-1)!(r^2-1)!}>\frac{\eps}{r}n.
\end{align}

Assume that $|S|\leq \xi n$ and $X\subseteq V(G)\setminus S$ is a set with size at most $2\xi n$. By a similar discussion as Case 2, we can prove that for any vertex $w\in S$, there exist $\Omega(n^{r-1})$ copies of $K_r$ in $G-X$ with $V(K_r)\cap S=\{w\}$. Thus, for any set $T\subseteq S$ with size at most $r-1$, we can find a set $T'\subseteq V(G)\setminus (S\cup X)$ with size at most $(r-1)^2$ such that $G[T\cup T']$ contains a perfect $K_r$-tiling.

Let $M'$ denote the disjoint union of the sets in $F'$. Define $M:=M'$ if $|S|\geq \xi n$ and $M:=M'\cup S\cup S'$ otherwise, where $|S'|\leq (r-1)^2$ and $S'\subseteq V(G)\setminus (S\cup M')$ such that $G[(M'\setminus S)\cup S']$ contains a perfect $K_r$-tiling. 
Then $|M|=|F'|r^2+\xi n+(r-1)^2\leq 2\xi n.$ Since $F'$ consists of disjoint $K_r$-absorbers and each $K_r$-absorber is covered by a perfect $K_r$-tiling, $G[M]$ contains a perfect $K_r$-tiling. Now, let $W\subseteq V(G)\setminus M$ be a set of at most $\eps n$ vertices such that $|W|=r\ell$ for some $\ell\in \mathbb{N}$. We arbitrarily partition $W$ into $r$-sets $Q_1,\ldots,Q_{\ell}$. Notice that $Q_i\in \mathcal{Q}$ for each $i\in [\ell]$ as $S\subseteq M$ if $|S|\leq \xi n$. Together with \eqref{eq:F3}, we are able to find a different $K_r$-absorber from $L_{Q_i}\cap F'$ for each $Q_i$ with $i\in [\ell]$. 
Therefore, $G[M\cup W]$ contains a $K_r$-factor, as desired. 
\end{proof}

\section{Extremal case}\label{section5}
We begin by introducing some necessary parameters and definitions central to our proof. Let $r,s,n$ be three positive integers satisfying 
\begin{align}\label{eq:p1}
    r\ge 3,\ 1\le s\le r\ \text{and}\ r\mid n.
\end{align}
Define constants $\gamma, \gamma_1,\gamma_2,\dots,\gamma_r, \gamma_{r+1}$ satisfying 
\begin{align}\label{eq:p2}
    0<\frac{1}{n}\ll \gamma \ll \gamma_1\ll \dots \ll \gamma_r\ll\gamma_{r+1}=\frac{1}{r}.
\end{align}
When the constant $s$ is fixed, we define 
\begin{align}\label{eq:p3}
    \gamma_s\ll \alpha \ll\beta'\ll\beta\ll \gamma_{s+1}.
\end{align}



Let $G$ be an $n$-vertex graph with $r\mid n$. We say a vertex partition $\mathcal{P}:= \{A_1, \ldots, A_s, B\}$ is an \textit{$(r,s)$-partition} of $G$ if each $A_i$ with $i\in [s]$ has size exactly $\frac{n}{r}$. 
Given a constant $\delta>0$, we define the following three types of vertices with respect to an $(r,s)$-partition $\mathcal{P}$ of $G$. 
\begin{itemize}
    \item For a vertex $x\in A_i$ with $i\in [s]$, we say $x$ is \textit{$(\delta, A_i)$-bad} (or simply {$(\delta, i)$-bad}) if $d(x,A_i)\ge \delta n$. 
    Let $V_b(\mathcal{P},\delta,i)$ be the set of $(\delta,i)$-bad vertices.
    \item For a vertex $x\in V(G)\setminus A_i$ with $i\in [s]$, we say that  $x$ is \textit{$(\delta, A_i)$-exceptional} (or simply {$(\delta, i)$-exceptional}) if $d(x,A_i)\le \delta n$. 
    Let $V_{ex}(\mathcal{P},\delta,i)$ be the set of $(\delta,i)$-exceptional vertices. 
    \item For a vertex $x\in V(G)\setminus A_i$ with $i\in [s]$, we say that $x$ is \textit{$(\delta, A_i)$-excellent} (or simply {$(\delta, i)$-excellent}) if $d(x,A_i)\ge |A_i|-\delta n$. Let $V_e(\mathcal{P}, \delta,i)$ be the set of $(\delta ,i)$-excellent vertices, and let $V_{ne}(\mathcal{P}, \delta,i):=V(G)\setminus (A_i\cup V_e(\mathcal{P}, \delta,i))$. 
    
    \item A vertex $x\in V(G)\setminus B$ is called \textit{$(\delta,B)$-excellent} if $d(x,B)\ge |B|-\delta n$.  Denote by $V_e(\mathcal{P},\delta, B)$ the set of all such  $(\delta,B)$-excellent vertices, and let $V_{ne}(\mathcal{P},\delta, B):=V(G)\setminus (B\cup V_e(\mathcal{P}, \delta,B))$.
\end{itemize}

When the partition $\mathcal{P}$ is clear from context and  $\bullet\in \{b,ex,e,ne\}$,  we  abbreviate $V_{\bullet}(\mathcal{P},\delta,i)$ to $V_{\bullet}(\delta,i)$ when no ambiguity arises. Define  $S:=\{x\in V(G):d(x)<(1-\frac{1}{r})n-1\}$, and denote 
$$
V_{\bullet}^S(\beta,i):=V_{\bullet}(\beta,i)\cap S\ \text{and}\ V_{\bullet}^L(\beta,i):=V_{\bullet}(\beta,i)\setminus S.$$

By the Pigeonhole Principle, we obtain the following fact, which will be used frequently in the subsequent proof. 
\begin{fact}\label{ExFact}
   Let $G$ be an $n$-vertex graph with an $(r,s)$-partition $\mathcal{P}=\{A_1,\ldots,A_s,B\}$. If $d(v)> (1-\frac{2}{r}+2\delta)n$ for some vertex $v\in V(G)$, then there exists at most one $i\in[s]$ such that $d(v,A_i)\le \delta n$. 
\end{fact}

We divide the proof of Theorem \ref{Extremal case} into two parts. In Section \ref{section5.1}, we show that for any $n$-vertex graph $G$ satisfying $\sigma(G) \geq 2\left(1-\frac{1}{r}\right)n - 2$, the vertex set of $G$ admits a ``good partition'' (see Lemma \ref{Lpartition}). Then, in Section \ref{section5.2}, under such a partition, we construct a $K_r$-tiling that covers all non-excellent vertices (see Lemma \ref{K_r-tiling}). Furthermore, the remaining graph still preserves an $(r,s)$-partition. 
By using those two lemmas, we provide a proof of Theorem \ref{Extremal case} in Section \ref{section5.3}. The proofs of Lemma \ref{Lpartition} and Lemma \ref{K_r-tiling} will be given in Section \ref{section6} and Section \ref{section7}, respectively. 



\subsection{Partition of graphs}\label{section5.1}

We aim to establish a partition of the graph $G$ that exhibits certain desirable properties. For this, we introduce the following definition.
\begin{definition}[Good partition]\label{Partition}
Let $G$ be an $n$-vertex graph. Given constants $0\le \delta\ll \zeta \ll 1/r$ and a vertex subset $S\subseteq V(G)$, an $(r,s)$-partition $\mathcal{Q}=\{A_1,\ldots,A_s,B\}$ with $S\subseteq B$ is called \textit{$(r,s;\delta, \zeta;S)$-good} if there exist constants $\alpha,
\beta',\beta$ with $\delta\ll \alpha \ll \beta' \ll \beta \ll \zeta$ such that the following conditions hold.
\begin{enumerate}
[label =\rm  (A\arabic{enumi})]
    \item\label{B1}  For each $i\in [s]$, $A_i$ is a $\sqrt{\alpha}$-independent set of size $\frac{n}{r}$ in $G$.
    \item\label{B2} If $|B| \ge \frac{2n}{r}$, then $G[B]$ does not contain any $\frac{\zeta^2}{4}$-independent set of size $\frac{n}{r}$. 
    \item\label{B3} For each $i\in [s]$, $|V_b(2\beta,i)|\le \sqrt{\alpha} n$ and $ |V_{ne}(2\beta',i)|\le \sqrt{\alpha} n$.
    \item\label{B4} For each $i\in [s]$, either $V_b(2\beta,i)=\emptyset$ or $V_{ex}^L(\beta/2,i)=\emptyset$. Moreover, $G[V_{ex}^L(\beta/2,i),A_i]$ contains a matching $M_i$ covering $V_{ex}^L(\beta/2,i)$ such that $V(M_i)\cap V(M_j)=\emptyset$ for $i\neq j$ 
    (here $M_i=\emptyset$ if $V_{ex}^L(\beta/2,i)=\emptyset$). 
\end{enumerate}
\end{definition}

\begin{lemma}\label{Lpartition}
     Given the parameters defined in \eqref{eq:p1}-\eqref{eq:p3}, let $G$ be an $n$-vertex graph with
$\sigma(G)\ge 2\big(1-\frac{1}{r}\big)n-2$ 
and suppose $G$ has a $\gamma$-independent set of size $\frac{n}{r}$. Then either $G$ is the extremal graph  \ref{EX1}, or there exists an $(r,s;\gamma_s,\gamma_{s+1};S)$-good partition for some $s\in[r]$, where $S:=\big\{x\in V(G):d(x)<(1-\frac{1}{r})n-1\big\}.$
\end{lemma}

\subsection{Covering non-excellent vertices}\label{section5.2}

Note that Lemma \ref{Lpartition} yields an $(r,s;\gamma_s,\gamma_{s+1};S)$-good partition of $G$. Under such a partition, there will be some vertices that disrupt the uniformity--specifically, the non-excellent vertices. To address this issue, we introduce the following lemma to cover such vertices. 

\begin{lemma}\label{K_r-tiling}
      Given the parameters defined in \eqref{eq:p1}-\eqref{eq:p3}, let $G$ be an $n$-vertex graph with
$\sigma(G)\ge 2\left(1-\frac{1}{r}\right)n-2$ 
and suppose $G$ has a $\gamma$-independent set of size $\frac{n}{r}$. Let $\mathcal{Q}$ be an $(r,s;\gamma_s,\gamma_{s+1};S)$-good partition of $G$ for some $s\in[r]$, where $S:=\{x\in V(G):d(x)<(1-\frac{1}{r})n-1\}.$ 
Then one of the following holds. 
    \begin{enumerate}
    [label =\rm  (B\arabic{enumi})]
        \item\label{L1} $G$ is the extremal graph  \ref{EX2}.
        \item\label{L2} There exists a $K_r$-tiling $\mathcal{T}$ with at most $5r\alpha^{1/4}n$ vertices such that 
        \begin{itemize}
        \item $\bigcup_{i\in [s]}V_{ne}(2\beta',i)\cup V_{ne}(2\beta',B)\subseteq V(\mathcal{T})$;
        \item the partition $\mathcal{Q'}=        \{A_1',\ldots, A_s',B'\}$, which is obtained  from $\mathcal{Q}$ by deleting all vertices in $\mathcal{T}$, is an $(r,s)$-partition of $G-V(\mathcal{T})$;
        \item and there exists a $K_{r-s}$-factor in $G[B']$.
        \end{itemize} 
        \end{enumerate} 
\end{lemma}

\subsection{Proof of Theorem \ref{Extremal case}}\label{section5.3}

Based on Lemma \ref{Lpartition} and Lemma \ref{K_r-tiling}, we give the proof of Theorem \ref{Extremal case}. 


\begin{proof}[\bf Proof of Theorem \ref{Extremal case}]
Given the parameters defined in \eqref{eq:p1}-\eqref{eq:p3}, let $G$ be an $n$-vertex graph with $\sigma(G)\ge 2(1-\frac{1}{r})n-2$. Assume that $G$ has a $\gamma$-independent set of size $\frac{n}{r}$. By Lemma~\ref{Lpartition}, either $G$ has an independent set of size $\frac{n}{r}+1$ or $G$ has  an $(r,s;\gamma_s,\gamma_{s+1};S)$-good partition $\mathcal{Q}$ for some $s\in[r]$, where $S:=\{x\in V(G):d(x)<(1-\frac{1}{r})n-1\}$. For the former case, $G$ is the extremal graph $\ref{EX1}$. For the latter case, together with Lemma \ref{K_r-tiling}, we conclude that either $G$ is the extremal graph  \ref{EX2} or \ref{L2} holds.

Assume that there exists a $K_r$-tiling $\mathcal{T}$ satisfying \ref{L2}. Let $\mathcal{Q}'=        \{A_1',\ldots, A_s',B'\}$ be the $(r,s)$-partition obtained from $\mathcal{Q}$. Then for each $D'\in \mathcal{Q}'$ and each $v\notin D'$ we have 
\begin{align}\label{DofQ'}
d(v,D')\ge |D'|-2\beta'n -5r\alpha^{1/5}n\ge |D'|-\beta n.
\end{align}
Notice that there exist $t:=\frac{|B'|}{r-s}$ vertex-disjoint copies of $K_{r-s}$, say $K^1,\ldots,K^{t}$, in $G[B']$. Contracting each $K^i$ (where $i\in [t]$) into a vertex $w_i$ yields a new vertex set $B^*$ of size $t$. 
We are to construct an auxiliary graph $G^*$ with $V(G^*)=\big (\bigcup_{i\in[s]}A_i'\big)\cup B^*$, and $xy\in E(G^*)$ if and only if 
\begin{itemize}
    \item $x\in A_i'$, $y\in A_j'$, and $xy\in E(G)$ for $i,j\in[s]$ with $i\ne j$; or 
    \item $x=w_i\in B^*$, $y\in A_j'$, and $y\in N(V(K^i))$ for $j\in [s]$ and $i\in[t]$.
\end{itemize}
Thus, $\mathcal{Q}^*=\{A_1',\ldots,A_s',B^*\}$ is a $(s+1,s)$-partition of $G^*$. 
Then  \eqref{DofQ'} implies that  
$$
d_{G^*}(v,D^*)\geq |D^*|-(r-s)\beta n\geq \Big(1-\frac{1}{2(s+1)}\Big)|D^*| \ \text{for each}\ D^*\in \mathcal{Q}^* \ \text{and each}\ v\notin D^*.
$$

By Lemma~\ref{GHH2024}, we conclude that $G^*$ admits a  $K_{s+1}$-factor. Consequently,  $G-V(\mathcal{T})$ admits a $K_r$-factor. Combining this with $\mathcal{T}$ yields the desired $K_r$-factor in $G$. 
\end{proof}

\section{Proof of Lemma \ref{Lpartition}}\label{section6}
In this section, we prove Lemma \ref{Lpartition} by constructing the required good partition of $G$.

\begin{proof}[\bf Proof of Lemma \ref{Lpartition}]


Given the parameters defined in \eqref{eq:p1}-\eqref{eq:p3}, let $G$ be an $n$-vertex graph with
$\sigma(G)\ge 2\left(1-\frac{1}{r}\right)n-2$
and $G$ has a $\gamma$-independent set of size $\frac{n}{r}$. Denote $S:=\{x\in V(G):d(x)<(1-\frac{1}{r})n-1\}.$ 
We first show that $G$ admits an $(r,s)$-partition for some $s\in [r]$.  
\begin{claim}\label{rs-partition}
    There exists an integer $s\in [r]$ such that $G$ has an $(r,s)$-partition $\mathcal{P}:=\{A_1,\dots,A_s,B\}$ with $S\subseteq B$. Moreover, $A_i$ is a $\gamma_i$-independent set in $G$ for each  $i\in [s]$ and $G[B\setminus S]$ contains no $\gamma_{s+1}$-independent set of size $\frac{n}{r}$.
\end{claim}

\begin{proof}[Proof of Claim \ref{rs-partition}]


It follows from Fact \ref{clique} that $G[S]$ is a clique. Let $A$ be a $\gamma$-independent set of size $\frac{n}{r}$ in $G$. Then, 
$$\binom{|A\cap S|}{2}=e(G[A\cap S])\le e(G[A])\le \gamma n^2.$$
Thus, $|A\cap S|\le 2\sqrt{\gamma}n$. Let $A'$ be a subset of $V(G)\setminus (A\cup S)$ with size exactly $|A\cap S|$, and let $A_1=A'\cup (A\setminus S)$.  
Since $\gamma \ll \gamma_1$, we have 
$$e(G[A_1])\le e(G[A])+\frac{n}{r}|A'|\le (\gamma+2\sqrt{\gamma})n^2<\gamma_1 n^2,$$
that is, $A_1$ is a $\gamma_1$-independent set of size $\frac{n}{r}$ in $G-S$. 

Next, we iteratively select all possible sets $A_i$ satisfying $A_i\cap S=\emptyset$, $|A_i|=\frac{n}{r}$ and $A_i$ is a $\gamma_i$-independent set in $G-\bigcup_{j\in [i-1]}A_j$, until no such sets remain. Suppose the process terminates after $s$ steps, resulting in sets $A_1,A_2,\ldots,A_s$. We denote the remaining vertices by $B$. Clearly, $G[B\setminus S]$ contains no $\gamma_{s+1}$-independent set of size $\frac{n}{r}$, as desired. 
\end{proof}







Let $\mathcal{P}_0:=\{A_1^0,\dots,A_s^0,B^0\}$ be an $(r,s)$-partition of $G$ obtained in Claim \ref{rs-partition}. Next, we estimate the number of bad vertices and non-excellent vertices under the partition $\mathcal{P}_0$. For each $i\in [s]$, since $A_i^0$ is a $\gamma_i$-independent set in $G$ and $\gamma_i\ll \alpha \ll \beta$, we have 
$$|V_b(\mathcal{P}_0,\beta,i)|\le\frac{2e(G[A_i^0])}{\beta n} \le  \frac{2\gamma_i n^2}{\beta n}\le \alpha n.$$ Observe that at least $\binom{|A_i^0|}{2}-\gamma_i n^2$ edges are missing in  $G[A_i^0]$. By Ore's condition, the number of edges in $G[A_i^0,V(G)\setminus A_i^0]$ is at least
$$\underset{x\in A_i^0}{\sum}d(x)-2e(G[A_i^0])\ge 
\frac{1}{|A_i^0|} \cdot
e\left (\overline{G[A_i^0]}\right )\cdot\sigma (G)-2 e\left(G[A_i^0]\right) \ge |A_i^0||V(G)\setminus A_i^0|-4\gamma_i n^2.
$$
Since $\gamma_i \ll \alpha \ll \beta' \ll \beta$, for each $i\in [s]$, we have that 
$$|V_{ex}(\mathcal{P}_0,\beta,i)|\leq |V_{ne}(\mathcal{P}_0,\beta',i)|\le \frac{4\gamma_i n^2}{\beta'n}\leq  \alpha n.$$ 


In the following part, we are to construct an $(r,s;\gamma_s,\gamma_{s+1}; S)$-good partition by utilizing the following operations. Recall that   $\mathcal{P}_0=\{A_1^0,\ldots,A_s^0,B^0\}$. For each $k\in [s]$, we recursively construct an $(r,s)$-partition $\mathcal{P}_k=\{A_1^k,\ldots,A_s^k,B^k\}$ of $G$ from the previous partition $\mathcal{P}_{k-1}=\{A_1^{k-1},\ldots,A_s^{k-1},B^{k-1}\}$. 
Let 
$$
X_k:=V_{ex}^L(\mathcal{P}_{k-1},\beta+(k-1)\alpha,k),\ Y_k:=V_b(\mathcal{P}_{k-1},\beta+(k-1)\alpha,k),\ \text{and}\ H_k:=G[(A_k^{k-1}\cup X_k) \setminus Y_k].$$ Denote $t_k:=\min \{|X_k|,|Y_k|\}$. Assume  $\{x_1^k,\dots , x_{t_k}^k\}\subseteq X_k$ and $\{y_1^k,\dots, y_{t_k}^k\}\subseteq Y_k$. We now perform the following steps  on $\mathcal{P}_{k-1}$ to construct the partition  $\mathcal{P}_{k}$.
\begin{itemize}
    \item For each pair of vertices $x_j^k,\,y_j^k$ with $j\in [t_k]$, we move $x_j^k$ to $A_k^{k-1}$ and reassign $y_j^k$ to the original part that $x_j^k$ belongs to (see Figure \ref{fig:exchang}).
    \item If $|X_k|>|Y_k|$,  then we will find a matching $M_k$ of size $c_k:=|X_k|-|Y_k|$ in $H_k$. Otherwise, there exists an independent set of size $\frac{n}{r}+1$.
    \item A suitable exchange of at most $c_k$ vertices between $X_k$ and  $A_k^{k-1}$ (denote the resulting set by $A_k^k$) ensures that each edge in $M_k$ has exactly one endpoint in $A_k^k$, which yields the partition $\mathcal{P}_{k}=\{A_1^k,\ldots,A_s^k,B^k\}$. 
\end{itemize} 
The details of the last two steps and the verification that  $\mathcal{P}_s$ satisfies Definition~\ref{Partition} will be provided in the subsequent claim. 
\begin{figure}[ht!]
    \centering
    \begin{tikzpicture}[
    node distance=1.5cm,
    set/.style={ellipse, draw, minimum width=1cm, minimum height=2cm},
    smallset/.style={ellipse, draw, minimum width=0.23cm, minimum height=0.9cm},
    label/.style={above, font=\footnotesize},
    point/.style={circle, fill=black, inner sep=1pt},
    box/.style={draw, rectangle, minimum width=5.2cm, minimum height=3cm},
    verysmallset/.style={ellipse, draw, minimum size=0pt,  
    inner sep=0pt,     
    outer sep=0pt, minimum width=0.2cm, minimum height=0.5cm},
]

\node[set, fill=blue!0] (A1) at (2,0) {};
\node[label] at (2,-0.2) {};
\node[label] at (2,-1.5) {$D$};
\node[set, fill=blue!0] (A2) at (3.5,0) {};
\node[label] at (3.5,-0.2) {};
\node[label] at (3.5,-1.6) {$A_k^{k-1}$};
\node[point] (b1) at (2.3,0.5) {};
\node[label] at (2.3,0.1) {$x_j$};
\node[label] at (2.7,0.6) {$<\delta n$};
\node[point] (b2) at (3.3,0.7) {};
\node[point]  (b3) at (3.3,0.4) {};
\node[verysmallset]  (b4) at (3.3,0.55) {};
\draw (b2)--(b1)--(b3);

\node[point] (c1) at (3.3,-0.25) {};
\node[label] at (3.3,-0.67) {$y_j$};
\node[point] (c2) at (3.7,0.15) {};
\node[point] (c3) at (3.7,0) {};
\node[label] at (3.7,-0.57) {$\vdots$};
\node[point] (c4) at (3.7,-0.55) {};
\node[smallset]  (c5) at (3.7,-0.2) {};
\draw (c2)--(c1)--(c3) (c1)--(c4);
\node[label] at (4.23,-0.5) {$>\delta n$};

\node[label] at (4.9,-0.1) {$\Longrightarrow$};

\node[set, fill=blue!0] (A1) at (6,0) {};
\node[label] at (6,-0.2) {};
\node[label] at (6,-1.5) {$D$};
\node[set, fill=blue!0] (A2) at (7.5,0) {};
\node[label] at (7.5,-0.2) {};
\node[label] at (7.5,-1.6) {$A_k^{k-1}$};
\node[point] (b1) at (7.7,0.6) {};
\node[label] at (7.9,0.4) {$x_j$};
\node[point] (b2) at (7.3,0.7) {};
\node[point]  (b3) at (7.3,0.4) {};
\node[verysmallset]  (b4) at (7.3,0.55) {};
\draw (b2)--(b1)--(b3);

\node[point] (c1) at (6.3,-0.25) {};
\node[label] at (6.3,-0.67) {$y_j$};
\node[point] (c2) at (7.7,0.15) {};
\node[point] (c3) at (7.7,0) {};
\node[label] at (7.7,-0.57) {$\vdots$};
\node[point] (c4) at (7.7,-0.55) {};
\node[smallset]  (c5) at (7.7,-0.2) {};
\draw (c2)--(c1)--(c3) (c1)--(c4); 
\end{tikzpicture}

    \caption{Vertices $x_j,y_j$ exchanged in the $k$-th step, where 
$\delta=\beta+(k-1)\alpha$ and $D \in \mathcal{P}_{k-1}\setminus{A_k^{k-1}}$.}
    \label{fig:exchang}
\end{figure}
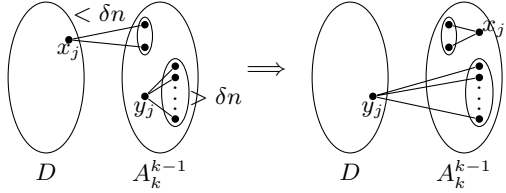



\begin{claim}\label{A2}
   Either $G$ contains an independent set of size $\frac{n}{r}+1$, or for each $k\in [s]$, there exists an $(r,s)$-partition $\mathcal{P}_k:=\{A_1^k,\ldots,A_s^k,B^k\}$ satisfying  the following.
\begin{enumerate}
[label =\rm  (C\arabic{enumi})]
    \item\label{B01} For each $i\in [s]$, $A_i^k$ is a $(2^k-1)\alpha$-independent set of size $\frac{n}{r}$ in $G$.
    \item\label{B02} If $|B^k| \ge \frac{2n}{r}$, then $G[B^k\setminus S]$ contains no $(\gamma_{s+1}-(2^k-1)\alpha)$-independent set of size $\frac{n}{r}$.
    \item\label{B03} For each $i\in [s]$, we have $$
    |V_b(\mathcal{P}_k,\beta+(2^{k}-1)\alpha, i)|\le 2^k\alpha n\ \text{and}\ |V_{ne}(\mathcal{P}_k,\beta'+(2^k-1)\alpha,i)|\le 2^k\alpha n.
    $$
    \item\label{B05} For each $i\in [k]$, there exists a matching $M_i$ of size $c_i$ in $H_i$ such that  each edge in $M_i$ has exactly one endpoint in $A_i^k$, and $V(M_i)\cap V(M_j)=\emptyset$ for all $i\neq j$. If $M_i\neq \emptyset$, then there are no $(\beta+(2^k-1)\alpha, i)$-bad vertices in $A_i^k$. 
    \item\label{B04}For each $i\in[k]$, there are no $(\beta-(2^k-1)\alpha,i)$-exceptional vertices in $V(G)\setminus (V(H_i)\cup S)$.
\end{enumerate}
\end{claim}

\begin{proof}[Proof of Claim \ref{A2}]
We prove this claim by induction on $k$. We first consider the case that $k=1$. Observe that for each $i\in [s]$, the number of vertices exchanged from $\mathcal{P}_0$ to $\mathcal{P}_1$ within $A_i^1$ (resp. $B^1$) is bounded by $|V_{ex}^L(\mathcal{P}_0,\beta,1)|\le  \alpha n$. Thus, under the partition $\mathcal{P}_1$, for each $i\in[s]$, we have 
$$
e(G[A_i^1])\leq \gamma_in^2+\alpha n \cdot \frac{n}{r}\le \alpha n^2,
$$
and for any set $B'\subseteq B^1\setminus S$ of  size $\frac{n}{r}$, we have 
$$
e(G[B'])\ge \gamma_{s+1} n^2-\alpha n\cdot \frac{n}{r}\geq (\gamma_{s+1}-\alpha)n^2.
$$
Therefore, $A_i^1$ is an $\alpha$-independent set of size $\frac{n}{r}$ for each $i\in [s]$,  and $G[B^1\setminus S]$  contains no  $(\gamma_{s+1}-\alpha)$-independent set of size $\frac{n}{r}$. Hence, \ref{B01} and \ref{B02} hold. 

For each vertex $v \in V(G)$, it is straightforward to check that for each $i\in [s]$, 
\begin{align}\label{DV}
    d(v,A_i^0)-\alpha n \le d(v,A_i^1)\le  d(v,A_i^0)+\alpha n.
\end{align}
Therefore, if $v$ is not moved in the first step and $v\notin V_{ne}(\mathcal{P}_0,\beta',i)\cup V_{b}(\mathcal{P}_0,\beta,i)$, then 
$v\notin V_{ne}(\mathcal{P}_1,\beta'+\alpha,i)\cup V_{b}(\mathcal{P}_1,\beta+\alpha,i)$. 
Thus, for each $i\in [s]$ we have 
$$|V_{ne}(\mathcal{P}_1,\beta'+\alpha, i)|\le |V_{ne}(\mathcal{P}_0,\beta',i)|+\alpha n \le 2\alpha n
\text{~and~} |V_{b}(\mathcal{P}_1,\beta+\alpha,i)|\le |V_{b}(\mathcal{P}_0,\beta,i)|+\alpha n \le 2\alpha n.$$
Hence \ref{B03} holds. 

Note that $|V(H_1)|=\frac{n}{r}+c_1$. If $c_1= 1$, then either $H_1$ has a matching $M_1$ of size $c_1$ or $V(H_1)$ is an independent set of size $\frac{n}{r}+1$. 
If $c_1\ge 2$, then since  $d_G(v)\ge (1-\frac{1}{r})n-1$ for each $v\in V(H_1)$, we have $\delta(G[H_1])\ge c_1-1$.  
Recall that $H_1=G[(A^0_1\cup X_1)\setminus Y_1]$, \eqref{DV} implies that for each $v\in V(H_1)$ we have 
$$d_{H_1}(v)\le \beta n+|V_{ex}^L(\mathcal{P}_0, \beta,1)|\le (\beta+\alpha)n,$$
 that is, $\Delta(H_1)\le(\beta+\alpha)n$. 
 Combining with Lemma~\ref{MStab}, one has that $H_1$ contains a matching $M_1$ of size $c_1$. Notice that 
the number of matching edges of $M_1$ within $A_1^0\cup \{x_1^1,\ldots,x_{t_1}^1\}$ is the same as the number of matching edges of $M_1$ within $X_1\setminus \{x_1^1,\ldots,x_{t_1}^1\}$. 
Consequently, we can perform vertex exchanges in $V(H_1)$ to ensure that under the partition $\mathcal{P}_1$, every $e \in E(M_1)$ has exactly one endpoint in $A_1^1$ (see Figure~\ref{fig:exchange-matching}). This can be achieved by exchanging at most $c_1$ vertices. 
Furthermore, if $M\ne \emptyset$, then we have $A_1^1\subseteq V(H_1)$, that is, $V_b(\mathcal{P}_1, \beta+\alpha,1)=\emptyset$.
Thus, \ref{B05} holds. 
\begin{figure}[ht]
\centering
\begin{tikzpicture}[
    node distance=1.5cm,
    set/.style={shape=ellipse, draw, minimum width=0.7cm, minimum height=1.8cm},
    smallset/.style={ellipse, draw, minimum width=0.23cm, minimum height=0.9cm},
    label/.style={above, font=\footnotesize},
    point/.style={circle, fill=black, inner sep=1pt},
    box/.style={draw, rectangle, minimum width=5.2cm, minimum height=3cm},
    verysmallset/.style={ellipse, draw, minimum size=0pt,  
    inner sep=0pt,     
    outer sep=0pt, minimum width=0.2cm, minimum height=0.5cm},
]

\node[set, fill=blue!1] (A1) at (2,0) {};
\node[label] at (2,-1.5) {$A_1^0$};
\node[point] (b1) at (2,0.3) {};
\node[label] at (2,0.28) {$u_1$};
\node[point] (b2) at (2,-0.3) {};
\node[label] at (2,-0.7) {$v_1$};
\node[set, fill=blue!1] (A2) at (3,0) {};
\node[point] (b3) at (3,0) {};
\node[label] at (3,-0.02) {$u_2$};
\node[label] at (3,-1.5) {$A_i^0$};
\node[set, fill=blue!1] (A3) at (4,0) {};
\node[label] at (4,-1.5) {$A_j^0$};
\node[point] (b4) at (4,0) {};
\node[label] at (4,-0.02) {$v_2$};
\draw (b1)--(b2) (b3)--(b4);

\node[label] at (4.8,-0.2) {$\Longrightarrow$};

\node[set, fill=blue!1] (A1) at (5.6,0) {};

\node[label] at (5.6,-1.5) {$A_1^1$};
\node[point] (b1) at (5.6,0.3) {};
\node[label] at (5.6,0.28) {$u_1$};
\node[point] (b2) at (5.6,-0.3) {};
\node[label] at (5.6,-0.7) {$v_2$};
\node[set, fill=blue!1] (A2) at (6.6,0) {};
\node[point] (b3) at (6.6,0) {};
\node[label] at (6.6,-0.4) {$u_2$};
\node[label] at (6.6,-1.5) {$A_i^1$};
\node[set, fill=blue!1] (A3) at (7.6,0) {};
\node[label] at (7.6,-1.5) {$A_j^1$};
\node[point] (b4) at (7.6,0) {};
\node[label] at (7.6,-0.02) {$v_1$};
\draw (b1)--(b4) (b3)--(b2);
\end{tikzpicture}
\caption{Here, $u_1v_1,u_2v_2\in E(M_1)$ with  $u_1,v_1\in A_1^0\cup \{x_1^1,\ldots,x_{t_1}^1\}$, $u_2,v_2\in X_1\setminus \{x_1^1,\ldots,x_{t_1}^1\}$.    
    Exchange $v_1$ and $v_2$ to guarantee that every edge in $M_1$ has exactly one endpoint in $A_1^1$.}
    \label{fig:exchange-matching}
    \end{figure}
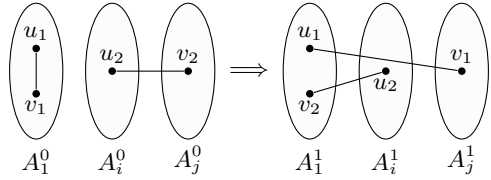

Let $v$ be a vertex in $V(G)\setminus (V(H_1)\cup S)$. If $d(v,A_1^0)\ge \beta n$, then $d(v,A_1^1)\ge (\beta-\alpha)n$, that is, there are no $(\beta-\alpha, i)$-exceptional vertices in $V(G)\setminus V(G)\setminus (V(H_i)\cup S)$ for partition $\mathcal{P}_1$. Thus, \ref{B04} holds.
Hence, \ref{B01}-\ref{B04} hold when $k=1$. Next, we assume that \ref{B01}-\ref{B04} hold for $ k-1$, and we will prove that $\mathcal{P}_k$ satisfies all of them.

By induction, for each $i\in [s]$, we know that $A_i^{k-1}$ is a $(2^{k-1}-1) \alpha$-independent set and $$\big|V_{ex}(\mathcal{P}_{k-1},\beta,i)\big|\le \big|V_{ne}^L(\mathcal{P}_{k-1},\beta+(2^{k-1}-1)\alpha, i)\big|\le 2^{k-1}\alpha n.
$$
Hence, at most $2^{k-1}\alpha n$ vertices are exchanged in the $k$-th step. Thus, we have 
\begin{align}\label{DegreeC}
    d(x,A_i^{k-1})-2^{k-1} \alpha n \le d(x,A_i^k)\le d(x,A_i^{k-1})+2^{k-1} \alpha n.
\end{align}
By an argument analogous to that for $\mathcal{P}_1$, \ref{B01}-\ref{B03} also hold for $\mathcal{P}_k$. Hence, it remains to prove \ref{B05} and \ref{B04}.

For \ref{B05}, the same method as for $k=1$ yields an independent set of size $\frac{n}{r}+1$ or a 
matching $M_k$ of size $c_k$ in $H_k$ such that $V_b(\mathcal{P}_k, \beta+(2^k-1)\alpha, k) = \emptyset$ whenever $M_k\neq \emptyset$. By induction, for each $i \in [k-1]$ with $M_i \ne \emptyset$, there are no $(\beta +(2^{k-1}-1)\alpha,i)$-bad vertices in $A_{i}^{k-1}$. Given that \( d(v) \ge \left(1 - \frac{1}{r}\right)n - 1 \) for every \( v \in V(G) \setminus S \), and by Fact~\ref{ExFact}, it follows that no vertex  in \(A_i \) is exchanged in the \( k \)-th step. Therefore, \( V_b\bigl(\mathcal{P}_k, \beta + (2^k - 1)\alpha, i\bigr) = \emptyset \), and every edge in \( M_i \) has exactly one endpoint in \( A_i^k \). 
Moreover, by induction, the inequality \eqref{DegreeC} implies that 
$$d(x_i,A^k_i)\le (\beta+(2^{k-1}-1)\alpha)n+2^{k-1}\alpha n\le (\beta+(2^k-1)\alpha) n$$ 
for each $i \in [k-1]$ and each $x_i\in V(H_i)$. For $H_k$ with any $x_k\in V(H_k)$, we have 
$$d(x_k,A^k_k)\le (\beta+(k-1)\alpha)n+2^{k-1}\alpha n\le (\beta+(2^k-1)\alpha )n.$$
Combining with Fact \ref{ExFact}, since $d(v)\ge (1-\frac{1}{r})n-1$ for any $v\in V(G)\setminus S$, we have that $V(H_i)\cap V(H_j)=\emptyset$ for all $i\ne j\in [k]$, which implies  $V(M_i)\cap V(M_j)=\emptyset$ for all $i\ne j\in [k]$. Thus, \ref{B05} holds.

By a similar argument, we have that there are no $(\beta-(2^{k}-1)\alpha,k)$-exceptional vertices in $V(G)\setminus(V(H_k)\cup S)$. For each $i\in [k-1]$, let $v\in V(G)\setminus (V(H_i)\cup S)$. By induction, we have $d(v,A_i^{k-1})\ge (\beta -(2^{k-1}-1)\alpha)n$. Thus, \eqref{DegreeC} implies that 
$$d(v,A_i^{k})\ge (\beta -(2^{k-1}-1)\alpha)n-2^{k-1}\alpha n\ge (\beta-(2^{k}-1)\alpha)n.$$ Hence, \ref{B04} holds. 
\end{proof}

By Claim \ref{A2}, either $G$ has an  independent set of size $\frac{n}{r}+1$, or it admits an $(r,s)$-good partition. In the former case, the proof is complete. We may thus assume the latter, and let $\mathcal{Q}:=\{A_1,\ldots,A_s,B\}$ be the partition $\mathcal{P}_s$ obtained by Claim \ref{A2}. We will show that $\mathcal{Q}$ is the desired good partition. Clearly, \ref{B1} follows by \ref{B01} and the fact that $(2^s-1)\alpha \le \sqrt{\alpha}$. By \ref{B05}, there exist $s$ vertex-disjoint matchings $M_1,\dots, M_s$. 
 Thus \ref{B05} and \ref{B04} implies \ref{B4}.

It is routine to check that for two constants  $\mu_1,\mu_2$ with  $\mu_1\le \mu_2$, for each $i\in [s]$ we have  
$$
|V_{ne}(\mu_1,i)|\ge |V_{ne}(\mu_2,i)|,\,|V_{b}(\mu_1,i)|\ge |V_{b}(\mu_2,i)|\ \text{and}\ |V_{ex}(\mu_2,i)|\ge |V_{ex}(\mu_1,i)|.
$$ 
Hence, \ref{B3} holds  by \ref{B03}.

Finally, we prove \ref{B2} by considering the following two cases: $|B\setminus S|\ge \frac{n}{r}$ and $|B\setminus S| < \frac{n}{r}$. We start with $|B\setminus S| \ge \frac{n}{r}$. Let $B' \subseteq B$ be an arbitrary  subset of size $\frac{n}{r}$,  
and let $C\subseteq B\setminus (S\cup B')$ be a set with size $|B'\cap S|$. If $|B'\cap S|\geq \gamma_{s+1}n$, then 
$$
e(G[B'])\geq e(G[B' \cap S])\geq \frac{\gamma_{s+1}^2}{4}n^2.
$$
If $|B'\cap S|<\gamma_{s+1}n$, then 
\begin{align*}
    e(G[B'])-e(G[(B'\setminus S)\cup C])&\geq e(G[B'\setminus S,B'\cap S])-e(G[B'\setminus S,C])\geq -|B'\setminus S||C|\geq -\frac{n}{r}\cdot \gamma_{s+1}n.
\end{align*}
Therefore, by \ref{B02} we have 
$$
e(G[B'])\geq e(G[(B'\setminus S)\cup C])-\frac{n}{r}\cdot \gamma_{s+1}n\geq (\gamma_{s+1}-(2^s-1)\alpha)n^2-\frac{n}{r}\cdot \gamma_{s+1}n 
    \geq \frac{\gamma_{s+1}^2}{4}n^2.
$$
That is, $B'$ is not a $\frac{\gamma_{s+1}^2}{4}$-independent set in $G$.



Now, we consider that $|B\setminus S|<\frac{n}{r}$. It is easy to see that if $|B\setminus S|\leq \frac{n}{2r}$, then \ref{B2} holds immediately. Hence we may assume that $|B\setminus S|> \frac{n}{2r}$ in the following. Since $|B|\ge \frac{2n}{r}$, we obtain $|S|> \frac{n}{r}$. Let $v$ be a vertex in $S$. If  $v$ is $(2\beta',i)$-excellent for all $i\in [s]$, then $d(v,B\setminus S)<2r \beta' n$, as $d(v)< (1-\frac{1}{r})n -2$. It follows that 
$$e(G[B\setminus S,S])\le d(v,B\setminus S)\cdot |S|+|B\setminus S|\cdot \Big|\bigcup_{i\in [s]}V_{ne}(2\beta',i)\Big|\leq 2r\beta' n\cdot |S|+\frac{n}{r}\cdot r  \sqrt{{\alpha}}n=2r\beta' n\cdot |S|+\sqrt{{\alpha}}n^2.$$
Thus, we obtain that
\begin{align*}
    2e(G[B\setminus S])&\geq \Big(\big(1-\frac{1}{r}\big)n-1-\frac{r-2}{r}n\Big)|B\setminus S|-e(G[S,B\setminus S])\\
    &\geq \Big(\frac{n}{r}-1\Big)\frac{n}{2r}-2r\beta'n \Big(\frac{n}{r}+\beta n+1\Big)-\sqrt{\alpha}n^2\\
    &\geq \frac{\gamma_{s+1}^2}{2}n^2.
\end{align*}
Therefore, any subset of $B$ with size $\frac{n}{r}$ is not a $\frac{\gamma_{s+1}^2}{4}$-independent set in $G$. Hence \ref{B2} holds. 
\end{proof}

\section{Proof of Lemma \ref{K_r-tiling}}\label{section7}


With a good partition of $G$ provided by Lemma~\ref{Lpartition}, our task in this section is to
construct a $K_r$-tiling that covers all non-excellent vertices under the partition. Throughout, we assume the following hypothesis:

\begin{enumerate}
[label =$(\dag)$]
    \item\label{assumption} Given the parameters defined in \eqref{eq:p1}-\eqref{eq:p3}, let $G$ be an $n$-vertex graph with
$\sigma(G)\ge 2\left(1-\frac{1}{r}\right)n-2$ 
and suppose $G$ has a $\gamma$-independent set of size $\frac{n}{r}$. Let $\mathcal{Q}:=\{A_1,\ldots,A_s,B\}$ be an $(r,s;\gamma_s,\gamma_{s+1};S)$-good partition of $G$ for some $s\in[r]$, where $S:=\{x\in V(G):d(x)<(1-\frac{1}{r})n-1\}.$
\end{enumerate}

We first define some necessary notation and provide several properties of the $(r,s;\gamma_s,\gamma_{s+1}; S)$-good partition $\mathcal{Q}$. Define 
\begin{align*}
    &V_{ex}(\beta):=\bigcup_{i=1}^{s}V_{ex}(\beta,i),\ V_{ex}^S(\beta):=V_{ex}(\beta)\cap S\ \text{and}\ V_{ex}^L(\beta):=V_{ex}(\beta)\setminus V_{ex}^S(\beta);\\
    &V_e{(2\beta')}:=\bigcap_{D\in \mathcal{Q}} (V_e(2\beta',D)\cup D),\ V_e^S{(2\beta')}:=V_e{(2\beta')}\cap S\ \text{and}\ V_e^L{(2\beta')}:=V_e{(2\beta')}\setminus V_e^S{(2\beta')}.
\end{align*}
Let $V_{ne}(2\beta'):=V(G)\setminus V_e(2\beta')$.  
We claim that 
$$
|V_{ne}(2\beta',B)|\le \alpha^{1/4} n.
$$
Otherwise, suppose that $|V_{ne}(2\beta',B)|> \alpha^{1/4} n$. By the Pigeonhole Principle, there exists some $i\in [s]$ such that there are at least $\frac{\alpha^{1/4} n}{s}$ vertices $v\in A_i$ satisfying $d(v,B)<|B|-2\beta' n$. Since $d(v)\ge (1-\frac{1}{r})n-1$ for each $v\in A_i$, we obtain that 
$$|E(G[A_i])|\geq \frac{1}{2}\cdot\frac{\alpha^{1/4} n}{s} \cdot (2\beta'n-1)\ge \sqrt{\alpha}n^2,$$ 
which contradicts \ref{B1} of Definition \ref{Partition}. Combining with \ref{B3}, we get
\begin{align}\label{size-of-ne}
    |V_{ne}(2\beta')|\leq 2\alpha^{1/4}n.
\end{align}

To prove Lemma \ref{K_r-tiling}, we employ a critical structure, called a \textit{base}, that covers all non-excellent vertices and whose extension property ensures the existence of a vertex-disjoint $K_r$-tiling.


\begin{definition}[Base]\label{Base}
Given integers $n,r,s$  with $n\gg r\ge 3$ and a positive constant $\xi$, let $G$ be an $n$-vertex graph and $\mathcal{P}$ be an $(r,s)$-partition of $G$. 
\begin{itemize}
    \item We say a clique $C$ is a \textit{$\xi$-base} in $G$ if for each  $D\in\mathcal{P}$ we have 
\begin{enumerate}
[label =\rm  (D\arabic{enumi})]
    \item\label{C01} $|V(C)\cap D|\le \frac{r|D|}{n}$,
    \item\label{C02}  $|N(V(C))\cap D|\ge |D|-(|V(C)\cap D|+1)\cdot \frac{n}{r} +\xi n$.
\end{enumerate}

\item We say a pair of vertex-disjoint cliques $C_0$ and $C_1$ is a \textit{$(2,\xi)$-base} in $G$ if there exist two  different parts $D_0, D_1\in \mathcal{P}$ satisfying the following conditions.
\begin{enumerate}
[label =\rm  (D\arabic{enumi})]
\addtocounter{enumi}{2}
    \item\label{C03} For each $i\in\mathbb{Z}_2$, we have $|V(C_i)\cap D_i|=\frac{r|D_i|}{n}+1$ and  $|V(C_i)\cap D|\le \frac{r|D|}{n}$ for each $D\in \mathcal{P}\setminus\{D_{i}\}$.
    \item\label{C04} For each $i\in \mathbb{Z}_2$, \ref{C02} holds for all $D\in \mathcal{P}\setminus \{D_{i+1}\}$, and 
    $$|N(V(C_i))\cap D_{i+1}|\ge |D_{i+1}|-(|V(C)\cap D_{i+1}|+2)\cdot \frac{n}{r}+\xi n.$$
\end{enumerate}
\end{itemize}
\end{definition}

A collection $\mathcal{H}$ of vertex-disjoint subgraphs of $G$ is called a \textit{$\xi$-base set} if for each $H \in \mathcal{H}$, there exists a parameter $\tau \geq \xi$ for which $H$ is either a $\tau$-base or a $(2, \tau)$-base. 
Here, \ref{C02} and \ref{C04} provide the fundamental criterion for extending a base structure to a copy of $K_r$, while \ref{C01} and \ref{C03} ensure that after extending the base set to a $K_r$-tiling, the partition $\mathcal{P}$ retains its $(r,s)$-partition property in the remaining graph.

In the following, we split the proof of Lemma \ref{K_r-tiling} into two parts. In Section \ref{section7.1}, we construct a $\beta/4$-base set that covers all vertices in $V_{ne}(2\beta')$. In Section \ref{section7.2}, we extend the base set to a $K_r$-tiling satisfying \ref{L2}. 

\subsection{Covering vertices in  $V_{ne}(2\beta')$}\label{section7.1}

Denote  $|V_{ex}^S(\beta/2,i)|=s_i$ for each $i\in [s]$. Let $f(x,y)$ be a function of $x$ and $y$ such that 
\begin{equation*}
f(x,y)= \left\{
\begin{array}{ll}
	1 & x=1, \\
     \left \lceil x/(y+1)\right \rceil +1 & x\ge 2 .
\end{array}\right.
\end{equation*}
Notice that $V_{ne}(2\beta')$ can be partitioned as follows:
$$
V_{ne}(2\beta')=V_{ex}(\beta/2)\cup (V_{ne}(2\beta')\setminus V_{ex}(\beta/2)). 
$$
The subsequent lemma constructs a base set that covers all vertices in $V_{ex}(\beta/2)$.


\begin{lemma}\label{Exceptional}
     Suppose that \ref{assumption} holds. 
     Then there exists a ${\beta}/{4}$-base set $\mathcal{F}$ in $G$ that covers all vertices in $V_{ex}(\beta/2)$.

{Moreover, for any $i \in [s]$ with $s_i > 0$, there exist families $\mathcal{F}_i$ and $\mathcal{F}'_i$ of vertex-disjoint subgraphs, with $V_{ex}^S(\beta/2,i)\subseteq V(\mathcal{F}_i)$, $|\mathcal{F}_i| = \lceil s_i/(r-s+1) \rceil$  and $|\mathcal{F}'_i| = f(s_i, r-s)$, such that each $F \in \mathcal{F}_i$ and $F' \in \mathcal{F}'_i$ satisfy that $F \cup F'$ is a $(2, \beta/4)$-base in $G$ with $D_0 = B$ and $D_1 = A_i$.}

\end{lemma}

\begin{proof}
Since $\mathcal{Q}:=\{A_1,\ldots,A_s,B\}$ is an $(r,s;\gamma_s,\gamma_{s+1};S)$-good partition of $G$, \ref{B1}-\ref{B4} in Definition \ref{Partition}  hold.  Let $M_i$ be the matching given in Definition~\ref{Partition}~\ref{B4}. Recall that $V_{ex}(\beta/2)=V_{ex}^L(\beta/2) \cup V_{ex}^S(\beta/2)$. We will construct two  vertex-disjoint $\beta/4$-base sets $\mathcal{F}^L$ and $\mathcal{F}^S$ covering $V_{ex}^L(\beta/2)$ and $V_{ex}^S(\beta/2)$, respectively. The subsequent  claim considers $\mathcal{F}^L$. 

\begin{claim}\label{claim-matching}
Let $M_i$ with $i\in [s]$ be the matching given in Definition \ref{Partition} \ref{B4}. Define $
\mathcal{F}^L:=\bigcup_{i\in [s]}E(M_i).
$ 
Then $\mathcal{F}^L$ is a $\beta/4$-base set that covers $V_{ex}^L(\beta/2)$. 
\end{claim}
\begin{proof}[Proof of Claim \ref{claim-matching}]
It is clear that $\mathcal{F}^L$ covers all vertices in $V_{ex}^L(\beta/2)$ and any two elements of $\mathcal{F}^L$ are vertex-disjoint. Let $uv$ be an arbitrary edge in $\mathcal{F}^L$. 
It suffices to show that $uv$ is a ${\beta}/{4}$-base in $G$. Without loss of generality, assume that $u\in V_{ex}^L(\beta/2,i)$ for some $i\in [s]$. Then, $d(u,A_i)\le \beta n/2$ and $d(v,A_i)\le 2\beta n.$

By Definition \ref{Partition} \ref{B4}, \ref{C01} holds obviously. As $u,v\notin S$, one has  
$$d(u, V(G)\setminus A_i)+d(v,V(G)\setminus A_i)\ge2\Big(1-\frac{1}{r}\Big)n-3\beta n.$$
Hence $|N(v,D)\cap N(u,D)|\ge |D|-3\beta n$ for any $D\in \mathcal{Q}\setminus\{A_i\}.$ Thus, \ref{C02} holds when  $\xi=\beta /4$.
\end{proof}




Next, we consider vertices in $V_{ex}^S(\beta/2)$, and begin by estimating the degrees of those in $V_{ne}(2\beta', i)$.
\begin{claim}\label{degree-V-S}
    If $u\in V_{ne}(2\beta',i)$ for some $i\in [s]$, then $d(u)\geq \left(1-\frac{1}{r}\right)n-2\alpha^{1/4}n.$
\end{claim}
\begin{proof}[Proof of Claim \ref{degree-V-S}]
    Since $A_i$ is a $\sqrt{\alpha}$-independent set, one has 
$$
\Big|\Big\{v\in A_i:d(v)\le \Big(1-\frac{1}{r}\Big)n+\alpha^{1/4} n\Big\}\Big|\geq \left|\left\{v\in A_i:d(v,A_i)\le \alpha^{1/4} n\right\}\right|\geq \frac{n}{r}-2\alpha^{1/4} n. 
$$
Therefore, for any $u\in V_{ne}(2\beta',i)$, there exists a vertex $v\in A_i\setminus N(u)$ such that $d(v)\le (1-\frac{1}{r})n+\alpha^{1/4} n.$ By Ore's condition, for any $u\in V_{ex}(\beta/2,i)$, one has  
\begin{align*}
    d(u)\ge 2\Big(1-\frac{1}{r}\Big)n-2-\Big(1-\frac{1}{r}\Big)n-\alpha^{1/4} n\ge \Big(1-\frac{1}{r}\Big)n-2\alpha^{1/4}n,
\end{align*} 
as desired.
\end{proof}
By Fact \ref{ExFact} and  Claim \ref{degree-V-S}, one has  $$V_{ex}^S(\beta/2,i)\cap V_{ex}^S(\beta/2,j)=\emptyset\ \text{for any distinct}\ i, j \in [s].$$
We now construct the $\beta/4$-base set $\mathcal{F}^S$, in which each base consists of two components: one for covering the vertices in $V_{ex}^S(\beta/2)$, while the other ensures that \ref{C04} is satisfied. 
The following claim describes the construction of the former one. 


\begin{claim}\label{r-scover}
 For each $i\in [s]$, there exists a family $\mathcal{F}_i$  of $\left\lceil s_i/(r-s+1)\right\rceil$ vertex-disjoint $K_{r-s+1}$ copies such that $V_{ex}^S(\beta/2,i) \subseteq V(\mathcal{F}_i)\subseteq \left(B\cap V_{e}(2\beta')\right)\cup V_{ex}^S(\beta/2,i)$. Furthermore, the families $\{\mathcal{F}_i:i\in [s]\}$ are pairwise vertex-disjoint.


    
\end{claim}
\begin{proof}[Proof of Claim \ref{r-scover}]
Let $\mathcal{F}_i=\emptyset$ if $s_i=0$ and $i\in [s]$. Next, fix an $i\in [s]$ with $s_i>0$, suppose the desired $\mathcal{F}_k$ have been constructed for all $k < i$. Note that 
$G[V_{ex}^S(\beta/2,i)]$ is a clique. Consequently, there are $\left\lfloor s_i/(r-s+1)\right\rfloor$ vertex-disjoint $K_{r-s+1}$ copies in $G[V_{ex}^S(\beta/2,i)]$. Let $T_i$ denote the set of remaining vertices in $V_{ex}^S(\beta/2,i)$ and denote $t_i:=|T_i|$, where $0\leq t_i \leq r-s$. 
If $t_i=0$, then we get the desired family  $\mathcal{F}_i$.  Next, we consider $t_i>0$. Define $B_{T_i}:=\big (N(T_i,B)\cap V_e(2\beta')\big)\setminus V(\bigcup_{k<i}\mathcal{F}_k)$.   We will find a $K_{r-s+1}\subseteq G[B_{T_i}]$ that covers $T_i$ (see Figure \ref{fig:cover-S}). 


\begin{figure}[ht!]
    \centering
    \begin{tikzpicture}[
    node distance=1.5cm,
    set/.style={ellipse, draw, minimum width=1.3cm, minimum height=2.3cm},
    smallset/.style={ellipse, draw, minimum width=1.1cm, minimum height=1.7cm},
    smallset0/.style={ellipse, draw, minimum width=1.6cm, minimum height=2.1cm},
    ssmallset/.style={ellipse, draw, minimum width=2.7cm, minimum height=1.7cm},
    sssmallset/.style={ellipse, draw, minimum width=1.1cm, minimum height=0.56cm},
    ssssmallset/.style={ellipse, draw, minimum width=0.6cm, minimum height=0.5cm},
    sssssmallset/.style={ellipse, draw, minimum width=0.8cm, minimum height=0.7cm},
    label/.style={above, font=\footnotesize},
    point/.style={circle, fill=black, inner sep=1pt},
    box/.style={draw, rectangle, minimum width=4.8cm, minimum height=3cm},
]


\node[set, fill=blue!0] (A3) at (5.2,0) {};
\node[smallset, fill=blue!5] (A4) at (5.2,-0.2) {};
\node[label] at (5.2,-0.4) {$d(v,A_i)$};
\node[label] at (5.2,-0.7) {small};
\node[label] at (5.2,-1.7) {$A_i$};
\node[point] (v) at (5.2,-0.7) {};

\node[smallset0, fill=red!20] at (7.5,0) (B) {};
\node[label] at (7.5,-1.6) {$V_{ex}^S(\beta/2,i)$};

\node[sssmallset, fill=orange!20] at (7.5,0.7) (C1) {};
\node[label] at (7.5,0.43) {$K_{r-s+1}$};
\node[sssmallset, fill=orange!20] at (7.5,0) (C2) {};
\node[label] at (7.5,0.15) {$\cdots$};
\node[label] at (7.5,-0.28) {$K_{r-s+1}$};

\node[ssssmallset, fill=orange!20] at (7.5,-0.7) (C3) {};
\node[label] at (7.5,-0.98) {$T_i$};
\draw[thick,dashed,red] (v)--(C3);
\node[label] at (6.15,-1.2) {missing};

\node[ssmallset, fill=blue!10] (A4) at (9.7,-0.5) {};
\node[box] at (8.9,-0.1) {};
\node[sssssmallset, fill=orange!30] at (9.5,-0.95) (C4) {};
\node[label] at (9.5,-1.25) {$K_{r-s}$};
\node[label] at (10.7,0.9) {$B$};
\draw[thick] (7.5,-0.45)--(9.1,0.27);
\draw[thick] (7.5,-0.95)--(9.45,-1.3);
\draw[thick] (7.5,-0.45)--(9.5,-0.60);
\node[label] at (9.7,-0.5) {$N(T_i)\cap V_{e}(2\beta')$};
\end{tikzpicture}
    \caption{Process of covering  $V_{ex}^S(\beta/2,i)$.}
    \label{fig:cover-S}
\end{figure}
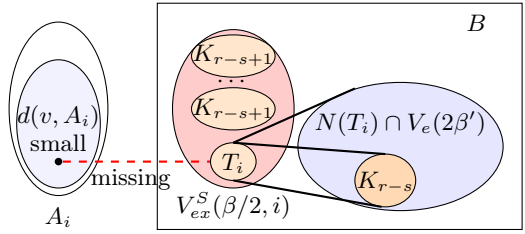

By Claim \ref{degree-V-S}, one has 
$d(u,B)\ge |B|- \beta n/2- 2\alpha^{1/4}n$ for each $u\in T_i$. 
Thus, 
\begin{align}\label{common-T}
|N(T_i,B)|\ge|B|-t_i \cdot \frac{\beta n}{2}-t_i\cdot 2\alpha^{1/4}n\ge |B|-r\cdot \beta n.
\end{align}
By Definition \ref{Partition} \ref{B3}, we get that $|B\cap V_e(2\beta')|\ge |B|-r\sqrt{\alpha} n.$ 
Then 
$$
|B_{T_i}|\ge |B|-r\beta n-r\sqrt{\alpha}n-\sum_{k<i}\Big\lceil \frac{s_k}{r-s+1}\Big\rceil\cdot (r-s+1) \ge |B|-(r+1)\beta n.
$$ 
It follows from Ore's condition that  $\sigma(G[B])\geq 2\big(1-\frac{1}{r-s}\big)|B|-2.$
Thus,
\begin{align*}
    \sigma(G[B_{T_i}])&\geq \sigma(G[B])-2(r+1)\beta n\geq 2\Big(1-\frac{1}{r-s}\Big)|B|-(2r+3)\beta n\\
    &\geq 2\Big(1-\frac{1}{r-s-1}\Big)|B_{T_i}|.
\end{align*}
Note that $G[B_{T_i}]$ contains no $\frac{\gamma_{s+1}^2}{4}$-independent set of size $\frac{n}{r}$, and hence no $\frac{\gamma_{s+1}^2}{10}$-independent set of size $\frac{|B_{T_i}|}{r-s}$. By Proposition \ref{SuperT}, we conclude that $G[B_{T_i}]$ contains $K_{r-s}$ as a subgraph. By selecting $r-s+1-t_i$ vertices from this $K_{r-s}$ and combining them with $T_i$, we obtain a $K_{r-s+1}$ that covers $T_i$.  Together with the existing $\lfloor s_i/(r-s+1)\rfloor$ copies, it  forms the desired family $\mathcal{F}_i$.
\end{proof}

Let $\mathcal{F}_i$ be the family obtained in Claim \ref{r-scover} for each $i\in [s]$. 
For all distinct $i,j\in [s]$ and each $K_{r-s+1}\in \mathcal{F}_i$, since  $V(\mathcal{F}_i)\subseteq \left(B\cap V_{e}(2\beta')\right)\cup V_{ex}^S(\beta/2,i)$,  together with Claim \ref{degree-V-S} we have 
\begin{align}\label{base-F}
    |N(V(K_{r-s+1}),A_j)|\ge |A_j|-(r-s+1)\cdot \max \bigg\{\frac{\beta}{2}n+2\alpha^{1/4}n, 2\beta' n\bigg \} \ge (1-r^2\beta)\frac{n}{r}.
\end{align}

Next, we are to construct the second component of each base in $\mathcal{F}^S$. 
\begin{claim}\label{com-2}
    For each $i\in [s]$ with $s_i>0$, there exists a family $\mathcal{F}_i'$ of vertex-disjoint subgraphs satisfying
    \begin{itemize}
        \item $|\mathcal{F}_i'|=f(s_i,r-s)$;
        \item $|V(F') \cap A_i|=2$ for each $F'\in \mathcal{F}_i'$;
        \item for any pair  $F\in \mathcal{F}_i$ and 
        $F'\in \mathcal{F}_i'$, their union $F\cup F'$ forms a $(2,\beta/4)$-base in $G$. 
    \end{itemize}
\end{claim}

\begin{proof}[Proof of Claim \ref{com-2}]
Fix an $i\in [s]$ with $s_i>0$,  suppose the desired $\mathcal{F}_k'$ have been constructed for all $k < i$. We split our proof into the following two cases. 



\medskip
\textbf{Case 1.} $|V_b(2\beta,i)|\ge f(s_i,r-s)$.
\medskip

In this case, we construct $\mathcal{F}_i'$ by utilizing vertices in  $V_{b}(2\beta,i)$. Observe that $V_{b}(2\beta, i)$ contains the following two types of vertices:
\begin{enumerate}[label =\rm  {\bf Type \arabic{enumi}}, leftmargin=6em]
    \item\label{v1} Vertices that are $(\beta/2,j)$-exceptional for some $j\in [s]\setminus \{i\}$;
    \item\label{v2} Vertices that are not $(\beta/2,j)$-exceptional for any $j\in[s].$ 
\end{enumerate}


Let $x$ be a vertex in $V_{b}(2\beta,i)$ that is of \ref{v1}, that is,  $x\in V_{ex}(\beta/2,j)$ for some $j\in [s]\setminus \{i\}$. By Definition \ref{Partition} \ref{B4}, there exists an edge $xy\in E(M_j)$ with $y\in A_j$ and  $d(x,A_j),d(y,A_j)\le 2\beta n$. Together with $x,y\notin S$, we obtain $d(\{x,y\},D)\ge |D|-4\beta n-2$ for each $D\in \mathcal{Q}\setminus \{A_j\}$. Recall that \eqref{size-of-ne} implies that at most  $4r\alpha^{1/4}n$ vertices have been used in all previously constructed bases, including those in $\mathcal{F}_k'$ with $k<s$. 
Together with $|V_b(2\beta, i)|\le \sqrt{\alpha} n$, there are at least
$$d(\{x,y\},A_i)-\big|V_{ne}(2\beta')\cup 
V_b(2\beta,i)\big|-4r\alpha^{1/4}n\ge \frac{n}{r}-5\beta n$$
vertices in $\big(V_{e}(2\beta')\setminus V_b(2\beta,i)\big)\cap N(\{x,y\},  A_i)$ that remain unused by any previous base. 
Choose $z$ to be one such vertex. 
Then for any $D\in \mathcal{Q}\setminus \{A_i, A_j,B\}$ we have  
$$d(\{x,y,z\},D) \ge |D|-4\beta n-2-2\beta'n\ge |D|-5\beta n,\ \text{and}\ d(\{x,y,z\},B) \ge |B|-4\beta n-2-\big(\frac{n}{r}+1\big).$$
Together with \eqref{base-F}, the union of $G[\{x,y,z\}]$ and any $F\in \mathcal{F}_i$ satisfy \ref{C03}-\ref{C04}, and thus forms a $(2,\beta/4)$-base in $G$.

Next, suppose $x\in V_{b}(2\beta,i)$ is of \ref{v2}. Then, $d(x,A_j)\ge \beta n/2$ for all $j\in [s]$ and $d(x,B)\ge |B|-n/r-1$. Since $|V_{ne}(2\beta')|\le 2\alpha^{1/4} n$ by \eqref{size-of-ne}, there are at least
$$d(x,A_i)-\big|V_{ne}(2\beta')\cup 
V_b(2\beta,i)\big|-4\alpha^{1/4} n\ge \frac{\beta n}{3}$$
vertices in $\big(V_{e}(2\beta')\setminus V_b(2\beta,i)\big)\cap N(x,  A_i)$ that are not used in any previous bases. 
Choose $z$ to be one such vertex. 
Then for any $D\in \mathcal{Q}\setminus \{A_i,B\}$ we have 
$$d(\{x,z\},D)\ge d(x,D)-\big(|D|-d(z,D)\big)\ge \frac{\beta n}{2}-2\beta'n\ge\frac{\beta n}{3},\ \text{and}\ d(\{x,z\},B)\geq |B|-\frac{n}{r}-3\beta'n.$$
Combining with \eqref{base-F}, the union of $G[\{x,z\}]$ and any $F\in \mathcal{F}_i$ satisfy \ref{C03}-\ref{C04}, and thus forms a $(2,\beta/4)$-base in $G$. 

As $\alpha\ll \beta$, we can construct a family  $\mathcal{F}_i'$ of $f(s_i,r-s)$ vertex-disjoint subgraphs such that for every $F'\in \mathcal{F}_i'$ and every  $F\in\mathcal{F}_i$, the  union $F\cup F'$ forms a $(2,\beta/4)$-base in $G$.  

\medskip
\textbf{Case 2.} $|V_b(2\beta,i)|< f(s_i,r-s)$.
\medskip

We claim that $G[A_i\setminus (V_b(2\beta,i)\cup V(M_i))]$ contains a matching $M^*_i$ of size $f(s_i,r-s)-|V_{b}(2\beta,i)|$. 
If so,  then for any $uv\in E(M_i^*)$ and $D\in \mathcal{Q}\setminus\{A_i\}$ we have $d(\{u,v\}, D)\ge |D|-4\beta n-2$. Therefore, \eqref{base-F} implies that the union of any edge in $M_i^*$ and any $F\in \mathcal{F}_i$ satisfy \ref{C03}-\ref{C04}, and thus forms a $(2,\beta/4)$-base in $G$. In the following, we will prove the existence of  $M_i^*$.



Notice that $e(G[V_{ex}(\beta/2,i), A_i])\le \frac{\beta}{2}n \cdot |V_{ex}(\beta/2,i)|.$ 
Therefore, 
$$\Big|\big\{v\in A_i:d(v, V_{ex}(\beta/2,i))\ge \frac{1}{9}|V_{ex}(\beta/2,i)|\big\}\Big|\leq \frac{|V_{ex}(\beta/2,i)|\beta n}{2}\cdot{\frac{1}{\frac{1}{9}|V_{ex}(\beta/2,i)|}}\le 5 \beta n.$$ 
Furthermore, by \ref{B4} one has 
$$
\Big|\{v\in A_i:d(v, V(M_i)\cap A_i) \ge \frac{8}{9}|E(M_i)|\}\Big|\leq \frac{e(G[V(M_i)\cap A_i, A_i])}{\frac{8}{9}|E(M_i)|}\leq \frac{2\beta n|E(M_i)|} {\frac{8}{9}|E(M_i)|}\le 4\beta n.
$$
Hence, there exists a set $A_i'\subseteq A_i\setminus (V_b(2\beta, i)\cup V(M_i)\cup N(V_{ex}^S(\beta/2,i)))$ of size at least $\frac{n}{2r}$ such that each vertex $v\in A_i'$ satisfying  
$$
d(v, V_{ex}(\beta/2,i))<\frac{|V_{ex}(\beta/2,i)|}{9},\ \text{and}\ 
d(v, V(M_i)\cap A_i)<\frac{8}{9}|E(M_i)|.
$$ 
Choose $v\in A_i'$. Since $v\notin N(V_{ex}^S(\beta/2,i))$, i.e., there exists a vertex $u\in S$ such that $uv\notin E(G)$, Ore's condition implies that 
$d(v)\ge (1-\frac{1}{r})n$. 
Thus, 
 \begin{align*}
     d(v, A_i)\geq \Big(1-\frac{1}{r}\Big)n-\big|V(G)\setminus (A_i\cup V_{ex}({\beta}/2,i))\big|-d(v,V_{ex}({\beta}/2,i))\ge \Big \lceil \frac{8 |V_{ex}(\beta/2,i)|}{9}\Big \rceil.
 \end{align*}
Therefore, 
\begin{align*}
    d\big(v, A_i \setminus (V_b(2\beta, i)\cup V(M_i))\big)&\ge \big(d(v, A_i)-d(v,V(M_i))\big)-|V_b(2\beta,i)|\geq \Big \lceil \frac{8 s_i}{9}\Big \rceil-|V_b(2\beta,i)|\\
    &\ge f(s_i,r-s)-|V_{b}(2\beta,i)|.
\end{align*}
Together with $|A_i'|\gg 2s_i$, there exists a matching $M_i^*$ of size $f(s_i,r-s)-|V_{b}(2\beta, i)|$ inside $G[A_i',A_i \setminus (V_b(2\beta, i)\cup V(M_i))]$. 

Note that all other existing bases can be chosen to be vertex-disjoint from all $M_i^*$.  Together with the argument in Case~1, there exists a family of $|V_b(2\beta, i)|$ vertex-disjoint subgraphs that covers $V_b(2\beta, i)$. Combining those with $M^*$ gives a family $\mathcal{F}_i'$ of size $f(s_i,r-s)$ such that for any $F'\in \mathcal{F}_i'$ and   $F\in\mathcal{F}_i$, the union $F\cup F'$ forms a $(2,\beta/4)$-base in $G$.
\end{proof}

By the above argument, we have 
$\left \lceil s_i/(r-s+1)\right\rceil =|\mathcal{F}_i|\le |\mathcal{F}_i'|=f(s_i,r-s).$ Notice that when $s_i\ge 2$ and $r-s\ge 2$, it holds that $|\mathcal{F}_i'|=|\mathcal{F}_i|+1$. Let $\mathcal{F}_i''$ be a subfamily of $\mathcal{F}_i'$ with size exactly $|\mathcal{F}_i|$. 
Thus, by combining  $\mathcal{F}_i$ and $\mathcal{F}_i''$, we obtain a  $\beta/4$-base set $\mathcal{F}_i^S$ of size $\left \lceil s_i/(r-s+1)\right\rceil$. 
Define $\mathcal{F}^S:=\bigcup_{i\in [s]}\mathcal{F}^S_i$. Clearly,  $\mathcal{F}^S$ is a $(2, \beta/4)$-base set that covers  $V_{ex}^S(\beta/2)$. 

Let $\mathcal{F}^L_*$ be a maximal subset of $\mathcal{F}^L$ where every graph is vertex-disjoint from all graphs in $\mathcal{F}^S$.  
Hence, $\mathcal{F}:=\mathcal{F}^L_*\cup \mathcal{F}^S$ is the desired $\beta/4$-base set. 
\end{proof}

Lemma \ref{Exceptional} guarantees the existence of a $\beta/4$-base set $\mathcal{F}$ covering $V_{ex}(\beta/2)$. However, vertices in $V_{ne}(2\beta')\setminus V_{ex}(\beta/2)$ may still fail to satisfy the conditions of Lemma~\ref{GHH2024}. We  
next construct a base set covering $V_{ne}(2\beta')\setminus V_{ex}(\beta/2)$. 


\begin{lemma}\label{Vne}
Suppose that \ref{assumption} holds, and let  $U\subseteq V(G)$ be a vertex subset with $|U|\le 4r\alpha^{1/4} n$. 
Then there exists a $\beta/4$-base set that covers $V_{ne}(2\beta')\setminus V_{ex}(\beta/2)$ and avoids $U$. Moreover, every vertex in this base set that does not belong to $B$ has degree at most $ (1-2\beta')n$. 


\end{lemma}

\begin{proof}
Choose $v\in V_{ne}(2\beta')\setminus V_{ex}(\beta/2)$. If $s=r$, then $\delta(G)\geq (1-\frac{1}{r})n-1$ and one may assume $v\in A_i$ for some $i\in [r]$. Thus, $d(v,
A_j)\ge \beta n/2$ for every $j\in [n]\setminus \{i\}$. Therefore, the vertex $v$ itself is a $\beta/4$-base under the partition $\mathcal{Q}$. In what follows, we only consider the case that $s<r$. 

Recall that 
$$
v\in V_{ne}(2\beta')=V(G)\setminus V_{e}(2\beta')\ \text{and}\ V_{e}(2\beta')=\bigcap_{D\in \mathcal{Q}}(V_e(2\beta',D)\cup D).
$$
Hence there exists a set $D\in \mathcal{Q}$ such that $v\in V(G)\setminus (V_{e}(2\beta',D)\cup D)$. We split the proof into the following two cases. 

\medskip
\textbf{Case 1.} $v\notin B$. 
\medskip

Since $v\notin B$, without loss of generality, assume that $v\in A_i$ for some $i\in [s]$. Let $D$ be a set in $\mathcal{Q}$ such that $v\in V(G)\setminus (V_{e}(2\beta',D)\cup D)$. 
Since $v\notin V_e(2\beta',D)$, we obtain 
$$
d(v,A_i)\geq \Big(1-\frac{1}{r}\Big)n-1-|V(G)\setminus (A_i\cup D)|-d(v,D)\geq |D|-1-(|D|-2\beta'n)=
2\beta' n-1.
$$
As $|V_{ne}(2\beta')| \le  2\alpha^{1/4}n$ and $|U|\le 4r\alpha^{1/4} n,$ we obtain 
$$\big|(N(v, A_i)\cap V_e(2\beta' ))\setminus U\big|=d(v,A_i)-|V_{ne}(2\beta')|-|U|\ge \beta' n.$$
Since $G[A_i]$ is $\sqrt
{\alpha}$-independent, there are at most $2r\sqrt{\alpha} n (\ll \beta'n)$ vertices $u\in A_i$ such that $d(u)\ge n-2\beta' n$.
Choose $w\in (N(v, A_i)\cap V_e(2\beta' ))\setminus U$ such that $d(w)\le n- 2\beta' n$. 
Since $v\notin V_{ex}(\beta/2)$ and $w\in V_e(2\beta')$, for each $D'\in \mathcal{Q}\setminus \{A_i,B\}$, we deduce  
\begin{align}\label{eq:uv-co}
    |N(\{v,w\},D')|\ge \frac{\beta n}{2}-2\beta'n\geq\frac{\beta n}{4}\ \text{and}\ |N(\{v,w\},B)|\ge |B|-\frac{n}{r}-3\beta'n.
\end{align} 

Let $B':=B\setminus (V_{ne}(2\beta')\cup U)$. 
Then $$|B'|\ge \frac{r-s}{r}n-|V_{ne}(2\beta')|-|U|\ge \frac{r-s}{r}n-5r\alpha^{1/4} n.$$
Based on $\sigma(G)\ge2(1- \frac{1}{r})n-2$, we get 
$$\sigma(G[B'])\ge \Big(1-\frac{1}{r-s}-\alpha^{1/5}\Big)|B'|.$$
Since $G[B]$ contains  no $\gamma_{s+1}^2/4$-independent set of size $\frac{n}{r}$, Proposition \ref{SuperT} implies that there exists a copy of $K_{r-s+1}$ in $G[B']$. 
It is straightforward to check that such a $K_{r-s+1}$ together with the edge $vw$ satisfy \ref{C03} and \ref{C04} with $\xi =\beta /4$, and thus forms a $(2,\beta/4)$-base in $G$, as desired. 

\medskip
\textbf{ Case 2.} $v\in B$. 
\medskip 


In this case, there exists an $i\in [s]$ such that  $v\in V(G)\setminus (V_{e}(2\beta',A_i)\cup A_i)$. It follows from Claim \ref{degree-V-S} that 
$$d(v,B)\ge \frac{r-s-1}{r}n+2(\beta'-\alpha^{1/4})n.$$ 
Let $B':= N(v, B)\setminus (V_{ne}(2\beta')\cup U)$. It is straightforward to check that $|B'|>\frac{r-s-1}{r}n$ and $\sigma(G[B'])\ge 2(1-\frac{1}{r-s-1})|B'|-2$. Recall that $G[B]$ contains  no $\gamma_{s+1}^2/4$-independent set of size $\frac{n}{r}$. 
Thus, Proposition \ref{SuperT} implies that $G[B']$ contains a copy of $K_{r-s-1}$. Therefore, there exists a copy of $K_{r-s}$ containing $v$, which forms a $\beta/4$-base under  the partition $\mathcal{Q}$, as desired. 

\medskip
As $\alpha\ll \beta'$ and $|V_{ne}(2\beta')|\leq 2\alpha^{1/4}n$, by applying the above argument iteratively, we deduce that for each vertex $v\in V_{ne}(2\beta')$, there is a base covering $v$ that avoids all vertices used in any previously constructed bases.
\end{proof}

\subsection{Proof of Lemma \ref{K_r-tiling}}\label{section7.2}

In this subsection, we will give the proof of Lemma \ref{K_r-tiling}. First, we show that the bases constructed above can be extended to some vertex-disjoint $K_r$ and the rest partition still has some good properties. 

\begin{lemma}\label{extend}
Suppose that \ref{assumption} holds. Let $H$ be a $\beta/5$-base or $(2,\beta/5)$-base in $G$, and $W$ be a vertex subset satisfying $|W| \leq 5r\alpha^{1/4}n$. Then, there exists a $K_r$-tiling ${T}_H$ in $G-W$ that covers $V(H)$, consisting of at most two vertex-disjoint cliques and  satisfying 
\[
|T_H \cap V(D)| = \frac{kr|D|}{n}\ \text{for every}\ D\in \mathcal{Q}.
\]
\end{lemma}
\begin{proof}

We first consider a $\beta/5$-base $H\in \mathcal{H}$. Let $t_H:=r-s-|V(H)\cap B|$. By \ref{C02} one has $|N(V(H),B)|\ge (t_H-1)\frac{n}{r}+\frac{\beta n}{5}$. Let $G':=G[(N(V(H), B)\setminus W)\cap V_{e}(2\beta')]$. Clearly, 
$$|V(G')|\ge (t_H-1)\frac{n}{r}+\frac{\beta n}{5}-5r\alpha^{1/4} n-2\alpha^{1/4} n\ge (t_h-1)\frac{n}{r}+\frac{\beta n}{8}.$$ 
Since $\sigma(G)\ge 2(1-\frac{1}{r})n-2$, one has 
$$\sigma(G')\ge 2|V(G')|-\frac{2n}{r}-2\ge 2\Big(1-\frac{1}{t_H-1}\Big)|V(G')|.$$
By \ref{B2}, $G'$ contains no $\frac{\gamma_{s+1}^2}{4}$-independent set of size $\frac{n}{r}$. 
Proposition \ref{SuperT} implies that there exists a $K_{t_H}$ in $G'$. Define  $H'$ as the join of this $K_{t_H}$ and $H$. Since $H$ is a $\beta/4$-base, for every $A_i\in\mathcal{Q}$ with $|A_i\cap V(H)|=0$ we have
$$|(N(H',A_i)\setminus W)\cap V_e(2\beta')|\ge \frac{\beta n}{5}-r \cdot 2\beta' n-5r\alpha^{1/4} n-2\alpha^{1/4} n\ge \frac{\beta n}{8}.$$
This enables us to iteratively select vertices $x_i\in \big(N(H',A_i)\cap V_e(2\beta')\big)\setminus W$ for each $i\in [s]$ with $|A_i\cap V(H)|=0$. Therefore, $G[V(H')\cup \{x_i:|A_i\cap V(H)|=0\ \text{and}\ i\in [s]\}]$ is a clique $K_r$, as desired. 

Next, we assume that $H$ is a $(2,\beta/5)$-base in $\mathcal{H}$. 
Without loss of generality, assume that $C_0$ and $C_1$ are two vertex-disjoint cliques in $H$ and $D_0, D_1 \in \mathcal{Q}$ are two parts satisfying $|N(C_i)\cap D_i|=\frac{r|D_i|}{n}+1$. 
For $C_0$, we will find a vertex set $T_1\subseteq D_1\setminus W$ of size $t_1:=\frac{r|D_1|}{n}-|V(C_0)\cap D_1|-1$ such that $G[V(C_0)\cup T_1]$ is a clique. 
If $t_1\le0$, then it is trivial. Thus, we only need to consider $D_1=B$ and $r-s\ge 2$. Let $G':=G[(N(C_0,B)\cap V_{e}^S(2\beta'))\setminus W]$. 
By applying a similar argument as the case that $H$ is a $\beta/5$-base, we conclude that 
$G'$ contains $K_{t_1}$ as a subgraph and there exists a copy of $K_r$ covering $K_{t_1}\cup C_0$ that avoids $W$ and all used vertices. 

By the same argument, there exists a copy of $K_r$ covering $C_1$ that avoids both $W$ and all previously used vertices. Moreover, the union of these two $K_r$ copies uses two distinct vertices from each $A_i\setminus W$ for each $i\in [s]$ and $2(r-s)$ vertices from $B\setminus W$, as required. 
\end{proof}


Now, we are ready to prove Lemma \ref{K_r-tiling}. 
\begin{proof}[\bf Proof of Lemma \ref{K_r-tiling}] 
Recall that $|V_{ne}(2\beta')|\le 2\alpha^{1/4} n$. By Lemma \ref{Exceptional} and Lemma~\ref{Vne}, we obtain a $\beta/4$-base set $\mathcal{H}$ covering $V_{ne}(2\beta')$ with at most $12\alpha^{1/4}n$ vertices. Let $W:=V(\mathcal{H})$. We then proceed iteratively: at each step, we apply Lemma~\ref{extend} to extend a base to a \( K_r \)-tiling, and then update \( W \) by adding the vertices of this new \( K_r \)-tiling. Therefore, we obtain a \( K_r \)-tiling \( \mathcal{T} \) of order at most \( 4r \alpha^{1/4} n \) that covers \( V_{\mathrm{ne}}(2\beta') \).  
Let $G':=G-V(\mathcal{T})$, $A_i':=A_i\setminus V(\mathcal{T})$ for each $i\in [s]$,  and  $B':=B\setminus V(\mathcal{T})$. 
By Lemma~\ref{extend}, one has that  $\{A_1',\ldots,A_s',B'\}$ is an $(r,s)$-partition of $G'$ and $(r-s)\mid |B'|$. 
Furthermore,  
\begin{align}\label{Ore:G'}
\sigma(G')\ge 2\Big(1- \frac{1}{r}\Big)n-8r\alpha^{1/4}n-2 \ \text{and} \ \sigma(G[B'])\ge 2\Big(1-\frac{1}{r-s}-\alpha^{1/5}\Big)|B'|.
\end{align} 

Observe that $G[B']$ contains no $\frac{\gamma_{s+1}^2}{8}$-independent set of size $\frac{n}{r}$, and hence no $\frac{\gamma_{s+1}^2}{10}$-independent set of size $\frac{|B'|}{r-s}$. Clearly, \ref{L2} is trivial for $r-s=1$. If $r-s\geq 3$, then by Theorem \ref{non-extremal} we obtain that there exists a $K_{r-s}$-tiling in $G[B']$, i.e., \ref{L2} holds. 
In what follows, it suffices to consider $r-s=2$. We will show that either $G$ is the extremal graph \ref{EX2}, or $G[B']$ contains a perfect matching, or by modifying some bases in $\mathcal{H}$ we can get a new $K_r$-tiling $\mathcal{T}'$ such that $G[B\setminus V(\mathcal{T}')]$ admits a perfect matching. 

If $G[B']$ contains a perfect matching, then we are done. Suppose that $G[B']$ contains no perfect matching. By Proposition \ref{PMatching}, we know that $G[B']$ consists of two odd components, say $F_0$ and $F_1$.

Suppose that $V_{ne}(2\beta')= \emptyset$. Then   $\mathcal{H}=\emptyset$ and $B'=B$. Without loss of generality, assume that  $|F_0|\le |F_1|$. 
By Ore's condition, we conclude that 
$N(z)=V(G)\setminus (V(F_{i+1})\cup \{z\})\ \text{for each}\ z\in V(F_i)\ \text{with}\ i\in \mathbb{Z}_2.$ Hence 
$F_0, F_1$ are two cliques with odd orders. 
If $G[A_i]=\emptyset$ for all $i\in [r-2]$, then $G$ is the extremal graph described in \ref{EX2}. Suppose that there exists an edge $uv\in E(G[A_i])$ for some $i\in [r-2]$. Then there exists a vertex  $w\in V(F_0)$ such that $G[\{u,v,w\}]$ forms a $K_3$, say $K^1$. Since $F_1$ is a clique and $|F_1|\ge |B'|/2$, there exists a $K_3$, say $K^2$, in $F_1$. Since $V_{ne}(2\beta')=\emptyset$, it is easy to check that $K^1\cup K^2$ is a $(2,\beta/4)$-base in $G$. By applying Lemma~\ref{extend} with $W=\emptyset$, we obtain a $K_r$-tiling $\mathcal{T}$ of $G$ such that  $|V(F_i)\setminus V(\mathcal{T})|$ is even for every $i\in \mathbb{Z}_2$. Applying Proposition~\ref{PMatching} again yields that $G[B\setminus V(\mathcal{T})]$ contains a perfect matching, as desired. Thus, overall, it suffices to consider the case where 
$$
r-s=2,\ G[B']=F_0\cup F_1\ \text{with $|F_0|$ and $|F_1|$ both  odd} ,\ \text{and}\ V_{ne}(2\beta')\neq \emptyset.
$$

Our goal is to  modify the $K_r$-tiling $\mathcal{T}$ into a new $K_{r}$-tiling $\mathcal{T}'$ that satisfies \ref{L2}. We begin with an analysis of the properties of $F_0$ and $F_1$.  
Since $S$ is a clique in $G$, at most one of the components $F_i$ with $i\in\mathbb{Z}_2$ satisfies $V(F_i)\cap S\neq \emptyset.$  Without loss of generality, assume that $V(F_1)\cap S=\emptyset$. Therefore, for a vertex $v\in V(F_1)$, one has 
$$
\Big(1-\frac{1}{r}\Big)n-1\leq d(v)\leq n-|B'|+d(v,F_1),
$$
which implies 
$$
|F_1|\geq d(v,F_1)\geq \frac{n}{r}-1-4r\alpha^{1/4}n\geq \frac{n}{r}-\alpha^{1/5}n. 
$$
Furthermore, if $|F_1|\ge \frac{n}{r}+1$, then $d(u)\leq (1-\frac{1}{r})n-2$ for every $u\in V(F_0)$. That is,  $V(F_0)\subseteq S$.  Similarly, if $V(F_0)\setminus S\neq \emptyset$, then $|F_0|\geq \frac{n}{r}-\alpha^{1/5}n.$ In particular, if $|F_0|\leq \frac{n}{r}-3\beta n$, then by Ore's condition one has 
\begin{align}\label{complete}
G[V_{ex}^L(\beta/2)\cup (V(G)\setminus (V_b(2\beta)\cup B)),V(F_0)]\ \text{is a complete bipartite graph}.
\end{align}




For any $\Tilde{\mathcal{H}}\subseteq \mathcal{H}$, let $\Tilde{\mathcal{T}}\subseteq \mathcal{T}$ be the set of $K_r$ copies obtained by extending the bases in $\widetilde{\mathcal{H}}$. Let $X_{\Tilde{\mathcal{H}}} :=\big(B\cap V(\Tilde{\mathcal{T}})\big)\setminus V_{ex}^S(\beta/2)$. It follows from \eqref{size-of-ne} that $|X_{\Tilde{\mathcal{H}}}|\le 8\alpha^{1/4}n$. 
If $|F_1|\ge \frac{3n}{2r}$, 
then define 
\begin{align}\label{eq:F_0}
F_0'(\Tilde{H}):=G\big[V(F_0)\cup (X_{\Tilde{\mathcal{H}}}\cap S)\big] \ \text{and} \ F_1'(\Tilde{H}):=G\big [(V(F_1)\cup X_{\Tilde{\mathcal{H}}})\setminus S\big];
\end{align}
if $|F_1|< \frac{3n}{2r}$, then define  
\begin{align}\label{eq:F_1}
F_1'(\Tilde{H}):=G\big[V(F_1)\cup \{x\in X_{\Tilde{\mathcal{H}}} : d(x,F_1)\ge \frac{2}{3}|F_1|\}\big] \ \text{and} \ F_0'(\Tilde{H}):= G\big[(V(F_0)\cup X_{\Tilde{\mathcal{H}}})\setminus V(F_1')\big].
\end{align}
For brevity, when $\widetilde{H}$ is clear from context, we write $F_0'$ and $F_1'$ for $F_0'(\widetilde{H})$ and $F_1'(\widetilde{H})$, respectively.
Note that if $V_{ex}^S(\beta/2) \cap V(\Tilde{\mathcal{H}})=\emptyset$, then 
\begin{align}\label{even-order}
|V(F_0\cup F_1)|\ \text{and}\ 
|V(F_0'\cup F_1')|\ \text{are even}.
\end{align}
The following claim demonstrates that both $F_0'$ and $F_1'$ remain connected even after the deletion of any small subset of vertices.

\begin{claim}\label{Connected}
For any $\Tilde{\mathcal{H}}\subseteq \mathcal{H}$ and any $X'\subseteq V(F_0'\cup F_1')$ of size at most $10\alpha^{1/4}n$, we have that $F_0'':=G[V(F_0')\setminus X']$ and $F_1'':=G[V(F_1')\setminus X']$ are connected.
\end{claim}
\begin{proof}[Proof of Claim \ref{Connected}]
   We first consider that $|F_1|\ge \frac{3n}{2r}$. 
   Clearly, $F_0'$ is a clique, and thus $F_0''$ is connected. 
For two vertices $x\in V(F_0)$ and $y\in V(F_1)$, we have
$$
2\Big(1-\frac{1}{r}\Big)n-2\leq d(x)+d(y)\leq 2(n-|B'|)+|F_0|-1+d(y,F_1).
$$
Thus, 
\begin{align}\label{eq:degree-y}    
d(y,F_1'')\geq d(y,F_1')-|X'|\geq d(y,F_1)-|X'|\geq |F_1|-1-|B\setminus B'|-|X'|\geq |F_1|-20\alpha^{1/4}n. 
\end{align}
    For any $y'\in V(F_1')\setminus V(F_1)$, we have
    $$ 
    d(y',F_1'')\ge d(y',F_1')-|X'|\geq d(y',B)-(|F_0|+|B\setminus B'|+|X'|)\ge \frac{n}{r}-1-\big(\frac{n}{2r}+20\alpha^{1/4}\big)\geq \frac{n}{3r}.
    $$
    Therefore, for any two distinct vertices $u,v\in V(F_1'')$, we obtain $N_{F_1''}(u)\cap N_{F_1''}(v)\neq \emptyset$ unless $u,v\in V(F_1'')\setminus V(F_1)$. Consequently, $F_1''$ is connected. 

    Now, we assume  $|F_1|<\frac{3n}{2r}$. For any $y'\in V(F_1'')\setminus V(F_1)$, by the construction of $F_1'$ we have $d(y',F_1'')\ge \frac{2}{3}|F_1|-|X'|\geq \frac{1}{2}|F_1''|+1$. Together with \eqref{eq:degree-y} and the classical Dirac's Theorem, we conclude that 
    $F_1''$ is connected.
    
   We now show that  $F_0''$ is connected. 
   Similar to \eqref{eq:degree-y} we obtain   $d(y,F_0'')\ge |F_0|-20\alpha^{1/4}\geq \frac{4}{5}|F_0''|$ for any $y \in V(F_0\cap F_0'')$. Let $y'\in V(F_0'')\setminus V(F_0)$. If there is a vertex $x\in V(F_0)\setminus N_G(y')$, then by Ore's condition we have
   $$
   2\Big(1-\frac{1}{r}\Big)n-2\leq d(x)+d(y')\leq 2(n-|B'|)+|F_0|-1+\frac{2}{3}|F_1|+d(y',F_0).
   $$
   Thus,
   $$
   d(y',F_0'')\geq d(y',F_0)-|X'|\geq \frac{1}{3}|F_1|-20\alpha^{1/4}n\geq \frac{n}{4r}>\frac{1}{5}|F_0''|,
$$
the last two  inequalities hold by $|F_1|\ge \frac{n}{r}-\alpha^{1/5} n$. Therefore, for any two distinct vertices $u,v\in V(F_0'')$, we obtain $N_{F_0''}(u)\cap N_{F_0''}(v)\neq \emptyset$ unless $u,v\in V(F_0'')\setminus V(F_0)$. It follows that $F_0''$ is connected. 
\end{proof}
To allow for minor modifications, we introduce a claim that modifies a given base into a desired form. 
\begin{claim}\label{Adjust}
{For any subset $\Tilde{\mathcal{H}}\subseteq \mathcal{H}$, let  $F_0'$ and $F_1'$ be given by \eqref{eq:F_0}-\eqref{eq:F_1}. Let $K$ be a clique satisfying \ref{C03} and \ref{C04} with $\xi=\beta/4$, $C_0=K$, $D_0=A_i$ for some $i\in [r-2]$,
and $D_1=B$. 
Suppose that $V_{ex}^S(\beta/2) \cap V(K)=\emptyset$ and $d(K,F_j')\ge \frac{n}{3r}$ for some $j\in \mathbb{Z}_2$. If $|F_0'|\ge 2$ or $d(x)\le n-2\beta'n$ for every $x\in V(K)$, then there exists a $(2,\beta/5)$-base $H'$ such that $V(K)\subseteq V(H')$, $|B\cap V(H')|=4$ and $|V(F_\ell')\setminus V(H')|$ is even for every $\ell\in \mathbb{Z}_2$. }
\end{claim}
\begin{proof}[Proof of Claim \ref{Adjust}]
We first consider $|F_0'|\geq 2$. Since $d(K,F_j')\ge \frac{n}{3r}$ for some $j\in \mathbb{Z}_2$, there exists a vertex $w\in V(F_j')\subseteq V_{e}(2\beta')$ such that $G[V(K)\cup \{w\}]$ forms a clique $K_{k+1}$. Let such a $K_{k+1}$ be $C_0$. 
By Ore's condition, for each $x\in V(F_1')$ one has 
$$d(x,F_1'-w)\ge |F_1'|-|B\setminus B'|-2\ge \frac{1}{2}|F_1'|+1.$$
Thus, Mantel's Theorem implies that $F_1'-w$ contains a $K_3$. Similarly, one may  obtain that $F_0$ contains a $K_3$ if $|F_0'|\geq \frac{n}{7r}$, and $F_0'$ is a clique if $2\le |F_0'|< \frac{n}{7r}$. 
Consequently, there exists a  $K_3$ in $F_{j+\ell}'-V(K)\cup \{w\}$, where $\ell\equiv |F_j'|\pmod{2}$. Define $C_1$ to be such a $K_3$. It is easy to check that $H':=C_0\cup C_1$ satisfy \ref{C03}-\ref{C04} with $D_0:=A_i$, $D_1:=B$, and $\xi:=\beta/5$. Together with \eqref{even-order}, $H'$ is a $(2,\beta/5)$-base such that $|B\cap V(H')|=4$ and $|V(F_i')\setminus V(H')|$ is even for every $i\in \mathbb{Z}_2$.   

Now, we consider  $d(x)\le n-2\beta'n$ for every $x\in V(K)$ and $|F_0'|=1$. Let $w$ be the unique vertex in $F_0'$. Then $d(w)\le (1-\frac{2}{r})n+8\alpha^{1/4}n$. Since $\beta' \gg \alpha$, Ore's condition implies that $G[V(K)\cup \{w\}]$ is a clique, say $C_0$.  Recall that $F_1'$ contains a $K_3$, say $C_1$. 
Thus, $H':=C_0\cup C_1$ is the desired $(2,\beta/5)$-base $H'$.
\end{proof}

Based on the construction of bases in Lemma \ref{Exceptional} and Lemma~\ref{Vne}, there are three possible  types for each base  $H\in \mathcal{H}$: 
${\rm (1)}\ V(H)\cap V_{ne}(2\beta')\subseteq V_{ne}(2\beta')\setminus V_{ex}(\beta/2)$ (in Lemma~\ref{Vne}); ${\rm (2)}\ V(H)\cap V_{ne}(2\beta')\subseteq V_{ex}^L(\beta/2)$ (in Claim \ref{claim-matching}); ${\rm (3)}\ V(H)\cap V_{ex}^S(\beta/2)\ne \emptyset$ (in Claims \ref{r-scover} and \ref{com-2}). We then modify $\mathcal{T}$ into a $K_r$-tiling $\mathcal{T}'$ based on the types of bases in $\mathcal{H}$. 

\medskip
\textbf{Case 1.} There exists a base $H\in \mathcal{H}$ satisfying $V(H)\cap V_{ne}(2\beta')\subseteq V_{ne}(2\beta')\setminus V_{ex}(\beta/2).$
\medskip

By the proof of Lemma \ref{Vne}, one has 
\begin{itemize}
    \item $V(H)\cap V_{ne}(2\beta')$ contains exactly one vertex, say $v$;
    \item if $v\notin B$, then $v\in A_i$ for some $i\in [r-2]$ and $H$ is a $(2,\beta/4)$-base. More precisely, $H$ has two components $C_0$ and $C_1$ satisfying \ref{C03}-\ref{C04}, where $C_0$ is an edge $vw$ in $G[A_i]$ with $d(w)\leq n-2\beta'n$ and $C_1$ is a $K_3$ copy in $G[B\cap V_e(2\beta')]$;
    \item if $v\in B$, then $H$ is an edge  $vw$ in $G[B]$, which forms a $\beta/4$-base.
\end{itemize} 
Suppose that $F_0'$ and $F_1'$ are defined as in \eqref{eq:F_0} and \eqref{eq:F_1} by letting $\Tilde{\mathcal{H}}:=\{H\}$. We proceed by considering whether $v$ belongs to $B$ or not. 

We first assume $v\notin B$. 
Recall that $H$ is extended to two vertex-disjoint copies of $K_r$, denoted $K^1$ and $K^2$, in $\mathcal{T}$ by Lemma~\ref{extend}.  
By \eqref{eq:uv-co} one has $d(\{v,w\},F_j')\ge \frac{n}{3r}$ for some $j\in \mathbb{Z}_2$. Claim \ref{Adjust} implies that there exists a $(2,\beta/4)$-base $H'$ such that $\{v,w\}\subseteq V(H')$, $|V(H')\cap B|=4$ and $|V(F_\ell')\setminus V(H')|$ is even for every $\ell\in \mathbb{Z}_2.$ Applying Lemma~\ref{extend} to $H'$ with $W:=V(\mathcal{T})\setminus V(K^1\cup K^2)$, one may  extend it to two vertex-disjoint copies of $K_r$ that avoids $W$, denoted  $\Tilde{K}^1$ and $\Tilde{K}^2$. Let $\mathcal{T}':=\mathcal{T}\setminus \{K^1,K^2\}\cup \{\Tilde{K}^1,\Tilde{K}^2\}$. Clearly,  $\mathcal{T}'$ is a $K_r$-tiling covering $V_{ne}(2\beta')$.  
For each $\ell\in \mathbb{Z}_2$, Claim~\ref{Connected} ensures that $G[V(F_\ell')\setminus V(\Tilde{K}^1\cup \Tilde{K}^2)]$ is a connected graph with even order. 
Combined with Proposition~\ref{PMatching}, we know  that  $G[B\setminus V(\mathcal{T}')]$ contains a perfect matching. 

Next, assume that $v\in B$. 
Then, $d(v,A_i)\le |A_i|-2\beta'n$ for some $i\in [r-s]$.  
Recall that $H$ is extended to a copy of $K_r$, say $K$, in $\mathcal{T}$. 
Suppose that $d(v,F_{j})=0$ for some $j\in \mathbb{Z}_2$ and let $x$ be a vertex in $F_{j}$. Then\allowdisplaybreaks
\begin{align*}
2\Big(1-\frac{1}{r}\Big)n-2&\leq d(v)+d(x)\leq 2(n-|B'|)-|A_i|+d(v,A_i)+d(v,F_{j+1})+d(x,F_{j})\\
&\leq 2n-|B'|-2-2\beta'n< 2n-|B|-2=2\Big(1-\frac{1}{r}\Big)n-2,
\end{align*}
a contradiction. 
Hence, $d(v,F_j)>0$ for each $j\in \mathbb{Z}_2$.  
Thus, there exists a $\beta/4$-base $H'$ such that $|V(H')\cap B|=2$ and $|V(F_i')\setminus V(H')|$ is even for every $i\in \mathbb{Z}_2$. By Lemma~\ref{extend}, Claim \ref{Connected} and Proposition~\ref{PMatching}, there exists a $K_r$-tiling $\mathcal{T}'$ such that $G[B\setminus V(\mathcal{T}')]$ admits a perfect matching.


\medskip
\textbf{Case 2.} There exists a base $H\in \mathcal{H}$ satisfying  $V(H)\cap V_{ne}(2\beta')\subseteq V_{ex}^L(\beta/2).$
\medskip

By the construction of Claim \ref{claim-matching} in Lemma \ref{Exceptional}, we know that $H$ is an edge $uv$ with $u\in A_i\cap V_{ex}^L(\beta/2,j)$ and $v\in A_j\cap V_e(2\beta')$ for some distinct $i,j\in [s]$, which forms a $\beta/4$-base. 
Suppose that $F_0'$ and $F_1'$ are defined as in \eqref{eq:F_0}-\eqref{eq:F_1} by letting $\Tilde{\mathcal{H}}:=\{H\}$. 
It is routine to check that there exists a vertex $v'\in (A_i\cap V_e(2\beta'))\setminus (V(\mathcal{T})\cup V_b(2\beta'))$ such that $G[\{u,v,v'\}]$ forms a $K_3$. Note that $d(u),d(v),d(v')\le n-2\beta'n$ and $N(\{u,v,v'\},F_j)\ge \frac{n}{3r}$ for some $j\in \mathbb{Z}_2$. 
Applying Claim~\ref{Adjust}, Lemma~\ref{extend}, Claim \ref{Connected} and Proposition~\ref{PMatching} again yield a $K_r$-tiling $\mathcal{T}'$ such that $G[B\setminus V(\mathcal{T}')]$ admits a perfect matching.

\medskip
\textbf{Case 3.} There exists a base $H\in \mathcal{H}$ satisfying   $V(H)\cap V_{ex}^S(\beta/2)\ne \emptyset.$
\medskip

In this case, one may assume that no base in  $\mathcal{H}$ satisfies Case 1 or Case 2. Thus, $V_{ne}(2\beta')=V_{ex}(\beta/2)$ and $|V_{ex}^S(\beta/2)|\geq 1$. Denote $s_i:=|V_{ex}^S(\beta/2,i)|$ for each $i\in [r-2]$. By Claims \ref{r-scover} and \ref{com-2} in  Lemma \ref{Exceptional}, there exist two families $\mathcal{F}_i$ and $\mathcal{F}_i'$ of subgraphs of $G$ for each $i\in [r-2]$ with $s_i>0$, where 
\begin{itemize}
    \item $\mathcal{F}_i$ consists of $\left\lceil s_i/3\right\rceil$ vertex-disjoint $K_3$ copies such that $V_{ex}^S(\beta/2,i) \subseteq V(\mathcal{F}_i)\subseteq \left(B\cap V_{e}(2\beta')\right)\cup V_{ex}^S(\beta/2,i)$,
    \item $\mathcal{F}_i'$ consists of $f(s_i,2)$ vertex-disjoint subgraphs each of which has exactly two vertices in $A_i$, and for any pair  $F\in \mathcal{F}_i$ and 
        $F'\in \mathcal{F}_i'$, their union $F\cup F'$ forms a $(2,\beta/4)$-base in $G$. 
\end{itemize}
Moreover, each graph $F'\in \mathcal{F}_i'$ is one of the following three types:
\begin{itemize}
    \item $F'$ is a triangle $xyz$, where $x\in V_b(2\beta,i)\cap V_{ex}^L(\beta/2,j)$, $y\in V_{e}(2\beta')\cap A_j$, $z\in \big(V_{e}(2\beta')\setminus V_b(2\beta,i)\big)\cap N(\{x,y\},  A_i)$;
    \item $F'$ is an edge $xz$, where $x\in V_b(2\beta,i)\setminus V_{ex}^L(\beta/2)$ and $z\in \big(V_{e}(2\beta')\setminus V_b(2\beta,i)\big)\cap N(x,  A_i)$;
    \item $F'$ is an edge in a matching of $G[A_i\setminus (V_b(2\beta,i)\cup V(M_i))]$. 
\end{itemize}
Let $\mathcal{F}_i=\mathcal{F}_i'=\emptyset$ for each  $i\in [r-2]$ with $s_i=0$. 
Define $F_0'$ and $F_1'$ as in \eqref{eq:F_0} and \eqref{eq:F_1} with $\Tilde{\mathcal{H}} := \mathcal{H}$. 
Clearly, $V(F_0'\cup F_1')\subseteq V_e(2\beta')$. 

Assume that there is an edge $uv\in E(G[V(F_0'),V(F_1')])$. By the construction of $\mathcal{F}_i$ in Claim~\ref{r-scover}, one may avoid $u,v$ in $\bigcup_{i\in [r-2]}\mathcal{F}_i$. Based on Lemma~\ref{extend} with $\{u,v\}\subseteq W$, there exists a $K_r$-tiling $\mathcal{T}'$ that covers all bases in $\mathcal{H}$ and  avoids $\{u,v\}$. By Claim \ref{Connected}, the graph  $F_j'':=G[V(F_j')\setminus V(\mathcal{T}')]$ is connected for each $j\in \mathbb{Z}_2$. As $uv\in E(G[V(F_0''),V(F_1'')])$, the graph  $G[V(F_0''\cup F_1'')]$ is connected. Hence, Proposition~\ref{PMatching} guarantees the existence of a perfect matching in $G[V(F_0''\cup F_1'')]$. Thus, in what follows, it suffices to consider 
\begin{align*}
E(G[V(F_0'),V(F_1')])=\emptyset.  
\end{align*}

 We first assume that there exists an integer $k\in[r-2]$ such that $s_k\ge 3$. Thus, $|\mathcal{F}_k'|\ge |\mathcal{F}_k|+1$. Choose $Q\in \mathcal{F}_k'$. Then $|V(Q)\cap A_k|=2$. Observe that $\mathcal{T}$ can be constructed with  $V(Q)\cap V(\mathcal{T})=\emptyset$. 
Assume that $|F_0|=1$, and let $w$ be the unique vertex in $F_0$. Clearly, $w\in S$.  
 Note that there exists a $(2, \beta/4)$-base $H:=H_1\cup H_2\in \mathcal{H}$ with $|V(H_1)\cap V_{ex}^S(\beta/2,k)|=3$ and $|V(H_2)\cap A_i|=2$. Assume that $H$ is extended to two copies of $K_r$ by Lemma \ref{extend}, denoted $K^1$ and $K^2$, in $\mathcal{T}$ with $H_2\subseteq K^2$. Let $\{x,y,z\}\subseteq V(H)$ be the three vertices in $V(H_1)\cap V_{ex}^S(\beta/2,k)$. In fact, one may assume that $|V(K^2)\cap V(F_1')|=1$. Otherwise, by replacing the vertex in $V(K^2)\cap V(F_0')$ with a vertex in $N(H_2,F_1)\cap V_e(2\beta')$, we get a new tiling $\mathcal{T}'$ such that $G[B\setminus V(\mathcal{T}')]$ contains a perfect matching. 
 Since $z\in V_{ex}^S(\beta/2,i)$, Claim \ref{degree-V-S} implies that there exist $z_1,z_2\in V(F_1')$ such that $G[\{z,z_1,z_2\}]$ is a $K_3$. 
 Thus, there exist two vertex-disjoint $(2,\beta/4)$-bases $H':=G[\{w,x,y\}]\cup H_2$ and $H'':=G[\{z,z_1,z_2\}]\cup Q$ in $G$. By applying Lemma~\ref{extend}, Claim~\ref{Connected} and Proposition~\ref{PMatching}, we obtain a $K_r$-tiling $\mathcal{T}'$ covering $V_{ex}(\beta/2)$ such that $G[B\setminus V(\mathcal{T}')]$ admits a perfect matching. For the case $|F_0|\geq 2$, reapplying these results together with Claim~\ref{Adjust} yields a $K_r$-tiling $\mathcal{T}'$ with the same property.

 Next, we assume that there exists an integer $k\in[r-2]$ such that $|V_{ex}^S(\beta/2,k)|=2$. Let $\{u,v\}=V_{ex}^S(\beta/2,k)$. It follows from \eqref{common-T} that $N(\{u,v\},F_i')\ne \emptyset$ for each $i\in \mathbb{Z}_2$. Hence, there exists a vertex $w\in V(F_j')$ for some $j\in \mathbb{Z}_2$ satisfying the following:
 \begin{itemize}
     \item $H:=G[\{u,v,w\}]\cup F'$ forms a $(2,\beta/4)$-base, where $F'\in \mathcal{F}_k'$.
     \item Assume $K^1\cup K^2$ are two vertex-disjoint $K_r$ copies covering $H$ and avoiding $V(\mathcal{T})$ (via Lemma~\ref{extend}). Define $\mathcal{T}'$ as the $K_r$-tiling obtained from $\mathcal{T}\cup \{K^1, K^2\}$ by deleting the $K_r$ copies that covered the base containing $\{u,v\}$ in $\mathcal{T}$. Then $|V(F_\ell')\setminus V(\mathcal{T})|$ is even for each $\ell\in \mathbb{Z}_2$. 
 \end{itemize}  
 Together with  Claim \ref{Connected}, and Proposition \ref{PMatching}, there exists a perfect matching in $G[B\setminus V(\mathcal{T}')]$. 
Thus, in what follows, we consider that  
\begin{align}\label{VexS}
    |V_{ex}^S(\beta/2,k)|\le 1 \ \text{for each}\ k\in[r-2].
\end{align}

Suppose that there exists a subgraph $L\in \bigcup_{k\in [r-2]}\mathcal{F}_k'$ such that $N(L,F_i')\neq \emptyset$ for all $i\in \mathbb{Z}_2$. 
{By an argument similar to the case that $|V_{ex}^S(\beta/2,k)|=2$ for some $k\in [s]$}, we conclude that there exists a $K_r$-tiling $\mathcal{T}'$ covering $V_{ex}(\beta/2)$ for which $G[B\setminus V(\mathcal{T}')]$ contains a perfect matching. 
In the following,  assume that 
\begin{align}\label{LF_k}
\text{for any}\ L\in \bigcup_{k\in[r-2]}\mathcal{F}_k', \ \text{we have}\ 
    N(L,F_i')=\emptyset \ \text{for some} \ i\in \mathbb{Z}_2.
\end{align}

Fix a $k\in [s]$ with $s_k=1$,  let $u_kv_k\in E(G[A_k])\setminus E(G[A_k\cap V_b(2\beta,k)])$. 
Since $V_{ne}(2\beta')=V_{ex}(\beta/2)$, one has $A_k\setminus V_b(2\beta,k)\subseteq V_e(2\beta')$. Without loss of generality, assume that $u_k\in V_e(2\beta')\setminus V_b(2\beta)$. Suppose that $v_k\notin V_b(2\beta,k)$. Then  $N(\{u_k,v_k\},B)\geq \frac{2n}{r}-4\beta n$ and $u_kv_k$ forms a $\beta/4$-base in $G$. It follows from \eqref{LF_k} that   $|F_0|\leq 5\beta n$. However, \eqref{complete} implies that $G[\{u_k,v_k\},V(F_0)]$ is a complete bipartite  graph, which also contradicts \eqref{LF_k}. Thus, $v_k\in V_b(2\beta,k)$. Similarly, we can show that $v_k\notin V_{ex}^L(\beta/2)$, that is,  $v_k\in V_b(2\beta,k)\setminus V_{ex}^L(\beta/2)$. Therefore, each graph in $\mathcal{F}_k'$ is exactly an edge. Moreover, $$
E(G[A_k])\setminus E(G[A_k\cap V_b(2\beta,k)])=E(G[V_e(2\beta')\setminus V_b(2\beta,k),V_b(2\beta,k)\setminus V_{ex}^L(\beta/2)]).
$$

Next, we show that $E(G[A_k])\setminus E(G[A_k\cap V_b(2\beta)])$ contains no edge that is vertex-disjoint from $u_kv_k$. Assume for contradiction that it contains two independent edges, say $u_kv_k$ and $u'v'$. 
Assume that $u'\in V_e(2\beta')\setminus V_b(2\beta,k)$ and  $v'\in V_b(2\beta,k)\setminus V_{ex}^L(\beta/2)$.  
If $|F_0|\geq 2$, then applying Claim~\ref{Adjust} with $K=u'v'$, one may obtain a $(2,\beta/5)$-base $H'$ such that $\{u,v\}\subseteq V(H')$, $|B\cap V(H')|=4$ and $|V(F_i)\setminus V(H')|$ is even for every $i\in \mathbb{Z}_2$. 
Thus,  Lemma \ref{extend} and Proposition~\ref{PMatching} guarantee the existence of a new $K_r$-tiling $\mathcal{T}'$ such that $G[B\setminus V(\mathcal{T}')]$ has a  perfect matching. 
{Now, we assume $|F_0|=1$ and let $w$ be the unique vertex in $F_0$. Let $M^*:=\big(\bigcup_{j\in[r-2]}\mathcal{F}_j'\big)\cup \{u'v'\}$. Clearly,  $M^*$ is a matching of size $|V_{ex}^S(\beta/2)|+1$. 
Since $E(G[V(F_0'),V(F_1')])=\emptyset$, every vertex $u\in V(F_1')$ satisfies $d(u)\le n-|F_0'|-1$. 
By Ore's condition, one has 
$$d(w)\ge n-|F_1'|-|V_{ex}^S(\beta/2)|-1.$$
Therefore, there exists  an edge $xy\in E(M^*)$ such that $\{x,y\}\subseteq N(w)$.  
Notice that $d(\{x,y\},B)\ge \frac{n}{r}-2\beta n$ and $|V(F_0')\cup V_{ex}^S(\beta/2)|<2\beta n$. Hence, $N(\{x,y\},F_1')>0$, which contradicts  \eqref{LF_k}.} Thus, we have that $G[A_k]$ is a star with center $v_k$ together with  some isolated vertices. Let $K^k$ be the $K_r$ copy in $\mathcal{T}$ that covers $v_k$.


Let $w_k$ denote the unique vertex in $V_{ex}^S(\beta/2,k)$. 
Suppose that there exists a vertex $x_k\in A_k\setminus V_b(2\beta,k)$ such that $w_kx_k\in E(G)$. Since $d(x_k)\geq (1-1/r)n-1$, all but at most two vertices in $B$ are adjacent to $x_k$. By Claim \ref{degree-V-S}, \eqref{complete}   and the fact that $w_k\in V_{ex}^S(\beta/2,k)$, it is easy to check that $N(\{w_k,x_k\},F_i')\neq \emptyset$ for each $i\in \mathbb{Z}_2$.  Assume that $V(K^k)\cap V(F_j')\neq \emptyset$ for some $j\in \mathbb{Z}_2$. Choose $r_k\in N(\{w_k,x_k\},F_{j+1}')$ such that $G[\{w_k,x_k,r_k\}]$ is a $\beta/4$-base. By replacing the $(2,\beta/4)$-base covering $w_k$ with $G[\{w_k,x_k,r_k\}]$ in $\mathcal{H}$, together with Lemma~\ref{extend}, Claim \ref{Connected}, and Proposition \ref{PMatching}, we obtain a $K_r$-tiling $\mathcal{T}'$ covering $V_{ex}(\beta/2)$ such that $G[B\setminus V(\mathcal{T}')]$ contains a perfect matching. In what follows, we assume that 
\begin{align*}
w_kx\notin E(G) \ \text{for any} \ x\in A_k\setminus V_b(2\beta,k).
\end{align*}
Thus, $d(x)\leq (1- \frac{1}{r})n$ for every $x\in A_k\setminus V_{b}(2\beta)$. Together with $w_k\in S$ and 
Ore's condition, one has 
$d(x)=(1- \frac{1}{r})n\ \text{for every}\ x\in A_k\setminus V_{b}(2\beta),\ \text{and}\ d(w_k)=(1-\frac{1}{r})n-2.$ Thus, 
$$G[A_k] \ \text{is a star with center $v_k$ for all $k\in [r-2]$ with $s_k=1$}.$$ 
A similar argument shows that for all $k \in [r-2]$ with $s_k = 0$, we have $G[A_k] = \emptyset$, whence Ore's condition implies that $G[A_k, V_{ex}^S(\beta/2)]$ is complete.

The above argument implies that for every $x\in A_i\setminus \{v_i\}$ and each $i\in [r-2]$, one has
\begin{align}\label{eq:neighbor}
    N(x) = \left\{
\begin{array}{ll}
\big( V(G) \setminus (A_i \cup \{w_i\}) \big) \cup \{v_i\} & \text{if } s_i=1, \\
V(G) \setminus A_i & \text{if } s_i=0.
\end{array}
\right.
\end{align}
Moreover, $|V_b(2\beta)|=|V_{ex}^S(\beta/2)|=:q$, and   by  \eqref{LF_k} one has $N(v_i,F_j')=\emptyset$ for some $j\in \mathbb{Z}_2$ and all $v_i\in V_b(2\beta)$. 
Let $$ Y_0:=\{y\in V_b(2\beta):N(y,F_0')=\emptyset\}\ \text{and} \ Y_1:=\{y\in V_b(2\beta):N(y,F_1')=\emptyset\}.$$
Choose $z_0\in V(F_0')$ and $z_1\in V(F_1')$. By Ore's condition, one has
$$
2\Big(1-\frac{1}{r}\Big)n-2\leq d(z_0)+d(z_1)\leq 2\Big(1-\frac{2}{r}\Big)n-(|Y_0|+|Y_1|)+2q+|F_0'|+|F_1'|-2=2\Big(1-\frac{1}{r}\Big)n-2. 
$$
Thus, $d(z_0,V_{ex}^S(\beta/2))=d(z_1,V_{ex}^S(\beta/2))=q$. Hence  $G[V_{ex}^S(\beta/2),V(F_0'\cup F_1')]$ is a complete bipartite graph. 

Define $\hat{G}:=G[V_{ex}^S(\beta/2),V_b(2\beta)]$. Recall that $d(w_i)=(1-\frac{1}{r})n-2$ for each $w_i\in V_{ex}^S(\beta/2)$. Together with \eqref{eq:neighbor}, one has 
$d(w_i,V_b(2\beta))= q-2$ for each $w_i\in V_{ex}^S(\beta/2)$, this implies  $e(\hat{G})=q(q-2).$ Thus, there exists a vertex $v_t\in V_b(2\beta)$ such that  $d(v_t,V_{ex}^S(\beta/2))\leq q-2$.  
Without loss of generality, assume that  
$v_t\in Y_0$.  Choose $y\in V(F_0')$. By Ore's condition one has
\begin{align*}
2\Big(1-\frac{1}{r}\Big)n-2\leq d(v_t)+d(y)\leq &\Big(\Big(1-\frac{2}{r}\Big)n-|V_b(2\beta)|+d(v_t,V_b(2\beta))+d(v_t,V_{ex}^S(\beta/2))+|F_1'|\Big)\\
&\ +\Big(\Big(1-\frac{2}{r}\Big)n-|Y_0|+|F_0'|-1+|V_{ex}^S(\beta/2)|\Big)\\
\leq& 2\Big(1-\frac{1}{r}\Big)n
-3+d(v_t,V_b(2\beta))-|Y_2|. 
\end{align*}
 It follows that $d(v_t,V_b(2\beta))\geq |Y_0|+1$. That is, there exists a vertex $v_{\ell}\in Y_1\cap A_{\ell}$ for some $\ell\in [r-2]$ such that $v_tv_{\ell}\in E(G)$. 

Obviously, there exists a copy of $K_r$, say $K$, containing $v_t,v_{\ell}$ such that $|V(K)\cap A_j|=1$ for each $j\in [r-2]\setminus \{t,\ell\}$ and $|V(K)\cap A_j|=2$ for each $j\in \{t,\ell\}$. Recall that $G[V_{ex}^S(\beta/2),V(F_0'\cup F_1')]$ is a complete bipartite graph. Hence, there are two vertices $z,z'\in V(F_i)$ for some $i\in \mathbb{Z}_2$ such that $G[\{w_t,w_{\ell},z,z'\}]$ forms a clique $K_4$. Moreover, there is a $K_r$ copy, say $K'$, containing $\{w_t,w_{\ell},z,z'\}$ such that $|V(K')\cap A_j|=1$ for each $j\in [r-2]\setminus \{t,\ell\}$. Let $\mathcal{T}':=(\mathcal{T}\setminus \{K_r\in \mathcal{T}: V(K_r)\cap \{v_t,v_{\ell},w_t,w_{\ell}\}\ne \emptyset\})\cup\{K,K'\}$. Since $E(G[F_0',F_1'])=\emptyset$, every $K_r$ in $\mathcal{T}$ covering some $w_i \in V_{ex}^S(\beta/2)$ must have two vertices within either $F_0'$ or $F_1'$. 
It follows that  $G[B\setminus V(\mathcal{T}')]$ consists of two components with even order.  Therefore,  
Proposition \ref{PMatching} implies that there exists a perfect matching in $G[B\setminus V(\mathcal{T}')].$

By the above argument, we obtain that when $r-s=2$, either $G$ is the extremal graph in \ref{EX2}, or there exists a $K_r$-tiling $\mathcal{T}'$ satisfying \ref{L2}, as desired. 
\end{proof}

\section{Concluding remarks}

\subsection{Algorithmic aspects of the proof}

We point out those parts of the proof where translation into an algorithm is nontrivial. 

\medskip
\textbf{Non-extremal case:} To construct the desired absorbing set, we first note that for almost all $r$-set in $G$, $\Omega(n^{r^2})$ absorbers can be found in $O(n^{r^2})$ time.  In our proof, the existence of the absorbing set is established by probabilistic arguments, which naturally lead to a randomized construction. However, we can directly apply the result of Garbe and Mycroft \cite[Proposition 4.7]{GM2018}, which provides an $O(n^{4r^2})$-time deterministic algorithm via the conditional expectation method.

For the almost cover, the algorithm of Alon et al. \cite{ADLR1994} yields a regular partition that satisfies the Regularity lemma in $O(n^{2.376})$ time. As the order of the reduced graph is bounded, Claims \ref{a-almost} and \ref{reduced-blow} take constant time. The blow-up process via Lemmas \ref{regular-subgraph} and \ref{reduce-original} takes $O(n^{r+1})$ time.  Thus, for the non-extremal case, the overall complexity of constructing a $K_r$-tiling is $n^{O_r(1)}$.


\medskip
\textbf{Extremal case:} For the extremal case, an $(r,s)$-partition as in Claim \ref{rs-partition} can be found in $O(n^{2.376})$ time using Algorithm 2 of Gan, Han, and Hu \cite{GHH2024}. The remaining steps are completed by a greedy algorithm, yielding an overall time complexity of $n^{O_r(1)}$.

When $r$ is a constant, Theorem \ref{FVersion} can be achieved in polynomial time. Therefore, Conjecture~\ref{cKKMS2008} holds for $k \ge c n$, where $n:=|G|$ is sufficiently large and $c > 0$ is any fixed constant.

\subsection{Further research}

In this paper, we prove that the Chen-Lih-Wu conjecture holds when $n$ is sufficiently large and $k\ge cn$ for any positive constant $c$. In fact, we establish a stronger result under the Ore-type condition. Our proof necessitates the condition $k = \Omega(n)$ due to the following requirements: in the non-extremal case, $n \ge r^{2^{c_0 r}}$ for Lemma \ref{almost-cover} and Lemma \ref{absorbing}; in the extremal case, $n \ge r^{2^{c_1 r}}$ due to the parameter relation $\gamma_i \le \gamma_{i+1}^5$ from \eqref{eq:p1}-\eqref{eq:p3}. Given that $n \le rk$, these bounds are satisfied for sufficiently large $n$ if $k = \Omega(n)$, thereby establishing Conjectures \ref{CLW} and \ref{cKK2008} under this condition.  A natural and interesting direction for future work is to consider the case when $k = o(n)$. 



Our paper focuses on the problem of equitable coloring under Ore-type conditions. Another natural direction for future research is to extend this work to degree sequence conditions. Inspired by the degree sequence version of the Hajnal-Szemer\'{e}di conjecture proposed by Balogh, Kostochka, and Treglown \cite{BKT2013}, we formulate a corresponding conjecture for equitable coloring, thereby introducing a new perspective rooted in degree sequences.


\begin{conj}
    Let $n,k\in \mathbb{N}$ such that $k\le n$. Suppose that $G$ is an $n$-vertex graph with degree sequence $d_1\ge  \cdots \ge d_n$ such that 
    \begin{itemize}
        \item $d_i< 2k-i$ for all $i< k$;
        \item $d_{k+1}<k$.
    \end{itemize}
Then $G$ contains an equitable $k$-coloring.
\end{conj}

\section*{Acknowledgements}
We extend our sincere gratitude to Jie Han for bringing the algorithm problem to our attention. 

\bibliographystyle{abbrv}
\bibliography{equitable-coloring}

\end{document}